\documentclass[a4paper,twoside]{article}  
\usepackage{amssymb,amsmath,amsthm,dsfont,amsfonts,color,latexsym}
\usepackage{hyperref}
\usepackage{abstract}
\usepackage{bbm}
\usepackage{mathrsfs} 
\numberwithin{equation}{section}
\newcommand{\re}{\mathrm{Re}}
\newcommand{\im}{\mathrm{Im}}
\newcommand{\res}{\mathrm{Res}}
\definecolor{blue}{rgb}{0.00,0.00,1.00}
\definecolor{red}{rgb}{1.00,0.00,0.00}

\allowdisplaybreaks[4]
\renewcommand{\baselinestretch}{1.2}
\def\bq{\begin{equation}}
\def\eq{\end{equation}}
\def\ball{\begin{aligned}}
\def\eall{\end{aligned}}
\def\ba{\begin{array}{ccc}}
\def\bal{\begin{array}{lll}}
\def\ea{\end{array}}
\def\bsp{\begin{split}}
\def\esp{\end{split}}

\def\({\left(}\def\){\right)}
\def\[{\left[}\def\]{\right]}

\def \E   {\mathrm{E}}
\def \D   {\mathrm{D}}
\def \M   {\mathbb{M}}

\def \L   {\mathbf{L}}

\def \R   {\mathbb{R}}

\def\C    {\mathbb{C}}

\def\i    {\mathrm{i}}

\def\S    {\mathbb{S}}
\def\eps  {\epsilon}
\def\intr {\int_{\R^3}}

\def\intt {\int^t_0}

\def \N    {\mathbb{N}}

\def \Dt   {\frac{\rm d}{{\rm d}t}}
\def \Dtau   {\frac{\rm d}{{\rm d}\tau}}

\def \dt    {\partial_t}

\def \dx    {\partial_x}

\def \dxa   {\partial^{\alpha}_x}

\def \dxap   {\partial^{\alpha'}_x}

\def \ddx   {\Delta_x}
\def \ddv   {\Delta_v}

\def \div  {{\rm div}}
\def \divx  {{\rm div}_x}

\def\tdx   {\nabla_x}
\def\tdv   {\nabla_v}

\def\bq{\begin{equation}}
\def\eq{\end{equation}}
\def\be{\begin{equation}}
\def\ee{\end{equation}}

\def\bag{\begin{align*}}
\def\eag{\end{align*}}
\def\bma#1\ema{{\allowdisplaybreaks\begin{align}#1\end{align}}}
\def\bmas#1\emas{{\allowdisplaybreaks\begin{align*}#1\end{align*}}}
\def\bln#1\eln{{\allowdisplaybreaks\begin{aligned}#1\end{aligned}}}
\def\nnm{\notag}
\def\bgr#1\egr{\allowdisplaybreaks\begin{gather}#1\end{gather}}
\def\bgrs#1\egrs{\allowdisplaybreaks\begin{gather*}#1\end{gather*}}

\theoremstyle{plain}
\newtheorem{lem}{\bf Lemma}[section]
\newtheorem{thm}[lem]{\textbf{Theorem}}

\newtheorem{defn}[lem]{\textbf{Definition}}
\newtheorem{remark}[lem]{\bf Remark}

\usepackage{marginnote}
\usepackage{varwidth}
\begin{document}
\title{Spectrum analysis and optimal time decay rates of a kinetic-fluid-Poisson system}
\author{Junhao Chen$^1$, Hailiang Li$^2$, Mingying Zhong$^{1,3}$
\\\emph
{\small\it $^1$School of  Mathematics, Guangxi University, China.}
\\{\small\it E-mail:\ junhaochenmath@163.com}
\\{$^2$\small\it School of Mathematical Sciences, Capital Normal University, China.}
\\{\small\it E-mail:\ hailiang.li.math@gmail.com}
\\{\small\it $^3$Center for Applied Mathematical of Guangxi (Guangxi University), Guangxi University, China.}
\\{\small\it E-mail:\ zhongmingying@gxu.edu.cn}
}

\pagestyle{myheadings}
\markboth{kinetic-fluid-Poisson system}%
{J.-H. Chen, H.-L. Li, M.-Y. Zhong}
\date{}
\maketitle
\begin{abstract}
In this paper, we consider the Cauchy problem for the Vlasov-Poisson-Fokker-Planck/Navier-Stokes-Poisson (VPFP/NSP) system, which couples the VPFP system with the compressible NSP system through a friction force dependent on the relative velocity and a self-consistent Poisson equation \cite{lxliu0,nsp1,dlzt2}. Motivated by the spectrum analysis for the Vlasov-Poisson-Boltzmann (VPB) system \cite{zhongspvpb}, we introduce a suitable norm to capture the effect of the forcing induced by the Poisson equation and give a detailed spectrum analysis of the linearized system around a global equilibrium. Our results show that the two coupling mechanisms lead to an essentially different spectrum structure of the coupled system from those of the individual VPFP and NSP systems \cite{nsp1,vpfp1}. More precisely, the low-frequency spectrum contains a pair of acoustic branches with the propagation speed $\sqrt{\frac{\gamma+1}{2}}~(\gamma\ge1)$ and two diffusive branches, thereby restoring the usual acoustic wave propagation of classical compressible fluids. Moreover, we establish the global existence of the solution to the nonlinear system and obtain the optimal time decay rate $(1+t)^{-\frac34}$, with a faster rate $(1+t)^{-\frac54}$ for the electric field and relative velocity. The present analysis also provides a useful framework for studying related kinetic-fluid models coupled through friction and self-consistent fields.

\end{abstract}

{\tableofcontents}
\section{Introduction}
Kinetic-fluid models play an important role in the study of two-phase flows. In two-phase flows, the disperse phase (kinetic) is usually described using a statistical approach, while the dense phase (fluid) using a hydrodynamic approach \cite{lxliu0,lxliu3}. These models have extensive specific applications, including the description of diesel engines, the dynamics of sprays, pollution settling processes, rain formation, wastewater treatment and so on \cite{pphy3,pphy4,pphy5,pphy2}.

The kinetic-fluid-Poisson system can be used to model the temporal evolution of a collisionless plasma with hot electrons and cold positive ions, and capture the kinetic effects such as electron phase space holes, with related applications including the study of the nonlinear electrostatic waves, the generation and propagation of ion‑acoustic solitons \cite{dlzt1,dlzt2}. In the present paper, we consider the Vlasov-Poisson-Fokker-Planck/Navier-Stokes-Poisson (VPFP/NSP) system in three dimension:
\be\label{introe1}
\left\{\ball
&\partial_tF+v\cdot\nabla_xF+\nabla_x\Phi\cdot\nabla_vF=\nabla_v\cdot[\nabla_vF+(v-u)F],
\\&\partial_t\varrho+\div_x (\varrho u)=0,
\\&\partial_t(\varrho u)+\div_x (\varrho u\otimes u)-\ddx u+\nabla p+\varrho\nabla_x\Phi=\int_{\mathbb{R}^3}(v-u)Fdv,
\\&\ddx\Phi=\int_{\mathbb{R}^3}Fdv-\varrho .
\eall\right.
\ee
In this system, the distribution function of the dilute charged particles is given by $F(t, x, v)$ with the position $x=(x_1,x_2,x_3)\in\R^3$, velocity $v=(v_1,v_2,v_3)\in\R^3$ and time $t\in\R^+$. The fluid is characterized by its density $\varrho = \varrho(t, x) \ge 0$, velocity $u = u(t, x)$, and pressure $p(\varrho) = \varrho^{\gamma}$ with $\gamma\ge1$. $\Phi(t,x)$ is the electrostatic potential. The kinetic and fluid equations are coupled through two mechanisms, namely the mutual friction force $F_f:=c_f(v-u)$ dependent on the relative velocity with the friction coefficient $c_f = 1$,  and the electrostatic force derived from the electrostatic potential $\Phi(t,x)$ governed by the Poisson equation $\eqref{introe1}_4$.

There have been a lot of works on the kinetic-fluid system where the two phases are coupled through the friction force. For the Vlasov-Fokker-Planck-Navier-Stokes system, the global existence of weak solutions was studied in \cite{lxliu1}, and the asymptotic analysis of the solutions was studied in \cite{lxliu2}. Moreover, the global existence and large time behavior of the strong solution when the friction coefficient $c_f = \varrho$ were studied in \cite{lxliu3}, and the  global inviscid limit was studied in \cite{lxliu8}. The globally nonlinear time-asymptotical stability of the planar rarefaction wave was studied in \cite{lxliu7}. For the Vlasov-Fokker-Planck-Euler system, the global well-posedness of the Cauchy problem and rates of convergence of solutions toward equilibrium were studied in \cite{lxliu4} for the compressible case, and the incompressible case in \cite{lxliu5}. The global existence of weak solutions to the Vlasov-Poisson-Fokker-Planck-Navier-Stokes system with arbitrary large initial data was established in \cite{lxliu6}.

Spectrum analysis is an important method for the study of PDEs. For the VPFP system, the spectrum structure and the optimal time decay rates of the classic solution to the Cauchy problem  were investigated in \cite{vpfp1}, and based on this result, the corresponding Green's function and pointwise behavior was established in \cite{vpfp2}. For the Vlasov-Poisson(Maxwell)-Boltzmann system, the spectrum analysis and the optimal time decay rates of the global solution were studied in \cite{zhongspvpb,zhongspbvpb, zhongspvmb} respectively.

Notice that there is a stationary solution $U_*=(M,1,0)$ of system \eqref{introe1} with the normalized Maxwellian $M(v)$ given by
$$
M:=M(v)=\frac{1}{(2\pi)^{\frac{3}{2}}}e^{-\frac{|v|^2}{2}}, \quad v\in\mathbb{R}^3.
$$
Define the perturbation $(f,\rho,u)$ of $(F,\varrho,u)$ near $U_*$ by
$$
F=M+\sqrt{M}f,\quad \varrho=1+\rho,
$$
the system \eqref{introe1} for $(f,\rho,u)$ can be written as
\be\label{introe2}
\left\{\ball
&\partial_t f+v\cdot \nabla_xf -Lf-v\chi_0\cdot u-v\chi_0\cdot\nabla_x\Phi =R_1,
\\
&\partial_t\rho+\div_xu=-\divx(\rho u),
\\
&\partial_tu+\gamma\nabla_x\rho-\ddx u-(\tau_bf-u)+\nabla_x\Phi=R_2,
\\
&\ddx\Phi=\tau_af-\rho,
\eall\right.
\ee
where
\bma
R_1&\label{n1}=(\nabla_x\Phi+ u)\cdot\(\frac{1}{2}vf-\nabla_vf\),
\\
R_2&\label{n2}=-(u\cdot\nabla_x)u-\frac{\rho}{1+\rho}\ddx u-\[\frac{p'(1+\rho)}{1+\rho}-\gamma\]\tdx\rho-\frac{\rho}{1+\rho}(\tau_bf-u)-\frac{\tau_afu}{1+\rho},
\ema
and
$$
\tau_af=\intr\sqrt{M}f dv,\quad \tau_bf=\intr v\sqrt{M}f dv.
$$
The initial data is given by
\be (f(0,x,v),\rho(0,x),u(0,x))=(f_0(x,v),\rho_0(x),u_0(x)). \label{initial}\ee
The operator $L$ is the Fokker-Planck operator
$$
Lf=\ddv f-\frac{|v|^2}{4}f+\frac{3}{2}f.
$$

We denote $L^2(\mathbb{R}^3_v)$ be a Hlibert space of complex-value functions $f(v)$ on $\mathbb{R}^3$ with the inner product and the norm
$$
(f,g)=\int_{\mathbb{R}^3}f(v)\overline{g(v)}dv,\quad \|f\|=\left(\int_{\mathbb{R}^3}|f(v)|^2dv\right)^{\frac{1}{2}}.
$$

Corresponding to the linearized operator $L$, we define the dissipation norm by
$$
\|f\|_{L^2_\sigma}^2=\|\nabla_vf\|^2+\|vf\|^2.
$$
This norm is stronger than the $L^2$- norm because
$$
\|f\|_{L^2_\sigma}^2\ge2\|\nabla_vf\|\|vf\|\ge-2(\tdv f,vf)=3\|f\|^2.
$$

For convenience, we denote $\{\chi_j,~j=0,1,2,3\}$, with
$$
\chi_0=\sqrt{M},\quad \chi_j=v_j\sqrt{M},~j=1,2,3.
$$
Then, the null space of the operator $L$, denoted by $N_0$, is a subspcae spanned by $\chi_0$. Let $P_d$ be the projection operator from $L^2(\mathbb{R}^3_v)$ to the subspace $N_0$ with
$$
P_df=\tau_af\chi_0=(f,\chi_0)\chi_0,\quad P_r=I-P_d.
$$
Moreover, noticing that
$$
L\chi_0=0,\quad L\chi_j=-\chi_j,~~j=1,2,3,
$$
which means that $\mbox{span}\{\chi_j, j = 0,1,2,3\}$ is the invariant subspace of $L$. We introduce the projection operators $P_m$, $P_0$ in the invariant subspace of $L$ by
$$
P_mf=\tau_bf\cdot v\chi_0=\sum\limits_{j=1}^3 (f,\chi_j)\chi_j,\quad P_0=P_d+P_m,\quad P_1=I-P_0.
$$

Form \cite{yusplfp}, the linearized operator $L$ is non-positive and locally coercive in the sense that
$$
(Lf,f)\le-\|P_rf\|^2,\quad f\in D(L),
$$
where $D(L)$ is the domain of $L$ given by
$$D(L)=L^2_\sigma(\R^3_v):=\{f\in L^2(\mathbb{R}^3_v)\mid \|f\|_{L^2_\sigma}<+\infty\}. $$
Moreover, there exists a constant $\mu_0\in (0,1)$ such that
$$
(Lf,f)\le-\mu_0\|P_rf\|_{L^2_\sigma}^2.
$$
\noindent\textbf{Notations:} For any $\alpha=(\alpha_1,\alpha_2,\alpha_3)\in\N^3$, denote $\dxa=\partial_{x_1}^{\alpha_1}\partial_{x_2}^{\alpha_2}\partial_{x_3}^{\alpha_3}$. Define the Fourier transform of $f=f(t,x,v)$ by
$$\hat{f}(\xi,v)=\mathcal{F} f(\xi,v)=\frac{1}{(2\pi)^{3/2}}\intr f(x,v)e^{-ix\cdot\xi}dx,$$
where and throughout this paper we denote $i=\sqrt{-1}$. Set $\xi=s\omega$ with $\omega=\xi/|\xi|\in\S^2$.

Set the Sobolev space $H^N_x=\{f\in L^2(\R^3_x)\mid\|f\|_{H^N_x}<\infty\}$ equipped with the norm
$$\|f\|^2_{H^N_x}=\sum\limits_{|\alpha|\le N}\|\dxa f\|_{L^2_x}^2.$$
For $q\ge1$, we also define
$$
L^{2,q}=L^2(\R^3_v,L^q(\R^3_x)),\quad \|f\|_{L^{2,q}}=\(\intr\(\intr|f(x,v)|^qdx\)^{2/q}dv\)^{1/2}.
$$
In what follows, we denote by $\|\cdot\|_{L^2_{x,v}}$ and $\|\cdot\|_{L^2_{\xi,v}}$
the norm of the function spaces $L^2(\R^3_x\times \R^3_v)$ and
$L^2(\R^3_\xi\times \R^3_v)$ respectively, and denote by
$\|\cdot\|_{L^q_x}$ and $\|\cdot\|_{L^q_\xi}$ the norms of the function spaces $L^q(\R^3_x)$ and $L^q(\R^3_\xi)$ respectively. For any integer $N\ge 1$, we denote by
$\|\cdot\|_{L^2_v(H^N_x)}$ the norm in the space
$L^2(\R^3_v,H^N(\R^3_x))$.
Moreover, we denote by $\|\cdot\|_{L^2_\sigma(L^q_x)}$ and $\|\cdot\|_{L^2_\sigma(H^N_x)}$ the norms in the spaces $L^2_\sigma(\R^3_v,L^q(\R^3_x))$ and
$L^2_\sigma(\R^3_v,H^N(\R^3_x))$ respectively.


Let $\varsigma\ge0$, and define the Banach space for the function $f(v)$ as
$$
L^\infty_{v,\varsigma}(\R^3_v)=\Big\{f\in L^\infty(\R^3_v)\mid\|f\|_{L^\infty_{v,\varsigma}}=\sup\limits_{v\in\R^3}(1+|v|)^\varsigma|f(v)|<\infty\Big\}.
$$
and the Banach space for the function $f(x,v)$ as
\bmas
L^\infty_{v,\varsigma}(L^q_x)&=L^\infty_{v,\varsigma}(\R^3_v,L^q(\R^3_x)), ~~\|f\|_{L^\infty_{v,\varsigma}(L^q_x)}=\sup\limits_{v\in\R^3}(1+|v|)^\varsigma\|f(\cdot,v)\|_{L^q_x},
\\L^\infty_{v,\varsigma}(H^N_x)&=L^\infty_{v,\varsigma}(\R^3_v,H^N_x(\R^3_x)), ~~\|f\|_{L^\infty_{v,\varsigma}(H^N_x)}=\sup\limits_{v\in\R^3}(1+|v|)^\varsigma\|f(\cdot,v)\|_{H^N_x}.
\emas

First, we have the following result on well-posedness of the system \eqref{introe2}  and \eqref{initial}:
Let $N\ge3$ and $U=(f,\rho,u)$, and define
\bmas
E_N(U)=&\sum\limits_{|\alpha|\le N}\(\|\dxa f\|_{L^2_{x,v}}^2+\gamma\|\dxa\rho\|_{L^2_x}^2+\|\dxa u\|_{L^2_x}^2+\|\dxa\tdx\Phi\|_{L^2_x}^2\),
\\
D_N(U)=&\sum\limits_{|\alpha|\le N}\(\|\dxa P_1f\|_{L^2_\sigma(L^2_x)}^2+\|\dxa\nabla_xu\|_{L^2_x}^2+\|\dxa(\tau_bf-u)\|_{L^2_x}^2\)
\\&+\sum\limits_{|\alpha|\le N-1}\(\|\dxa\tdx P_0f\|_{L^2_v(L^2_x)}^2+\|\dxa\tdx\rho\|_{L^2_x}^2+\|\dxa\tdx\Phi\|_{H^1_x}^2\).
\emas
\begin{thm}\label{ztcz}
There exists a small constant $\delta_0>0$ such that if the initial data $U_0$ satisfies $E_3(U_0)\le\delta_0$, then the Cauchy problem for the VPFP/NSP system \eqref{introe2}  and \eqref{initial} admits a unique global solution $U=(f,\rho,u)$ satisfying
$$
E_3(U(t))+\intt D_3(U(s))ds\le CE_3(U_0).
$$
\end{thm}
Then, we have the large time behavior of the global solution to the system \eqref{introe2}  and \eqref{initial} as follows:
\begin{thm}\label{introth2}
Assume that  there exists a small constant $\delta_0>0$  such that the initial data $U_0=(f_0,\rho_0,u_0) $ satisfies
$$
\|f_0\|_{L^\infty_{v,3}(H^3_x\cap L^1_x)}+\|(\rho_0,u_0)\|_{H^3_x\cap L^1_x}\le\delta_0.
$$
Then, the global solution $U=(f,\rho,u)$ to the VPFP/NSP system \eqref{introe2} and \eqref{initial} satisfies that
\be\label{nldcre1}\left\{\ball
&\|f(t)\|_{L^\infty_{v,3}(L^2_x)}+\|(\rho,u)(t)\|_{L^2_x}\le C(1+t)^{-\frac{3}{4}},
\\&\|\tdv f(t)\|_{L^\infty_{v,2}(L^2_x)}\le C\(1+\frac{1}{\sqrt{t}}\)(1+t)^{-\frac{3}{4}},
\\&\|\tdx\Phi(t)\|_{L^2_x}+ \|(\tau_bf-u)(t)\|_{L^2_x}\le C(1+t)^{-\frac{5}{4}},
\eall\right.\ee
and for $|\alpha|=1,2,3,$
\be\label{nldcre2}\left\{\ball
&\|\dxa f(t)\|_{L^\infty_{v,3}(L^2_x)}+\|\dxa(\rho,u)(t)\|_{L^2_x}+\|\dxa\tdx\Phi(t)\|_{L^2_x}\le C(1+t)^{-\frac{5}{4}},
\\&\|\dxa\tdv f(t)\|_{L^\infty_{v,2}(L^2_x)}\le C\(1+\frac{1}{\sqrt{t}}\)(1+t)^{-\frac{5}{4}}.
\eall\right.\ee
Moreover, if we further assume that there exist two constants $d_0,r_0>0$ such that the Fourier transform $\hat{U}_0(\xi, v)=(\hat{f}_0,\hat{\rho}_0,\hat{u}_0)$ of $U_0(x, v)$ satisfies
$$
\inf\limits_{|\xi|\le r_0}|\tau_a\hat{f}_0+\hat{\rho}_0|\ge d_0,\quad \sup\limits_{|\xi|\le r_0}|\tau_b\hat{f}_0+\hat{u}_0|=0.
$$
Then, for time $t >0$ large enough, the global solution $U=(f,\rho,u)$ satisfies
\be\label{opnldr}\left\{\ball
C_1(1+t)^{-\frac{3}{4}}\le&\| f(t)\|_{L^2_{x,v}}\le C_2(1+t)^{-\frac{3}{4}},
\\C_1(1+t)^{-\frac{3}{4}}\le&\|(\rho,u)(t)\|_{L^2_x}\le C_2(1+t)^{-\frac{3}{4}},
\eall\right.\ee
and for $\gamma>1$, the electric field $\nabla_x\Phi$ satisfies
$$
C_1(1+t)^{-\frac{5}{4}}\le\| \nabla_x\Phi(t)\|_{L^2_x}\le C_2(1+t)^{-\frac{5}{4}},
$$
where $0<C_1\le C_2$ are two generic constants.
\end{thm}
\begin{remark}
For $\gamma=1$, the electric field $\nabla_x\Phi$ does not achieve the optimal time decay rate. In fact, as noted in Remark \ref{phi4jie}, spectrum analysis shows that the linear part of $\nabla_x\Phi$ decays rapidly at least as $(1+t)^{-\frac74}$. However, by higher-order energy estimate \eqref{en2}, the nonlinear terms lead to a slower decay rate $(1+t)^{-\frac54}$, as shown in \eqref{highclose}.
\end{remark}

\begin{remark}
In the study of the NSP and the Vlasov-Poisson-Boltzmann (VPB) systems \cite{nsp1,zhonggfvpb,zhongspvpb}, the self-consistent electric field is known to significantly affect the large time behavior of the solution. A common feature of these systems is that the Poisson equation is driven by a single density, namely, the fluid density $\rho$ in NSP and the macroscopic density $\tau_af$ of particles in VPB. In both cases, the resulting electric field generates the low-frequency nonlocal effects and breaks the usual acoustic wave propagation in the classical Navier Stokes (NS) and Boltzmann equations without electric field coupling \cite{glhsns,glhsb}. More precisely, detailed spectrum analysis and Green's function show that, for the NSP system, the electric field affects the dispersion of fluid and destroys the usual acoustic wave propagation in the classical compressible viscous fluids, leading to the invalidity of Huygens’ principle, and consequently a reduced time decay rate $(1+t)^{-\frac14}$ for both momentum and electric field. For the VPB system, the same slow decay rate is also obtained for the electric field and macroscopic momentum $\tau_bf$ of particles, and the Huygens' type sound wave propagation cannot be observed.

By contrast, the VPFP system, which describes the time evolution of charged particles subject to Brownian motion, exhibits a spectrum structure different from that of the linear Fokker-Planck (FP) equation. For the FP equation, the low-frequency, spectrum consists of a single eigenvalue with the asymptotic expansion $\lambda(s)=-s^2+O(s^3)$, which reveals the diffusion behavior and yields the time decay rate $(1+t)^{-\frac34}$ for the solution \cite{yusplfp}. For the VPFP system, the Poisson equation provides additional macroscopic density dissipation $\|\tau_af\|_{L^2_x}^2$, and the solution decays exponentially at the rate $e^{-\eta t}$ for $\eta>0$, which corresponds to the spectrum gap revealed by spectrum analysis in \cite{vpfp1}.

Different from the results above, the large time behavior of the global solution to the VPFP/NSP system established in Theorem \ref{introth2} shows the solution $U=(f,\rho,u)$ decays at the rate $(1+t)^{-\frac34}$, which is consistent with the decay rates for the NS and FP equations without electric field coupling \cite{yusplfp,ns1,ns2}, while the electric field $\tdx\Phi$ decays at a faster rate $(1+t)^{-\frac54}$.
\end{remark}

\begin{remark}
The relative velocity $\tau_bf-u$ induced by the friction force coupling decays at a faster rate $(1+t)^{-\frac54}$, and the corresponding damping $\|\tau_bf-u\|_{H^3_x}^2$ essentially dissipates the relative motion between particles and fluid and drives them toward synchronization. Furthermore, in the energy estimates, unlike \cite{lxliu5}, we do not need to directly construct the viscous dissipation $\|\tdx\tau_bf\|_{H^2_x}^2$ of particles. Indeed, the inequality in \eqref{macrovis} yields the following relation
$$
\|\nabla_xu\|_{H^3_x}^2+\|\tau_bf-u\|_{H^3_x}^2\ge\frac12\|\tdx\tau_bf\|^2_{H^2_x},
$$
which reveals that the viscous dissipation $\|\tdx\tau_bf\|_{H^2_x}^2$ can be compensated by the high-order fluid viscous dissipation $\|\tdx u\|_{H^3_x}^2$ from the NS equations and the damping $\|\tau_bf-u\|_{H^3_x}^2$. Conversely, for the coupling with the Euler equations, the viscous dissipation of fluid is $\|\tdx u\|_{H^2_x}^2$ instead of $\|\tdx u\|_{H^3_x}^2$, and it can be obtained from $\|\tdx\tau_bf\|_{H^2_x}^2$ of particles through damping $\|\tau_bf-u\|_{H^3_x}^2$ as established in \cite{lxliu5,lxliu4}. This indicates that the damping associated with the relative velocity can transfer the existing viscous dissipation between particles and fluid.
\end{remark}
\begin{remark}
The coupling of the self-consistent Poisson equation $\ddx\Phi=\tau_af-\rho$ is governed by the difference between the macroscopic density $\tau_af$ of particles and fluid density $\rho$, leading to distinct dynamics not observed in either individual system. More precisely, the detailed asymptotic expansions of the eigenvectors $\Psi_j(s,\omega)=(\psi_j, \zeta_j, \vartheta_j)(s,\omega)~(j=-1,0,1,2)$ given in Theorem \ref{srlem13} reveal that
$$
\tau_a\psi_{\pm1}-\zeta_{\pm1}=O(s^2);\quad \tau_a\psi_{j}-\zeta_{j}\equiv 0\quad(j=0,2 ),
$$
which indicates that the source term of the Poisson equation is supported only on the acoustic  branches $\lambda_{\pm1}(s)$ for $s\le r_0$. Combined the expansions above with the decomposition of the semigroup in Theorem \ref{orlem5}, the low-frequency nonlocal effect coming from the operator $\tdx\ddx^{-1}$ vanishes as shown in \eqref{orl20}, and consequently the electric field does not destroy the acoustic wave propagation. The macroscopic momentum $\tau_bf$ and the fluid velocity $u$ decay at the rate $(1+t)^{-3/4}$, while the electric field decays faster at the rate $(1+t)^{-5/4}$. These large time behaviors are completely different from those of solutions to the NSP, VPB, and VPFP systems. Moreover, the source term $\tau_af-\rho$ decays at the rate $(1+t)^{-5/4}$.
\end{remark}

We now briefly summarize the main novelties of this work, outline the main ideas of the proofs of the above theorems, and make several related comments. The present VPFP/NSP system is not a simple combination of the individual VPFP and NSP systems. Instead, the coupling mechanisms, namely the mutual friction force and the self-consistent Poisson equation, synchronize the motions of the two phases and eliminate the low-frequency nonlocal effects of the electric field, restoring the hyperbolic wave propagation behavior of classical compressible fluids.

To the best of our knowledge, this paper presents the first spectrum analysis for the kinetic-fluid-Poisson system. More precisely, the detailed spectrum structure and eigenvectors of the linearized VPFP/NSP operator $\M(\xi)$ defined in \eqref{suanzi} are established in Theorem \ref{srlem13}. For the structure of the spectrum $\sigma(\M(\xi))$, there exist four eigenvalues $\lambda_j(|\xi|)~(j=-1,0,1,2)$ at low frequency and they admit the following asymptotic expansions for $|\xi|\le r_0$
$$
\left\{\ball
\lambda_{\pm1}(|\xi|)&=\pm i\sqrt{\frac{\gamma+1}{2}}|\xi|-\frac{1}{2}|\xi|^2+O(|\xi|^3),
\\
\lambda_0(|\xi|)&=\lambda_2(|\xi|)=-\frac{3}{4}|\xi|^2+O(|\xi|^3),
\eall\right.
$$
which capture both the hyperbolic wave propagation and diffusion behaviors, a phenomenon that has not been observed in the spectrum analysis of the individual VPFP and NSP systems before \cite{nsp1,vpfp1}. Moreover, there is a spectral gap at high frequency, namely, there exists $\beta =\beta (r_0)>0$ such that
$$
\sigma(\M(\xi))\subset\{\lambda\in\mathbb{C}\mid\re\lambda<-\beta \},\quad |\xi|\ge r_0.
$$
Note that the existence of the eigenvalues $\lambda_{\pm1}(s)$ cannot be obtained directly via the implicit function theorem, since the corresponding eigenvalue equation $D_1(\lambda,s)=0$ satisfies $\frac{\partial}{\partial_\lambda}D_1(0,0)=0$. To overcome this degeneracy, the Weierstrass preparation theorem in several complex variables stated in Theorem \ref{wei} is applied to establish the analyticity and asymptotic expansions of $\lambda_{\pm1}(s)$.
Based on the spectrum structure above, the decomposition of the semigroup $ e^{t\M(\xi)}$ is obtained in Theorem \ref{orlem5}, from which we obtain the optimal time decay rates of solution to the linearized VPFP/NSP system \eqref{introe4} in Theorem \ref{linearoptdr}.

For the original nonlinear problem, we establish a uniform energy estimate motivated by \cite{lxliu0,lxliu4}. Combining Duhamel's principle with the energy estimate, we derive the large time behavior of the solution to the original nonlinear problem, which can be proved to be optimal. To overcome the difficulties resulting from the high-order derivative term $\tdv f$ and the velocity weighted term $vf$ in the nonlinear term $R_1$, we introduce the Fokker-Planck equation with damping \eqref{znfp} from \cite{vpfp2}, for which the explicit expression of Green's function has been established, and thereby derive Lemmas \ref{green1} and \ref{green2}. These results further yield a decay rate $(1+1/\sqrt{t})(1+t)^{-\frac34}$ for $\tdv f$ in a suitable weighted velocity space $L^\infty_{v,2}(L^2_x)$, and $(1+t)^{-\frac34}$ for $f$ in the space $L^\infty_{v,3}(L^2_x)$, thereby closing the a priori assumption on the large time behavior of the solution.
\bigskip

The rest of this paper will be organized as follows: in Section 2, we study the spectrum and resolvent of the linear VPFP/NSP operator $\M(\xi)$ and establish the asymptotic expansions of the eigenvalues and eigenfunctions at low frequency. In Section 3, we decompose the semigroup $e^{t\M(\xi)}$ generated by the linear operator $\M(\xi)$ with respect to the low and high frequency, and thus establish the optimal time decay rates of the global solution to the linearized VPFP/NSP system in terms of the semigroup $e^{t\M(\xi)}$. In Section 4, we establish the global existence and the optimal time decay rates of the solution to the original nonlinear VPFP/NSP system. Some useful lemmas and the local existence to the original nonlinear VPFP/NSP system are provided in Section 5 for the convenience of the readers.

\section{Spectrum analysis}

In this section, we are concerned with the spectral analysis of the operator $\M(\xi)$ defined by \eqref{suanzi}, which will be applied to study the optimal time decay rates of the global solution to the VPFP/NSP system \eqref{introe2}.
\subsection{Spectrum and resolvent}
From the system \eqref{introe2}, we have the linearized VPFP/NSP system for $U=(f,\rho,u)$:
\be\label{introe4}
\left\{\ball
&\partial_tU=\mathbbm{M}U,\quad t>0,
\\
&U(0,x,v)=U_0(x,v)=(f_0,\rho_0,u_0),
\eall\right.
\ee
where $\mathbbm{M}$ is the operator on $ L^2(\mathbb{R}^3_v\times\mathbb{R}^3_x)\times L^2_x\times L^2_x$ defined by
$$
\mathbbm{M}=\begin{pmatrix}
-v\cdot \nabla_x+L+v\cdot\nabla_x\ddx ^{-1}P_d&-v\chi_0\cdot\nabla_x\ddx ^{-1}&v\chi_0
\\0&0&-\div_x
\\-\nabla_x\ddx^{-1}\tau_a+\tau_b&-\gamma\nabla_x+\nabla_x\ddx ^{-1}&\ddx I_3-I_3
\end{pmatrix}.
$$
Taking the Fourier transform in system \eqref{introe4} with respec to $x$, we have
$$
\left\{\ball
&\partial_t\hat{U}=\mathbbm{M}(\xi)\hat{U},\quad t>0,
\\
&\hat{U}(0,\xi,v)=\hat{U}_0(\xi,v)=(\hat{f}_0,\hat{\rho}_0,\hat{u}_0).
\eall\right.
$$
Here, for $\xi\neq 0$
\be\label{suanzi}
\mathbbm{M}(\xi)=\begin{pmatrix}
B(\xi)&i(v\cdot\xi)|\xi|^{-2}\chi_0&v\chi_0
\\0&0&-i\xi^T
\\i\xi|\xi|^{-2}\tau_a+\tau_b&-i\xi(\gamma+|\xi|^{-2})&-(1+|\xi|^2)I_3
\end{pmatrix},
\ee
with
$$
B(\xi)=L-i(v\cdot\xi)-i(v\cdot\xi)|\xi|^{-2}P_d.
$$
We decompose $L$ as
$$
L=-A+\frac{3}{2},\quad A=-\ddv+\frac{|v|^2}{4},
$$
where
\be
(Af,f)=\|\nabla_vf\|^2+\frac{1}{4}\|vf\|^2\ge\frac{3}{2}\|f\|^2.
\ee
Thus
\be
B(\xi)=A(\xi)+\frac32-i(v\cdot\xi)|\xi|^{-2}P_d, \quad A(\xi)=-A-i(v\cdot\xi).
\ee
For any vectors $U=(f,\rho_1,u_1), V=(g,\rho_2,u_2)\in L^2(\mathbb{R}^3_v)\times\C\times \C^3$, we define the normal inner product and norm by
$$
(U,V)=(f,g)+(\rho_1,\rho_2)+(u_1,u_2),\quad \|U\|=\sqrt{(U,U)}.
$$
For $\xi\neq0$, we introduce the weighted Hilbert space $Z_{\xi,\gamma}$ as
$$
Z_{\xi,\gamma}=\{U\in L^2(\mathbb{R}^3_v)\times\C\times \C^3\mid \|U\|_{\xi,\gamma}=\sqrt{(U,U)_{\xi,\gamma}}<\infty\},
$$
equipped with the inner product
\be\label{zxigamma1}
(U,V)_{\xi,\gamma}=(f,g)+\gamma\left(\rho_1,\rho_2\right)+\left(u_1,u_2\right)+|\xi|^{-2}\left(\tau_af-\rho_1,\tau_ag-\rho_2\right).
\ee

We denote $\rho(A)$ and $\sigma(A)$ to be resolvent set and spectrum set of the operator $A$, respectively. The essential spectrum of $A$, denoted by $\sigma_{ess}(A)$, is the set $\{\lambda\in \mathbb{C} \,|\, \lambda-A~{\rm is~not~a~Fredholm~operator}\}$ (cf. \cite{Kato}). The discrete spectrum of $A$, denoted by $\sigma_d(A)$, is the set $\sigma(A)\setminus \sigma_{ess} (A)$ which consists of  all isolated eigenvalues with finite multiplicity.
\begin{lem}\label{srlem1}
For any $\xi\in\R^3$, the operator $\mathbbm{M}(\xi)$ generates a strongly continues contraction semigroup on $Z_{\xi,\gamma}$ satisfying
$$
\|e^{t\mathbbm{M}(\xi)}U\|_{\xi,\gamma}\le\|U\|_{\xi,\gamma},\quad t>0,~U\in Z_{\xi,\gamma}.
$$
\end{lem}
\begin{proof}
First we show that both $\mathbbm{M}(\xi)$ and $\mathbbm{M}(\xi)^*$ are dissipative operators on $Z_{\xi,\gamma}$. For any $U=(f,\rho,u),V=(g,\zeta,\vartheta)\in Z_{\xi,\gamma}$, we obtain
\bmas
(\mathbbm{M}(\xi)U,V)_{\xi,\gamma}&=(Lf-i(v\cdot\xi)f,g)-i\xi|\xi|^{-2}\cdot(\tau_af-\rho)\overline{\tau_bg}+u\cdot\overline{\tau_bg}-i\gamma\xi\cdot u\overline{\zeta}
\\&\quad+i\xi|\xi|^{-2}\cdot(\tau_af-\rho)\overline{\vartheta}+\tau_bf\cdot\overline{\vartheta}-i\gamma\xi\rho\cdot\overline{\vartheta}-(1+|\xi|^2)u\cdot\overline{\vartheta}
\\&\quad+\frac1{|\xi|^{2}}i\xi\cdot(-\tau_bf+u)(\overline{\tau_ag}-\overline{\zeta})
\\&=(U,\mathbbm{M}(\xi)^*V)_{\xi,\gamma},
\emas
where
$$
\mathbbm{M}(\xi)^*=\mathbbm{M}(-\xi)=\begin{pmatrix}
B(-\xi)&-i(v\cdot\xi)|\xi|^{-2}\chi_0&v\chi_0
\\0&0&i\xi^T
\\-i\xi|\xi|^{-2}\tau_a+\tau_b&i\xi(\gamma+|\xi|^{-2})&-(1+|\xi|^2)I_3
\end{pmatrix}.
$$
Let $U=(f,\rho,u)\in Z_{\xi,\gamma}$ and we have
\bmas
(\mathbbm{M}(\xi)U,U)_{\xi,\gamma}=&(Lf,f)+\tau_bf\cdot\overline{u}+\overline{\tau_bf}\cdot u-(1+|\xi|^{2})|u|^2
\\&-i((v\cdot\xi)f,f)-i\xi(\gamma+|\xi|^{-2}) \cdot(\rho\overline{u}+\overline{\rho}u)
\\&+i\xi|\xi|^{-2}\cdot(-\tau_af\overline{\tau_bf}-\overline{\tau_af}\tau_bf+\tau_bf\overline{\rho}+\overline{\tau_bf}\rho+\tau_af\overline{u}+\overline{\tau_af}u).
\emas
By using $(Lf,f)\le-\|P_rf\|^2$ and $P_r=P_m+P_1$, we obtain
\bmas
&\re(\mathbbm{M}(\xi)U,U)_{\xi,\gamma}=\re(\mathbbm{M}(\xi)^*U,U)_{\xi,\gamma}
=(Lf,f)+\tau_bf\cdot\overline{u}+\overline{\tau_bf}\cdot u-|u|^2-|\xi|^{2}|u|^2
\\
&\le-\|P_1f\|^2-|\tau_bf-u|^2-|\xi|^{2}|u|^2\le0,
\emas
which means $\mathbbm{M}(\xi)$ and $\mathbbm{M}(\xi)^*$ are dissipative. From Lemma \ref{S_1-1}, the densely defined closed operator $\mathbbm{M}(\xi)$ generates a $C_0$-contraction semigroup on $Z_{\xi,\gamma}$. The proof of the lemma is completed.
\end{proof}

For $\xi\neq0$, we define a $4\times4$ matrix as
$$
B_1(\xi)=
\begin{pmatrix}
0&-i\xi^T
\\-i\xi(\gamma+{|\xi|^{-2}})&-(1+{|\xi|^{2}})I_3
\end{pmatrix}.
$$
\begin{lem}\label{orlem3}Let $X=(\rho_1,u_1)$, $Y=(\rho_2,u_2)\in \C\times\C^3$. We introduce the weighted complex inner product space $\mathcal{C}_{\xi,\gamma}$ equipped with the weighted inner product and norm
$$
(X,Y)_{\mathcal{C}_{\xi,\gamma}}=(\gamma+|\xi|^{-2})(\rho_1,\rho_2)+(u_1,u_2) ,\quad \|X\|_{\mathcal{C}_{\xi,\gamma}}=\sqrt{(X,X)_{\mathcal{C}_{\xi,\gamma}}}.
$$
Then the matrix $B_1(\xi)$ satisfies the following properties:

(1) For $\lambda\neq\eta_j~(-1\le j\le2)$, the operator $\lambda-B_1(\xi)$ is invertible on $\mathcal{C}_{\xi,\gamma}$ and  satisfies
\be\label{B1sp}
\|(\lambda-B_1(\xi))^{-1}X\|_{\mathcal{C}_{\xi,\gamma}}\le \max\limits_{-1\le j\le 2}\frac{1}{|\lambda-\eta_j|}\|X\|_{\mathcal{C}_{\xi,\gamma}},
\ee
where
$\eta_j,~j=-1,0,1,2$ are the the eigenvalues of $B_1(\xi)$ given by
\be \eta_0=\eta_2=-(1+|\xi|^{2}),\quad \eta_{\pm1}= -\frac12(1+|\xi|^{2})\pm\frac12\sqrt{|\xi|^4+(2-4\gamma)|\xi|^2-3} . \label{eigen2}\ee

(2) Let $\zeta(\gamma)=\sqrt{(2\gamma-1)+\sqrt{(2\gamma-1)^2+3}}$. For $|\xi|\ge\zeta(\gamma)$, $\eta_{\pm1}(|\xi|)$ are real functions and satisfy that $\eta_1$ is monotonically increasing in $|\xi|$ and tends to $-\gamma$ as $|\xi|\to+\infty$ while $\eta_{-1}$ is monotonically decreasing in $|\xi|$ and tends to $-\infty$ as $|\xi|\to+\infty$. For $|\xi|\le\zeta(\gamma)$, $\eta_{\pm1}(|\xi|)$ are conjugate complex functions and satisfy that $\re\eta_{\pm1}=-\frac12(1+|\xi|^2)$. 
\end{lem}
\begin{proof}
From the eigenvalue problem
$$
\det(\lambda I_4-B_1(\xi))=[\lambda+(1+|\xi|^{2})]^2[\lambda^2+(1+|\xi|^{2})\lambda+1+\gamma|\xi|^2]=0,
$$
we can obtain  the eigenvalues $\eta_j $  and the eigenvectors $\beta_j,$ $j=-1,0,1,2$ given by
\be\label{B1vector1}
\left\{\bln
&\eta_0=\eta_2=-(1+|\xi|^{2}),\quad \eta_{\pm1}= -\frac12(1+|\xi|^{2})\pm\frac12\sqrt{|\xi|^4+(2-4\gamma)|\xi|^2-3}, \\
&\beta_j=(0,W^j(\omega)),~~j=0,2;\quad \beta_{\pm1}= (-i|\xi|, \eta_{\pm1}\omega),
\eln\right.
\ee
where $W^k=W^k(\omega)~(k=0,2)$ are normal vectors satisfying
$$
W^0 \cdot W^2 =0,\quad W^0 \cdot \omega=W^2 \cdot \omega=0.
$$
It's easy to verify that the eigenvectors $\beta_j~(-1\le j\le2)$ satisfy the orthogonal relation
$$(\beta_i,\beta_j)_{\mathcal{C}_{\xi,\gamma}}=0, \quad -1\le i\ne j\le2.$$
Thus, $\mathcal{C}_{\xi,\gamma}=\mbox{span}\{\beta_j,~ j=-1,0,1,2\}$. By Lemma \ref{srlem6}, we obtain for $\lambda\neq\eta_j~(-1\le j\le2)$ that
$$
\|(\lambda-B_1(\xi))^{-1}X\|^2_{\mathcal{C}_{\xi,\gamma}}= \sum\limits_{j=-1}^2\frac{1}{|\lambda-\eta_j|^2}\frac{1}{\|\beta_j\|_{\mathcal{C}_{\xi,\gamma}}^2} |(X,\beta_j)_{\mathcal{C}_{\xi,\gamma}} |^2
\le  \max\limits_{-1\le j\le 2}\frac{1}{|\lambda-\eta_j|^2}\|X\|^2_{\mathcal{C}_{\xi,\gamma}}.
$$

Let $h(|\xi|)=|\xi|^4+(2-4\gamma)|\xi|^2-3$. Then $h(|\xi|)$ has a unique positive root $\zeta(\gamma)=\sqrt{(2\gamma-1)+\sqrt{(2\gamma-1)^2+3}}$. For $|\xi|\ge\zeta(\gamma)$, $\eta_{\pm1}(|\xi|)$ are real functions and satisfy
$$\frac{d}{d|\xi|}\eta_1(|\xi|)=-|\xi|\bigg(1-\frac{|\xi|^2+1-2\gamma}{\sqrt{h(|\xi|)}}\bigg)>0, \quad \frac{d}{d|\xi|}\eta_{-1}(|\xi|)=-|\xi|\bigg(1+\frac{|\xi|^2+1-2\gamma}{\sqrt{h(|\xi|)}}\bigg)<0.$$
Moreover,
\bmas
\eta_{1}&=-\gamma-(\gamma^2-\gamma+1)|\xi|^{-2}+O(|\xi|^{-4})\to-\gamma,\quad |\xi|\to+\infty,
\\
\eta_{-1}&=-|\xi|^2+(\gamma-1)+(\gamma^2-\gamma+1)|\xi|^{-2}+O(|\xi|^{-4})\to-\infty,\quad |\xi|\to+\infty.
\emas
For $|\xi|\le\zeta(\gamma)$, $\eta_{\pm1}(|\xi|)$ are conjugate complex functions and satisfy
$$
\eta_{\pm1}=\frac{-(1+|\xi|^{2})}{2}\pm i\frac{\sqrt3}{2}\sqrt{1-\[\frac{(2-4\gamma)}{3}|\xi|^2+\frac{1}{3}|\xi|^4\]}.
$$
The proof of the lemma is completed.
\end{proof}

\begin{lem}\label{srlem2}The following conditions hold for all $\xi\neq0$.
 \begin{enumerate}
\item[(1)] $\sigma_{ess}(\mathbbm{M}(\xi))\subset\{\lambda\in\mathbb{C}\mid\re\lambda\le-\frac32\}$ and $\sigma(\mathbbm{M}(\xi))\cap\{\lambda\in\mathbb{C}\mid-\frac32<\re\lambda\le0\}\subset\sigma_d(\mathbbm{M}(\xi))$.

\item[(2)] If $\lambda(\xi)$ is an eigenvalue of $\mathbbm{M}(\xi)$, then $\re\lambda(\xi)<0$ for any $\xi\neq0$.
 \end{enumerate}
\end{lem}
\begin{proof}
 We decompose $\mathbbm{M}(\xi)$ as
$$
\mathbbm{M}(\xi) =G_1(\xi)+G_2(\xi),
$$
with
\bma\label{sra2}
&G_1(\xi)=\begin{pmatrix}
A(\xi)&0 &0\\
0 &0&-i\xi^T\\
0&-i\xi(\gamma+|\xi|^{-2})&-(1+{|\xi|^{2}})I_3
\end{pmatrix},
\\
&G_2(\xi)=\begin{pmatrix}
\frac32-i(v\cdot\xi)|\xi|^{-2}P_d& i(v\cdot\xi)|\xi|^{-2}\chi_0&v\chi_0 \\
0&0&0\\
i \xi |\xi|^{-2}\tau_a+\tau_b&0&0
\end{pmatrix}.\label{G2}
\ema
Here $A(\xi)=-A-i(v\cdot\xi)$ is a closed operator that generates a $C_0-$contraction semigroup on $L^2(\R^3_v)$, i.e., $\re(A(\xi)f,f)=-(Af,f)\le-\frac32\|f\|^2$, $f\in D(A)$, and then, $\{\lambda\in\mathbb{C}\mid\re\lambda>-\frac32\}\subset\rho(A(\xi))$, $\sigma_{ess}(A(\xi))\subset\sigma(A(\xi))\subset\{\lambda\in\mathbb{C}\mid\re\lambda\le-\frac32\}$. 
Combining this and Lemma \ref{orlem3}, $\lambda-G_1(\xi)$ is invertible for $ \re\lambda>-\frac32 $ and $\lambda\ne \eta_j~(j=-1,0,1,2)$. Since $\frac32$ is $A(\xi)-$compact \cite{vmfp1}, i.e., for any bounded sequences $\{f_n\}$ and $\{A(\xi)f_n\}$, $\{\frac32f_n\}$ contains a convergent subsequence, the other element operators in $G_2(\xi)$ are compact operators for any fixed $\xi\neq0$, $G_2(\xi)$ is $G_1(\xi)-$compact on $L^2(\R^3_v)\times\C^4$.
By Theorem 5.35 on p.244 of \cite{Kato}, $\mathbbm{M}(\xi)$ and $G_1(\xi)$ have the same essential spectrum, namely $\sigma_{ess}(\mathbbm{M}(\xi))=\sigma_{ess}(A(\xi))\subset\{\lambda\in\mathbb{C}\mid\re\lambda\le-\frac32\}$. Thus, the spectrum of $\M(\xi)$ in the domain $\re\lambda>-\frac32$ only consists of discrete eigenvalues. This proves part (1).

Now, we want to show part (2). Let $U=(f,\rho,u) $ be the eigenvector corresponding to the eigenvalue $\lambda$ of  $\mathbbm{M}(\xi)$. Then
\be\label{sra3}
\left\{\ball
\lambda f&=Lf-i(v\cdot\xi) f-i(v\cdot\xi)|\xi|^{-2}(P_df-\rho\chi_0)+v\chi_0\cdot u,
\\
\lambda\rho&=-i\xi\cdot u, &
\\
\lambda u&=i\xi|\xi|^{-2}\tau_af+\tau_bf-i\gamma\xi\rho-i\xi|\xi|^{-2}\rho-(1+|\xi|^2)u.
\eall\right.\ee
Taking the inner product $(\cdot,\cdot)_{\xi,\gamma}$ of \eqref{sra3} with $U$, we have
$$
\re\lambda \|U\|^2_{\xi,\gamma} =\re(\mathbbm{M}(\xi)U,U)_{\xi,\gamma} \le-\|P_1f\|^2-|\tau_bf-u|^2-|\xi|^{2}|u|^2
\le0,
$$
which implies $\re\lambda\le0$. Suppose that there exists an eigenvalue $\lambda$ satisfying $\re\lambda=0$. Then we have
\be\label{sra4}
0\le-\|P_1f\|^2-|\tau_bf-u|^2-|\xi|^{2}|u|^2\le-|\xi|^{2}|u|^2\le0,
\ee
which implies $u=0$. Substituting it into \eqref{sra4}, we have $(Lf,f)=0$, which means $f\in {\rm ker} L$. By substituting $f=C_0\chi_0 $ and $u=0$ into \eqref{sra3}, we have
\be\label{sra5}\left\{\ball
&\lambda C_0\chi_0=-i\xi\cdot v\chi_0C_0-i(v\cdot\xi)|\xi|^{-2}\chi_0(C_0-\rho),
\\
&\lambda\rho=0,
\\
&i\xi|\xi|^{-2}C_0=i\xi(\gamma+|\xi|^{-2})\rho,
\eall\right.\ee
which implies $\rho=0$ and $C_0=0$. Therefore,   $U=(f,\rho,u)=0$, which is a contradiction. Thus it holds $\re\lambda<0$ for any $\lambda\in\sigma_d(\mathbbm{M}(\xi))$. The proof of Lemma is completed.
\end{proof}
We denote $T$ to be a linear operator on $L^2(\mathbb{R}^3_v)\times\C \times\C^3$ or $Z_{\xi,\gamma}$,  and define the corresponding norms of $T$ by
$$
\|T\|=\sup\limits_{\|U\|=1}\|TU\|, \quad \|T\|_{\xi,\gamma}=\sup\limits_{\|U\|_{\xi,\gamma}=1}\|TU\|_{\xi,\gamma}.
$$
For $\re\lambda>-\frac32$, $\lambda\ne \eta_j~(j=-1,0,1,2)$ and $|\xi|>0$, we decompose $\mathbbm{M}(\xi)$ as follows:
$$
\lambda-\mathbbm{M}(\xi)=\lambda-G_1(\xi)-G_2(\xi)=(I-G_2(\xi)(\lambda-G_1(\xi))^{-1})(\lambda-G_1(\xi)),
$$
where $G_1(\xi)$ and $G_2(\xi)$ are defined by \eqref{sra2} and \eqref{G2}, respectively. Thus
\bmas
&(\lambda-G_1(\xi))^{-1}=\begin{pmatrix}
(\lambda-A(\xi))^{-1}&0\\
0&(\lambda-B_1(\xi))^{-1}
\end{pmatrix},
\\
&G_2(\xi)(\lambda-G_1(\xi))^{-1}=\begin{pmatrix}
\mathcal{X}_1(\lambda,\xi)&\mathcal{X}_2(\lambda,\xi)\\
\mathcal{X}_3(\lambda,\xi)&0_{4\times4}\end{pmatrix},
\emas
where
\be\label{sra6}
\left\{\ball
\mathcal{X}_1(\lambda,\xi)&=\(\frac32-i\frac{v\cdot\xi}{|\xi|^{2}}P_d\)(\lambda-A(\xi))^{-1},
\\
\mathcal{X}_2(\lambda,\xi)&=\(i(v\cdot\xi)|\xi|^{-2}\chi_0,v\chi_0\)(\lambda-B_1(\xi))^{-1},
\\
\mathcal{X}_3(\lambda,\xi)&=\begin{pmatrix}0\\i\xi|\xi|^{-2}\tau_a+\tau_b\end{pmatrix}(\lambda-A(\xi))^{-1}.
\eall\right.\ee

\begin{lem}[\cite{vmfp1}] \label{Azhengze} There exists a constant $C>0$ such that the following holds.

(1) For any $\delta > 0$, we have
$$
\sup\limits_{\re\lambda\geq-\frac{3}{2}+\delta,\im\lambda\in\mathbb{R}}\|(\lambda-A(\xi))^{-1}\|\leq C(1+\delta^{-1})(1+|\xi|)^{-\frac23}.
$$

(2) For any $\delta>0$ and $\tau_0>0$, we have
$$
\sup\limits_{\re\lambda\geq-\frac{3}{2}+\delta,|\xi|\leq\tau_0}\|(\lambda-A(\xi))^{-1}\|\leq C(1+\delta^{-1})(1+\tau_0)(1+|\im\lambda|)^{-1}.
$$
\end{lem}

We have the spectrum gap of the operator $\mathbbm{M}(\xi)$ for high frequency.
\begin{lem}\label{srlem5}
For any $r_0>0$, there exists $\beta =\beta (r_0)>0$ such that for $|\xi|\ge r_0$,
$$
\sigma(\mathbbm{M}(\xi))\subset\{\lambda\in\mathbb{C}\mid\re\lambda<-\beta \}.
$$
\end{lem}

\begin{proof}
Let $\lambda\in\sigma(\M(\xi))\cap\{\lambda\in\C\mid\re\lambda>-1/2\}$. Firstly, we claim that $\sup_{|\xi|\ge r_0}|\im\lambda(\xi)|<+\infty$. By Lemma \ref{orlem3} and Lemma \ref{Azhengze}, there exists a sufficiently large constant $r_1>\zeta(\gamma)>1$ such that for $\re\lambda>-1/2$ and $|\xi|\ge r_1$,
\bmas
\|\mathcal{X}_1(\lambda,\xi)\|&=\Big\|\Big(\frac32-i\frac{v\cdot\xi}{|\xi|^{2}}P_d\Big)(\lambda-A(\xi))^{-1}\Big\|\le \frac{1}{2},
\\
\|\mathcal{X}_2(\lambda,\xi)\| &=  \|(i(v\cdot\xi)|\xi|^{-2}\chi_0,v\chi_0)(\lambda-B_1(\xi))^{-1}\| \le\frac{2}{2\gamma-1} ,
\\
\|\mathcal{X}_3(\lambda,\xi)\| &=\|(i\xi|\xi|^{-2}\tau_a+\tau_b)(\lambda-A(\xi))^{-1}\| \le\frac{1}{4}(2\gamma-1).
\emas
Therefore, we obtain
\be\label{sra8}
\|\mathcal{X}_1(\lambda,\xi)\|\le \frac12, \quad \|\mathcal{X}_2(\lambda,\xi)\| \|\mathcal{X}_3(\lambda,\xi)\| \le\frac{1}{2}.
\ee
This and Lemma \ref{inver} imply that the operator $I-G_2(\xi)(\lambda-G_1(\xi))^{-1}$ is invertible on $L^2(\mathbb{R}^3_v)\times\C \times\C^3$ for $\re\lambda>-1/2$ and $|\xi|\ge r_1$, and
thus $\lambda-\mathbbm{M}(\xi)$ is also invertible on $L^2(\R^3_v)\times\C \times\C^3$ for $\re\lambda>-1/2$ and $|\xi|\ge r_1$ satisfying
\be\label{sra9}
(\lambda-\mathbbm{M}(\xi))^{-1}=(\lambda-G_1(\xi))^{-1}(I-G_2(\xi)(\lambda-G_1(\xi))^{-1})^{-1},
\ee
namely,
$$
\{\lambda\in\mathbb{C}\mid{\re\lambda>-1/2}\}\subset\rho(\mathbbm{M}(\xi)),\quad |\xi|>r_1.
$$
As for $r_0\le|\xi|\le r_1$, by Lemma \ref{orlem3} and Lemma \ref{Azhengze}, there exists a constant $y_1=y_1(r_0,r_1)>0$ large enough such that \eqref{sra8} still holds for $|\im \lambda|> y_1$.
This also implies that the invertibility of the operator $\lambda-\mathbbm{M}(\xi)$ on $L^2(\R^3_v)\times\C \times\C^3$, namely
$$
\{\lambda\in\mathbb{C}\mid{\re\lambda>-1/2},~|\im \lambda|> y_1\}\subset\rho(\mathbbm{M}(\xi)),\quad r_0\le|\xi|\le r_1.
$$
Thus, we conclude that
$$
\sigma(\mathbb{M}(\xi))\cap\{\lambda\in\mathbb{C}\mid{\re\lambda>-1/2}\}\subset\{\lambda\in\mathbb{C}\mid{\re\lambda>-1/2},~|\im\lambda|\le y_1\},\quad |\xi|\ge r_0.
$$
Next, we prove $\sup_{|\xi|\ge r_0}\re\lambda(\xi)<0$. Base on the argument above, we only need to show $\sup_{|\xi|\in[r_0,r_1]}\re\lambda(\xi)<0$. If it does not hold, then there exists a sequence of $\{\xi_n,\lambda_n, U_n=(f_n,\rho_n,u_n)\}$ satisfying $|\xi_n|\in[r_0,r_1]$, $U_n=(f_n,\rho_n,u_n)\in L^2(\R^3_v)\times\C \times\C^3$ with $\|U_n\|=1$, $\lambda_n\in\sigma(\mathbbm{M}(\xi_n))$, $\re\lambda_n\to0$ as $n\to\infty$ such that
\be\label{sra10}\left\{\ball
\lambda_n f_n&=Lf_n-i(v\cdot\xi_n)f_n-i(v\cdot\xi_n)|\xi_n|^{-2}(P_df_n-\rho_n\chi_0)+v\chi_0\cdot u_n,
\\
\lambda_n\rho_n&=-i\xi_n\cdot u_n,
\\
\lambda_n u_n&=i\xi_n|\xi_n|^{-2}\tau_af_n+\tau_bf_n-i\xi_n(\gamma+|\xi_n|^{-2})\rho_n-(1+|\xi_n|^{2})u_n.
\eall\right.\ee
We write $\eqref{sra10}_1$ as
$$
(\lambda_n-A(\xi_n))f_n=\frac32f_n-i\frac{v\cdot\xi_n}{|\xi_n|^{2}}(P_df_n-\chi_0\rho_n)+v\chi_0\cdot u_n,
$$
namely
$$
f_n=(\lambda_n-A(\xi_n))^{-1}g_n,~~g_n=\frac32f_n-i\frac{v\cdot\xi_n}{|\xi_n|^{2}}(P_df_n-\chi_0\rho_n)+v\chi_0\cdot u_n.
$$
Since
$$\re((\lambda_n-A(\xi_n))f_n,f_n)=\re\lambda_n\|f_n\|^2+(Af_n,f_n)\ge \(\frac14-\frac13|\re\lambda_n|\)\|f_n\|^2_{L^2_\sigma},$$
it follows that
$$\|(\lambda_n-A(\xi_n))^{-1}g_n\|_{L^2_\sigma}\le \(\frac14-\frac13|\re\lambda_n|\)^{-1}\|g_n\|\le C.$$
Since $L^2_\sigma$ is compact embedding into $L^2_v$, there exists a subsequence $\{f_{n_j}\}\subset\{f_n\}$ and $f_0\in L^2(\R^3_v)$ such that
$$
f_{n_j}=(\lambda_{n_j}-A(\xi_{n_j}))^{-1}g_{n_j}\to f_0,~~j\to\infty.
$$
Since $|\xi_n|\in[r_0,r_1]$, $|\rho_n|+|u_n|\le1$, there exists a subsequence $\{\xi_{n_j},\rho_{n_j},u_{n_j}\} $ such that $(\xi_{n_j},\rho_{n_j},u_{n_j})\to(\xi_0,\rho_0,u_0)$ as $j\to\infty$. Since $|\im\lambda_n|\le y_1$ and $\re\lambda_n\to0$, there exists a subsequence $\{\lambda_{n_j}\}\subset\{\lambda_n\}$ such that $\lambda_{n_j}\to\lambda_0$ with $\re\lambda_0=0$. Combining the results above and \eqref{sra10}, we obtain $\mathbbm{M}(\xi_0)U_0=\lambda_0U_0$ with $U_0=(f_0,\rho_0,u_0)\in L^2(\R^3_v)\times\C \times\C^3$. Thus, $\lambda_0$ is an eigenvalue $\mathbbm{M}(\xi_0)$ with $\re\lambda_0=0$, which contradicts the fact that $\re\lambda<0$ for $\xi\neq0$ obtained by Lemma \ref{srlem2}. The proof of Lemma is completed.
\end{proof}

We study the spectrum and resolvent sets of $\mathbbm{M}(\xi)$ in the low-frequency regime. Let $P_A$, $P_B$ are the orthogonal macro-micro projection operators in $L^2(\R^3)\times \C\times \C^3$ defined by
$$
P_A=\begin{pmatrix}
P_d&&\\&1&\\&&I_3
\end{pmatrix},
\quad P_B=\begin{pmatrix}
P_r&&\\&0&\\&&0
\end{pmatrix}.
$$
By the above macro-micro decomposition, we decompose $\lambda-\mathbbm{M}(\xi)$ as
$$
\lambda-\mathbbm{M}(\xi)=\lambda P_A-G_3(\xi)+\lambda P_B-G_4(\xi)-G_5(\xi) ,
$$
where
\bma\label{g3g4}
G_3(\xi)& =\begin{pmatrix}
0&0&0 \\
0&0&-i\xi^T\\
i\frac{\xi}{|\xi|^{2}}\tau_a&-i\xi(\gamma+\frac{1}{|\xi|^{2}})&-(1+|\xi|^2)I_3
\end{pmatrix},\\
G_4(\xi)& =\begin{pmatrix}L-iP_r(v\cdot\xi)P_r&0 &0\\
0 &0&0\\
0&0&0
\end{pmatrix},
\\
G_5(\xi)&=\begin{pmatrix}
-i(v\cdot\xi)(1+\frac{1}{|\xi|^{2}})P_d-iP_d(v\cdot\xi)P_r& i\frac{v\cdot\xi}{|\xi|^{2}}\chi_0 & v\chi_0 \\
0&0&0\\
\tau_b&0&0
\end{pmatrix}.\label{sraa11}
\ema

\begin{lem}\label{srlem8}
Let $\xi\neq0$. We have the following properties:

(1) For  $\lambda\neq\tilde{\eta}_j$, the operator $\lambda-G_3(\xi)$ is invertible on $ N_0\times\C \times\C^3 $ satisfying
\bma
\|(\lambda -G_3(\xi))^{-1}P_A\|_{\xi,\gamma}&\le \max\limits_{-1\le j\le 3}\frac{1}{|\lambda-\tilde{\eta}_j|}, \label{sraa10}\\
\|G_5(\xi)(\lambda -G_3(\xi))^{-1}P_A\|_{\xi,\gamma}&\le C(1+|\xi|)\max\limits_{-1\le j\le 3}\frac{1}{|\lambda-\tilde{\eta}_j|}.\label{sraa4}
\ema
where
$\tilde{\eta}_j,~j=-1,0,1,2,3$  are the eigenvalues of $G_3(\xi)$ given by
\be \tilde{\eta}_0=\tilde{\eta}_2=-(1+|\xi|^{2}), \quad \tilde{\eta}_{\pm1}=-\frac12(1+|\xi|^{2})\pm\frac12\sqrt{|\xi|^4+(2-4\gamma)|\xi|^2-3}, \quad \tilde{\eta}_3=0. \ee

(2) If $\re\lambda>-1$, then the operator $\lambda -G_4(\xi)$ is invertible on $N_0^{\perp}\times\{0\}\times\{0\}$ satisfying
\bma
\|(\lambda -G_4(\xi))^{-1}P_B\| &\le (1+\re\lambda)^{-1},\label{sraa12}
\\
\|G_5(\xi)(\lambda -G_4(\xi))^{-1}P_B\|_{\xi,\gamma}&\le C(1+|\lambda|)^{-1}[1+(1+\re\lambda)^{-1}](1+|\xi|)^2.\label{sraa14}
\ema
\end{lem}

\begin{proof} We can prove \eqref{sraa10} by the same way as \eqref{B1sp}. Indeed, we can solve  the eigenvalues $\tilde{\eta}_j=\eta_j$ $(-1\le j\le2)$ and $\tilde\eta_3=0$ via $\det(\lambda I_5-G_3(\xi))=\lambda\det(\lambda I_4-B_1(\xi))$, and the corresponding eigenvectors
$$
\tilde{\beta}_j=(0,\beta_j),~-1\le j\le2;\quad \tilde{\beta}_3=((1+\gamma|\xi|^2)\chi_0,1,0),
$$
where $\eta_j,\beta_j~(-1\le j\le2)$ are given by \eqref{B1vector1}. 
It is easy to verify that the eigenvectors $\tilde{\beta}_j~(-1\le j\le3)$ satisfy
$$(\tilde{\beta}_i,\tilde{\beta}_j)_{\xi,\gamma}=0, \quad -1\le i\ne j\le3.$$
By Lemma \ref{srlem6}, we obtain
$$
\|(\lambda-G_3(\xi))^{-1} U\|^2_{\xi,\gamma} =\sum\limits_{j=-1}^3\frac{1}{|\lambda-\tilde{\eta}_j|^2}\frac{1}{\|\tilde{\beta}_j\|^2_{\xi,\gamma}}|(U,\tilde{\beta}_j)_{\xi,\gamma}|^2\le  \max\limits_{-1\le j\le3}\frac{1}{|\lambda-\tilde{\eta}_j|^2}\|U\|^2_{\xi,\gamma}.
$$
Now we want to prove \eqref{sraa4}. For any $U=(f,\rho,u)\in N_0\times\C \times\C^3 $, we have
$$
G_5(\xi)U=( -i(v\cdot\xi)f-i(v\cdot\xi)|\xi|^{-2}\chi_0(\tau_af-\rho)+v\chi_0\cdot u,0,0)^T,
$$
which gives rise to
\be
 \|G_5(\xi)U\|_{\xi,\gamma}\le|\xi| \|f\|+\frac{1}{|\xi|}|\tau_af-\rho|+|u|\le (1+|\xi|)\|U\|_{\xi,\gamma}.\label{sraa5}
\ee
Combined with  \eqref{sraa10}, we can prove \eqref{sraa4}.

 For any $U=(f,0,0)$ with $f\in  N_0^\bot $, we have   $G_5(\xi)U=\(-iP_d(v\cdot\xi) f,0,\tau_bf\)^T$, which leads to
\be\label{sraa15}
\|G_5(\xi)U\|_{\xi,\gamma}\le \(1+\frac{1}{|\xi|}\)\|P_d(v\cdot\xi) f\| +|\tau_bf|\le C(1+|\xi|)\|U\|_{\xi,\gamma}.
\ee
Let $Q(\xi)=L-iP_r(v\cdot\xi)P_r$. By the same argument as Lemma 3.5 in \cite{zhongspbvpb}, we have that $\lambda-Q(\xi)$ is invertible on $N_0^\perp$ and satisfies that $\|(\lambda-Q(\xi))^{-1}\|\le(1+\re\lambda)^{-1}$ for $\re\lambda>-1$.
By \eqref{sraa15} and using the fact that
$$(\lambda -G_4(\xi))^{-1}U=((\lambda-Q(\xi))^{-1}f,0,0)^T, \quad U=(f,0,0),$$ 
 we obtain \eqref{sraa12} and
\be \|G_5(\xi)(\lambda -G_4(\xi))^{-1} U\|_{\xi,\gamma} \le C(1+|\xi|)(1+\re\lambda)^{-1}\|U\|_{\xi,\gamma}. \label{sraa13} \ee

On the other hand, we decompose
$$
(\lambda-G_4(\xi) )^{-1}=\lambda^{-1} +\lambda^{-1}G_4(\xi) (\lambda-G_4(\xi) )^{-1}.
$$
Noticing that
$
G_5(\xi) G_4(\xi) U=(-iP_d(v\cdot\xi) Q(\xi) f,0,\tau_bQ(\xi) f)^T,
$
we have $$\|G_5(\xi) G_4(\xi) U\|_{\xi,\gamma}\le C(1+|\xi|)^2\|U\|_{\xi,\gamma}.$$
Combined with the formula above, we obtain
\be
\|G_5(\xi)(\lambda -G_4(\xi))^{-1}P_B\|_{\xi,\gamma}\le C|\lambda|^{-1}[1+(1+\re\lambda)^{-1}](1+|\xi|)^2. \label{estimate1}
\ee
Combining \eqref{sraa13} and \eqref{estimate1}, we obtain \eqref{sraa14}. The proof of the lemma is completed.
\end{proof}

\begin{lem}\label{srlem9}
For all $\xi\neq0$, there exists a constant $y_2 >0$ such that
\be\label{sraa6}
\{\lambda\in\mathbb{C}\mid{\re\lambda>-1/2},~|\im\lambda|\ge y_2\}\cup\left\{\lambda\in\mathbb{C}\mid\re\lambda>0\right\}\subset\rho(\mathbbm{M}(\xi)).
\ee
\end{lem}

\begin{proof}
By Lemma \ref{srlem8}, we have for $\re\lambda>-1 $ and $\lambda\neq{\tilde\eta_j(\xi)}~(-1\le j\le3)$, the operator $\lambda P_A-G_3(\xi)+\lambda P_B-G_4(\xi)$ is invertible on $Z_{\xi,\gamma}$ satisfying
$$
(\lambda P_A-G_3(\xi)+\lambda P_B-G_4(\xi))^{-1}=(\lambda -G_3(\xi))^{-1}P_A+(\lambda -G_4(\xi))^{-1}P_B,
$$
because the operator $\lambda P_A-G_3(\xi)$ is orthogonal to $\lambda P_B-G_4(\xi)$. Therefore,
\bma
\nnm\lambda-\mathbbm{M}(\xi)=&\,[I-Y_1(\lambda,\xi)][\lambda P_A-G_3(\xi)+\lambda P_B-G_4(\xi)],
\\
Y_1(\lambda,\xi):=&\,\label{sraa7}G_5(\xi)[(\lambda -G_4(\xi))^{-1}P_B+ (\lambda -G_3(\xi))^{-1}P_A].
\ema
As shown in the proof of the lemma \ref{srlem5},  there exists a large constant $r_1>\zeta(\gamma)>1$ such that
$$
\{\lambda\in\mathbb{C}\mid{\re\lambda>-1/2}\}\subset\rho(\mathbbm{M}(\xi)),\quad |\xi|>r_1.
$$
For  $|\xi|\le r_1$, by Lemma \ref{srlem8}, we can choose $y_2 >0$ to obtain that for $\re\lambda>-1$ and $|\im \lambda|\ge y_2$,
\be\label{sraa8}
\|G_5(\xi)(\lambda -G_4(\xi))^{-1}P_B\|_{\xi,\gamma}\le\frac{1}{4},\quad \|G_5(\xi)(\lambda -G_3(\xi))^{-1}P_A\|_{\xi,\gamma}\le\frac{1}{4},
\ee
which implies $I-Y_1(\lambda,\xi)$ is invertible on $Z_{\xi,\gamma}$ and thus $\lambda-\mathbbm{M}(\xi)$ is also invertible on $Z_{\xi,\gamma}$ satisfying
\be\label{sraa9}
(\lambda-\mathbbm{M}(\xi))^{-1}=[(\lambda -G_3(\xi))^{-1}P_A+(\lambda -G_4(\xi))^{-1}P_B][I-Y_1(\lambda,\xi)]^{-1}.
\ee
Therefore, we obtain
$$
\{\lambda\in\mathbb{C}\mid\re\lambda>-1/2,~|\im y|\ge y_2\}\subset\rho(\mathbbm{M}(\xi)),\quad |\xi|\le r_1.
$$
This and Lemma \ref{srlem1} lead to \eqref{sraa6}. The proof of lemma is completed.
\end{proof}

\subsection{Low Frequency Asymptotics of Eigenvalues}
We study the asymptotics of the eigenvalue and eigenfunction of the operator $\mathbbm{M}(\xi)$ for low frequency. For any $U=(f,\rho,u)\in Z_{\xi,\gamma}$, the eigenvalue problem $\lambda U=\mathbbm{M(\xi)}U$ can be written as
\be\label{lfae1}\left\{\ball
\lambda f&=Lf-i(v\cdot\xi)f-i(v\cdot\xi)|\xi|^{-2}(P_df-\rho\chi_0)+v\chi_0\cdot u,
\\
\lambda \rho&=-i\xi\cdot u,
\\
\lambda u&=i\xi|\xi|^{-2}\tau_af+\tau_bf-i\xi (\gamma+|\xi|^{-2})\rho-(1+|\xi|^2)u.
\eall\right.\ee
By macro-micro decomposition $f=f_0+f_1:=P_df+P_rf$ with $P_df =C_0\chi_0$ and $P_rf=(I-P_d)f$, $\eqref{lfae1}_1$ can be divided into
\be\label{lfae2}
\left\{\bln
\lambda f_0&=-iP_d(v\cdot\xi)(f_0+f_1)=-iP_d(v\cdot\xi)f_1,
\\
\lambda f_1&=Lf_1-iP_r(v\cdot\xi)(f_0+f_1)-i(v\cdot\xi)|\xi|^{-2} (f_0-\rho\chi_0)+v\chi_0\cdot u.
\eln\right.\ee
Thus, the microscopic part $f_1$ can be represented by
\bmas
f_1&=R(\lambda,\xi)[i (v\cdot\xi)f_0+i(v\cdot\xi)|\xi|^{-2} (f_0-\rho\chi_0)-v\chi_0\cdot u]
\\
&=[i\xi(1+|\xi|^{-2}) C_0-i\xi|\xi|^{-2}\rho-u]\cdot R(\lambda,\xi)v\chi_0,
\emas
where
$$
R(\lambda,\xi)=\left[L-\lambda-iP_r(v\cdot\xi)\right]^{-1},\quad \re\lambda>-1.
$$
Substituting $f_1$ into $\eqref{lfae1}_1$ and $\eqref{lfae1}_3$, we obtain
\be\label{lfae8}
\left\{\ball
\lambda C_0& =[\xi(1+|\xi|^{-2})C_0-\xi|\xi|^{-2}\rho+iu]\cdot(R(\lambda,\xi)v\chi_0,(v\cdot\xi)\chi_0) ,
\\
\lambda \rho&=-i\xi\cdot u,
\\
\lambda u&=i|\xi|^{-2}\[\xi+(1+|\xi|^{2})(R(\lambda,\xi)(v\cdot\xi)\chi_0,v\chi_0)\]C_0
\\
&\quad-i|\xi|^{-2}\[\xi (1+\gamma|\xi|^{2})+(R(\lambda,\xi)(v\cdot\xi)\chi_0,v\chi_0)\]\rho
\\
&\quad-\[1+|\xi|^2+(R(\lambda,\xi)v\chi_0,v\chi_0)\]u.
\eall\right.\ee
In order to simplify the eigenvalue problem, we introduce the following lemma.

\begin{lem}\label{srlem10}
Let $e_1=(1,0,0)$, $\xi=s\omega$ with $s\in\mathbb{R}$, $\omega\in\S ^2$. Then, we have the following results for $\re\lambda>-1$ and $1\le i,j\le3$ that
$$
(R(\lambda,\xi)\chi_i,\chi_j)=\omega_i\omega_j\left(R(\lambda,se_1)\chi_1,\chi_1\right)+(\delta_{ij}-\omega_i\omega_j)\left(R(\lambda,se_1)\chi_2,\chi_2\right).
$$
\end{lem}

\begin{proof}
Let $\mathbb{O}$ be a rotation in $\R^3$ satisfying $\mathbb{O}^T\xi=(s,0,0)$. By changing the variable $v\to \mathbb{O}v$  and using the rotational invariance of the operator $L$, we can prove the lemma.
\end{proof}
For convenience of notation, denote
\be
R_{ij}=R_{ij}(\lambda,s)=(R(\lambda,se_1)\chi_i,\chi_j), \quad i,j=1,2,3. \label{rij}
\ee
With the help of Lemma \ref{srlem10}, the eigenvalue problem \eqref{lfae1} can be simplified as
\be\label{lfae3}
\left\{\bln
\lambda C_0&= (1+s^2)R_{11}C_0-R_{11}\rho+isR_{11}(\omega\cdot u),
\\
\lambda \rho&= -is(\omega\cdot u),
\\
\lambda u&= i\omega\frac{1}{s}[1+(1+s^2)R_{11}]C_0-i\omega\frac{1}{s}\(1+\gamma s^2+R_{11}\)\rho&
\\
&\quad-R_{11}\omega(\omega\cdot u)-R_{22}[u-\omega(\omega\cdot u)]-(1+s^2)u.
\eln\right.
\ee

\begin{remark}\label{eigen1} Let $U=( f,\rho,u)$ be the  eigenvector of $\mathbbm{M}(\xi)$ corresponding the eigenvalue $\lambda$. Then, the eigenvalue problem $\lambda U=\mathbbm{M}(\xi)U$ is equivalent to
\be
\lambda P_AU=G_3(\xi)P_AU+G_5(\xi)(\lambda -G_4(\xi))^{-1}G_5(\xi) P_AU,\quad\re\lambda>-1,
\ee
and then it is equivalent to \eqref{lfae3} with $P_AU=(C_0\chi_0,\rho,u)$.
\end{remark}

Taking the inner product on the both sides of $\eqref{lfae3}_3$ with $\omega$, we have
\be
\lambda (\omega\cdot u)=i\frac{1}{s}[1+(1+s^2)R_{11}]C_0-i\frac{1}{s}(1+\gamma s^2+R_{11})\rho-(1+s^2+R_{11})(\omega\cdot u). \label{div}
\ee
Taking the rotation $\omega\times$ to $\eqref{lfae3}_3$, we obtain
\be\label{curl}
\lambda(\omega\times u)=-(1+s^2)(\omega\times u)-R_{22}(\omega\times u).
\ee
Let $U=(C_0,\rho,(\omega\cdot u))$. We reduce the 5-dimensional system to a 3-dimensional system for $U$ and a 2-dimensional system for $\omega\times u$. The system for $U$ can be written as $\mathbbm{A}(s)U=\lambda U$ with the matrix $\mathbbm{A}(s)$ defined by
\be \label{As}
\mathbbm{A}(s)=
\begin{pmatrix}
(1+s^2)R_{11}&-R_{11}&isR_{11}\\
0&0&-is\\
i\frac{1}{s}[1+(1+s^2)R_{11}]&-i\frac{1}{s}(1+\gamma s^2+R_{11})&-(1+s^2+R_{11})
\end{pmatrix}.
\ee

Denote
\be
D_0(\lambda,s)=\lambda+1+s^2+R_{22},\quad D_1(\lambda,s)=\det(\lambda I_3-\mathbbm{A}(s)). \label{D01}
\ee
For any $a>0$, $b\in\mathbb{C}$,  define
$$
B_a(b)=\{\lambda\in\mathbb{C}\mid|\lambda-b|\le a\}.
$$
We have the following results for the solutions of $D_j(\lambda,s)=0$, $j=0,1$.

\begin{lem}\label{srlem11}
There are two small constants $r_0,r_1>0$  such that the equation $D_0(\lambda,s)=0$ with $\re\lambda\ge -1/2$ has a unique $C^{\infty}$ solution $\lambda=\lambda_0(s)$ for $(s,\lambda )\in[-r_0,r_0]\times B_{r_1}(0)$ satisfying
\be
\lambda_0(0)=0,\quad \lambda_0'(0)=0,\quad \lambda_0''(0)=-\frac{3}{2}.
\ee
\end{lem}

\begin{proof}
Since $(L-\lambda)\chi_k=-(1+\lambda)\chi_k$, $k=1,2,3$, we have
$$
(L-\lambda)^{-1}\chi_k=-(\lambda+1)^{-1}\chi_k.
$$
Thus, the equation
$$
D_0(\lambda,0)=\lambda+1+((L-\lambda)^{-1}\chi_2,\chi_2) =\frac{1}{\lambda+1}\lambda(\lambda+2)=0
$$
has two roots $\lambda_0=0$ and $\lambda_1=-2$. By a direct calculation, we have
\bmas
\partial_\lambda^kR(\lambda,se_1)&=k!R(\lambda,se_1)^{k+1},
\\\partial_s^kR(\lambda,se_1)&=i^{k}k![R(\lambda,se_1)P_rv_1]^kR(\lambda,se_1),
\emas
and then
\bmas
\partial_\lambda D_0(0,0)&=1+( L^{-1}\chi_2,L^{-1}\chi_2)=2\neq0,
\\\partial_sD_0(0,0)&=i(v_1L^{-1}\chi_2,L^{-1}\chi_2)=i(v_1\chi_2,\chi_2)=0.
\emas
By  the implicit function theorem, there exist two constants $r_0,r_1>0$ and a  $C^{\infty}$ function $ \lambda_0(s):[-r_0,r_0]\to B_{r_1}(0) $ such that $D_0(\lambda_0(s),s)=0$. Moreover,
$$
\lambda_0(0)=0,\quad \lambda_0'(0)=-\frac{\partial_sD_0(0,0)}{\partial_\lambda D_0(0,0)}=0.
$$
Since $L (v_1\chi_2)=-2 v_1\chi_2$, we have $L^{-1} (v_1\chi_2)=-\frac{1}{2} v_1\chi_2$. By direct calculation, we obtain
$$
\partial^2_sD_0(0,0)=2-2([L^{-1}P_rv_1]^2L^{-1}\chi_2,\chi_2)=2-2(L^{-1}(v_1\chi_2),v_1\chi_2)=3.
$$
This together with $ \lambda_0'(0)=0$ implies that
$$
\lambda''_0(0)=-\frac{\partial^2_sD_0(0,0)}{\partial_\lambda D_0(0,0)}=-\frac{3}{2}.
$$
The proof of the lemma is completed.
\end{proof}

\begin{lem}\label{weieig}
There are two small constants $r_0,r_1>0$  such that the equation $D_1(\lambda,s)=0$ with $\re\lambda\ge -1/2$ has two solutions $\lambda_{\pm1}$ for $(s,\lambda)\in[-r_0,r_0]\times B_{r_1}(0)$. Moreover, $\lambda_{\pm1}=\lambda_{\pm1}(s)$ are analytic with respect to $s\in[-r_0,r_0]$ with the following expansions
\be\label{fugonge}
\lambda_{\pm1}(s)=\pm i\sqrt{\frac{\gamma+1}{2}}s-\frac{1}{2}s^2+O(s^3).
\ee
Moreover, the eigenvalues $\lambda_{\pm1}(s)$ satisfy
\be
\overline{\lambda_{j}(s)}=\lambda_{-j}(s),\quad j=\pm1.
\ee
\end{lem}
\begin{proof}From \eqref{D01},
$$
D_1(\lambda,s) =\lambda^3+\lambda^2+(1+R_{11})\lambda+s^2[\lambda^2+\gamma\lambda-R_{11}(\lambda^2+2\lambda+1+\gamma)]-s^4R_{11}(\lambda+\gamma).
$$
Thus, the equation
$$
D_1(\lambda,0)=\lambda^3+\lambda^2+\[1+((L-\lambda)^{-1}\chi_1,\chi_1)\]\lambda =\frac{\lambda^2}{\lambda+1}(\lambda^2+2\lambda+2)=0
$$
has four roots $\lambda_{\pm1}=0$ and $\lambda_j=-1+(-1)^ji$ for $j=3,4$. For $\re\lambda>-1$, $R(\lambda,se_1)=[L-\lambda-isP_rv_1]^{-1}$ is analytic in $(\lambda,s) $, and then $D_1(\lambda,s)$ is also analytic in $(\lambda,s)$.  Moreover, by changing variable $v_1\to -v_1$, we have
$$
R_{11}(\lambda,-s)=((L-\lambda+isP_rv_1)^{-1}v_1\chi_0,v_1\chi_0)=R_{11}(\lambda,s),
$$
which implies that $R_{11}(\lambda,s)$ is an even function in $s$, and then $D_1(\lambda,s)$ is also an even function in $s$. We now expand $R_{11}(\lambda,s)$ around $s=0$. By a direct calculation, we have
\bmas
R_{11}(\lambda,0)&=(R(\lambda,0)\chi_1,\chi_1)=((L-\lambda)^{-1}\chi_1,\chi_1)=-\frac{1}{\lambda+1},\\
\partial_s^kR_{11}(\lambda,0)&=(\partial_s^k R(\lambda,0)\chi_1,\chi_1)=i^{k}k!([(L-\lambda)^{-1}P_rv_1]^k(L-\lambda)^{-1}\chi_1,\chi_1).
\emas
Since  $LP_r(v_1\chi_1)=2(\chi_0-v_1\chi_1)=-2P_r(v_1\chi_1)$, we have $(L-\lambda)^{-1}P_r(v_1\chi_1)=-(\lambda+2)^{-1}P_r(v_1\chi_1)$. Thus,
$$
\partial_s^2R_{11}(\lambda,0)=-2!\([(L-\lambda)^{-1}P_rv_1]^2(L-\lambda)^{-1}\chi_1,\chi_1\)=\frac{2(P_r(v_1\chi_1),v_1\chi_1)}{(\lambda+1)^2(\lambda+2)}=\frac{4}{(\lambda+1)^2(\lambda+2)}.
$$
Combing the results above, we obtain
\be
 R_{11}(\lambda,s)=\sum\limits_{k=0}^\infty\frac{1}{(2k)!}\partial_s^{2k}R_{11}(\lambda,0)s^{2k}
 =-\frac{1}{\lambda+1}+\frac{2}{(\lambda+1)^2(\lambda+2)}s^2+O(s^4). \label{sexpand}
\ee
Note that $D_1(0,0)=0$, $D_1(\lambda,0)=\lambda^2h(\lambda)$ with $h(0)\neq 0$, and $D_1(\lambda,s)$ is analytic on $B_{r_1}(0)\times [-r_0,r_0] $. By  Weierstrass Preparation Theorem \ref{wei}, there exists a Weierstrass polynomial $P(\lambda,s)=\lambda^2+p_1(s)\lambda+p_0(s)$ with $p_0(0)=p_1(0)=0$ such that $D_1(\lambda,s)$ can be written as
$$
D_1(\lambda,s)=U(\lambda,s)P(\lambda,s),
$$
where $U(\lambda,s) $ is analytic on $B_{r_1}(0)\times [-r_0,r_0] $ and $U(0,0)\neq 0$. Since $D_1(\lambda,s)$ is even in $s$, it follows that $U(\lambda,s)$, $p_0(s)$ and $p_1(s)$ are even in $s$. Set
$$U(\lambda,s)=\sum\limits_{k=0}^\infty U_k(\lambda)s^{2k}, \quad p_j(s)= \sum\limits_{k=1}^\infty p_{jk}s^{2k},~~j=0,1.$$
Substituting \eqref{sexpand} into $D_1(\lambda,s)$, we have
\bma
D_1(\lambda,s) &\nnm=\frac{\lambda^2}{\lambda+1}(\lambda^2+2\lambda+2)+\[\lambda^2+(\gamma+1)\lambda+1+\frac{\gamma}{\lambda+1}+\frac{2\lambda}{(\lambda+1)^2(\lambda+2)}\]s^2+O(s^4).
\ema
By the method of comparing coefficient, we have the following equations
\be
\left\{\ball
O(1)&:~\lambda^2U_0(\lambda)=\frac{\lambda^2}{\lambda+1}(\lambda^2+2\lambda+2),
\\ O(s^2)&:~\lambda^2U_1(\lambda)+ (\lambda p_{11}+p_{01})U_0(\lambda)
\\&\quad=\lambda^2+(\gamma+1)\lambda+1+\frac{\gamma}{\lambda+1}+\frac{2\lambda}{(\lambda+1)^2(\lambda+2)}.
\eall\right.
\ee
Thus, we obtain
$$
p_{01}=\frac{\gamma+1}{2},\quad p_{11}=1.
$$
Since $U(\lambda,s)\neq 0$ on $[-r_0,r_0]\times B_{r_1}(0)$, then $D_1(\lambda,s)=U(\lambda,s)P(\lambda,s)=0$ is equivalent to $P(\lambda,s)=0$. Note that
$$
\Delta_\lambda=p_1(s)^2-4p_0(s)=-2(\gamma+1)s^2+O(s^4)\le0,~~s\le r_0.
$$
Thus, we obtain the expansions of the solutions $\lambda_{\pm1}(s)$ to $P(\lambda,s)$ as
$$
\lambda_{\pm1}(s)=-\frac{p_1(s)}2\pm \frac{\sqrt{\Delta_\lambda} }2=\pm i\sqrt{\frac{\gamma+1}{2}}s-\frac{1}{2}s^2+O(s^3),
$$
and hence \eqref{fugonge} holds. The proof of the lemma is completed.
\end{proof}

\begin{thm}\label{srlem13}
There exists a constant $r_0>0$ such that $\mathbbm{M}(\xi)$ has the following spectrum structure for $s\le r_0$:
$$
\sigma(\mathbbm{M}(\xi))\cap \{ \lambda\in\mathbb{C}\mid \re\lambda\ge -1/2 \}=\{\lambda_j(s), ~j=-1,0,1,2\}.
$$
In particular, the eigenvalues $\lambda_j(s)~(j=-1,0,1,2)$ are analytic with respect to $s$ and admit the following asymptotic expansions for $s\le r_0$:
\be\label{egexpand}
\left\{\ball
\lambda_{\pm1}(s)&=\pm i\sqrt{\frac{\gamma+1}{2}}s-\frac{1}{2}s^2+O(s^3),
\\
\lambda_0(s)&=\lambda_2(s)=-\frac{3}{4}s^2+O(s^3).
\eall\right.\ee
The eigenvectors $\Psi_j(s,\omega)=(\psi_j, \zeta_j, \vartheta_j)(s,\omega)$ $(-1\le j\le 2)$ are orthogonal to each other satisfying
$$\left\{\ball
(\Psi_i ,\overline{\Psi_j })_{\xi,\gamma}&=(\psi_i,\overline{\psi_j})+\gamma(\zeta_i,\overline{\zeta_j})+(\vartheta_i,\overline{\vartheta_j})
+\frac{1}{|\xi|^2}(\tau_a\psi_i-\zeta_i,\overline{\tau_a\psi_j-\zeta_j})
\\
&=\delta_{ij},~~-1\le i, j\le2,
\\
\Psi_j(s,\omega)&=\Psi_{j,0}(\omega)+\Psi_{j,1}(\omega)s+O(s^2),~~|s|\le r_0,
\eall\right.$$
where $\overline{\Psi_j}=(\overline{\psi_j}, \overline{\zeta_j}, \overline{\vartheta_j})$, and the coefficients $\Psi_{j,n}=(\psi_{j,n}, \zeta_{j,n}, \vartheta_{j,n})$ satisfy for $j=\pm1$,
\be\label{lfae4}\left\{\ball
&\psi_{\pm1,0}=\frac{1}{\sqrt{2(\gamma+1)}}\chi_0\mp\frac{1}{2}(v\cdot\omega)\chi_0,
\\
&\psi_{\pm1,1}=\mp i\frac{1}{4(\gamma+1)}\chi_0-i\frac{1}{4\sqrt{2(\gamma+1)}}(v\cdot\omega)\chi_0\pm i\frac{1}{4}P_r(v\cdot\omega)^2\chi_0,
\\
&\zeta_{\pm1,0}=\frac{1}{\sqrt{2(\gamma+1)}},~~\zeta_{\pm1,1}=\pm i\frac{1}{4(\gamma+1)},~~\vartheta_{\pm1,0}=\mp\frac{1}{2}\omega,~~\vartheta_{\pm1,1}=-i\frac{1}{4\sqrt{2(\gamma+1)}}\omega,
\\
&\tau_a\psi_{\pm1,n}-\zeta_{\pm1,n}=0~(n=0,1),~~\tau_a\psi_{\pm1,2}-\zeta_{\pm1,2}=\frac{\gamma-1}{2\sqrt{2(\gamma+1)}},
\eall\right.\ee
and for $j=0,2$,
\be\label{lfae4a}\left\{\ball
&\psi_{j,0}=\frac{\sqrt{2}}{2}(v\cdot W^j)\chi_0,~~\psi_{j,1}=-i\frac{\sqrt{2}}{4}P_r(v\cdot\omega)(v\cdot W^j)\chi_0,
\\
&\zeta_{j,n}=0~(n\ge0),~~\vartheta_{j,0}=\frac{\sqrt{2}}{2}W^j,~~\vartheta_{j,1}=0, \\
&\tau_a\psi_{j,n}-\zeta_{j,n}=0~(n\ge0).
\eall\right.\ee
Here, $W^j~(j=0,2)$ are normal vectors satisfying $W^0\cdot W^2=0,~W^0\cdot\omega=W^2\cdot\omega=0$.
\end{thm}

\begin{proof}
The eigenvalues $\lambda_j(s)$ and the eigenvectors $ \Psi_j(s,\omega)=(\psi_j, \zeta_j, \vartheta_j)(s,\omega)~(-1\le j\le2)$ can be constructed as follows. For $j=0,2$, we take $\lambda_j=\lambda_0(s)$ to be the solution of the equation $D_0(\lambda,s)=0$, and choose $C_0=\zeta=(\omega\cdot \vartheta)=0$. The corresponding eigenvectors $\Psi_j(s,\omega)=(\psi_j, \zeta_j, \vartheta_j)(s,\omega)$ are defined by
$$\left\{\ball
&\psi_j =-c_j(s)[L-\lambda_j-isP_r(v\cdot\omega)]^{-1}(v\cdot W^j)\chi_0,
\\
&\zeta_j=0,\quad \vartheta_j=c_j(s)W^j(\omega),
\eall\right.$$
where $W^j~(j=0,2)$ to be the normal vectors satisfying $W^0\cdot W^2=0$ and $W^0\cdot\omega=W^2\cdot\omega=0$.

For $j=\pm1$, we take $\lambda_j=\lambda_j(s)$ to be the solutions of the equation $D_1(\lambda,s)=0$, and choose $\omega\times \vartheta =0$. We denote by $V_j=(a_j,a_j-sb_j, c_j) $ the solutions of $\mathbbm{A}(s)V_j=\lambda_j(s)V_j$. The corresponding eigenvectors $\Psi_j(s,\omega)=(\psi_j, \zeta_j, \vartheta_j)(s,\omega)$ are defined by
$$\left\{\ball
&\psi_j=a_j(s)\chi_0+\left[isa_j(s)+i b_j(s)-c_j(s) \right][L-\lambda_j-isP_r(v\cdot\omega)]^{-1}(v\cdot\omega)\chi_0,
\\&\zeta_j=a_j-sb_j,\quad \vartheta_j=c_j(s)\omega.
\eall\right.$$

We write
$$
\mathbbm{M}(\xi)\Psi_k(s,\omega)=\lambda_k(s)\Psi_k(s,\omega),\quad -1\le k\le2.
$$
Taking the inner product $(\cdot,\cdot)_{\xi,\gamma}$ of the equation above with $\overline{\Psi_j(s,\omega)}$ and using the facts that
\bmas
(\M(\xi)U,V)_{\xi,\gamma}&=(U,\M(-\xi)V)_{\xi,\gamma},\quad U,V\in Z_{\xi,\gamma},
\\
\M(-\xi)\overline{\Psi_j(s,\omega)}&=\overline{\lambda_j(s)}\cdot\overline{\Psi_j(s,\omega)},\quad -1\le j\le2,
\emas
we have
$$
(\lambda_j(s)-\lambda_k(s))(\Psi_j(s,\omega),\overline{\Psi_k(s,\omega)})_{\xi,\gamma}=0,~-1\le j\neq k\le3.
$$
For $s\neq0$ sufficiently small, $\lambda_k(s)\neq\lambda_j(s)$ for $-1\le j\neq k\le1$. Therefore, we have
$$
(\Psi_j(s,\omega),\overline{\Psi_k(s,\omega)})_{\xi,\gamma}=0,~-1\le j\neq k\le2.
$$
We can normalized them by taking $(\Psi_j(s,\omega),\overline{\Psi_j(s,\omega)})_{\xi,\gamma}=1$ for $-1\le j\le2$.

For $j=0,2$, the coefficient $c_0(s)=c_2(s)$ can be determined by the normalization condition
\be\label{02norm}
1=(\Psi_j,\overline{\Psi_j})_{\xi,\gamma}=c_j(s)^2(1+A_{22}),\quad j=0,2,
\ee
where $A_{22}=(R(\lambda_0,se_1)\chi_2,R(\overline{\lambda_0},-se_1)\chi_2) $. Similar to the treatment of  \eqref{sexpand}, we have by \eqref{egexpand} that
\bma
A_{22}=(R(0,0)\chi_2,R(0,0)\chi_2)+2(\partial_s R(0,0)\chi_2,R(0,0)\chi_2)s+O(\lambda_0(s))+O(s^2)= 1+O(s^2). \label{A22}
\ema
Substituting \eqref{A22} into \eqref{02norm}, we obtain
\be \label{cj}
c_j(s)=\frac{\sqrt{2}}{2}+O(s^2),\quad j=0,2.
\ee
By \eqref{cj} and using
\bmas
&[L-\lambda_j-isP_r(v\cdot\omega)]^{-1}(v\cdot W^j)\chi_0\\
&=-(v\cdot W^j)\chi_0+\frac12isP_r(v\cdot\omega)(v\cdot W^j)\chi_0+O(s^2),\quad j=0,2,
\emas
we can obtain the expansions of $\psi_{0,k}=\psi_{2,k}~(k=0,1)$ as listed in \eqref{lfae4a}.

From  \eqref{As}, we obtain the system for $(a_j,b_j,c_j)$:
\be\label{lfae5}\left\{\ball
&\lambda_ja_j-s^2R_{11}(\lambda_j,s)a_j-isR_{11}(\lambda_j,s)c_j-sR_{11}(\lambda_j,s) b_j=0,
\\
&\lambda_ja_j-s\lambda_jb_j+isc_j=0,
\\
&-isR_{11}(\lambda_j,s)a_j+is\gamma (a_j-sb_j)+[\lambda_j +1+s^2 +R_{11}(\lambda_j,s)]c_j
\\
&\quad-[i +i R_{11}(\lambda_j,s)] b_j =0.
\eall\right.\ee
Moreover, we have the normalization condition
\be\label{gyhtj}
1=(\Psi_j,\overline{\Psi_j})_{\xi,\gamma}=a_j^2+(isa_j -c_j +i b_j )^2A_{jj}+\gamma (a_j-sb_j)^2+c_j^2+ b_j^2,
\ee
where $A_{jj}=(R(\lambda_j,se_1)\chi_1,R(\overline{\lambda_j},-se_1)\chi_1)$. By 
\eqref{egexpand}, we have
\be
R_{11}(\lambda_j,s)=-1+ji\sqrt{\frac{\gamma+1}{2}}s+O(s^2), \quad
A_{jj} =1-2ji\sqrt{\frac{\gamma+1}{2}}s+O(s^2). \label{Ajj}\ee
For $j=\pm1$, we expand $a_j(s),b_j(s),c_j(s)$ as
$$
a_j(s)=\sum\limits_{k=0}^1a_{j,k}s^k+O(s^2),~~b_j(s)=\sum\limits_{k=0}^1b_{j,k}s^k+O(s^2),~~c_j(s)=\sum\limits_{k=0}^{1}c_{j,k}s^k+O(s^2).
$$
Substituting \eqref{egexpand} and \eqref{Ajj} into $\eqref{lfae5}_{1,2}$ and \eqref{gyhtj}, we obtain the following equations about $\{a_{j,k},b_{j,k},c_{j,k}\}$
\bmas
O(1):&~~
a_{j,0}^2+(i b_{j,0}-c_{j,0})^2+\gamma a_{j,0}^2+c_{j,0}^2+b_{j,0}^2=1,
\\
O(s):&~~
\left\{\ball
&b_{j,0}=-ji\sqrt{\frac{\gamma+1}{2}}a_{j,0}-ic_{j,0},
\\
&c_{j,0}=-j\sqrt{\frac{\gamma+1}{2}}a_{j,0},
\\
&2a_{j,0}a_{j,1}+2(i b_{j,0}-c_{j,0})(ia_{j,0}-c_{j,1}+ib_{j,0})
\\
&-2ji\sqrt{\frac{\gamma+1}{2}}(i b_{j,0}-c_{j,0})^2+2\gamma b_{j,0}b_{j,1}+2c_{j,0}c_{j,1}=0,
\eall\right.
\\
O(s^2):&~~
\left\{\ball
&b_{j,1}=-\frac12a_{j,0}-ji\sqrt{\frac{\gamma+1}{2}}a_{j,1}-j\sqrt{\frac{\gamma+1}{2}}c_{j,0}-ic_{j,1}+ji\sqrt{\frac{\gamma+1}{2}}b_{j,0},
\\
&c_{j,1}=-i\frac{1}{2}a_{j,0}-j\sqrt{\frac{\gamma+1}{2}}a_{j,1},
\eall\right.
\emas
which lead to
\be \label{abc}
\left\{\ball
&a_{j,0}=\frac{1}{\sqrt{2(\gamma+1)}},\quad b_{j,0}=0,\quad c_{j,0}=-j\frac{1}{2},
\\
&a_{j,1} =-ji\frac{1}{4(\gamma+1)},\quad
b_{j,1}=\frac{\gamma-1}{2\sqrt{2(\gamma+1)}}, \quad c_{j,1}=-i\frac{1}{4\sqrt{2(\gamma+1)}}.
\eall\right.\ee
By \eqref{abc} and using
\bmas
&[L-\lambda_j-isP_r(v\cdot\omega)]^{-1}(v\cdot\omega)\chi_0
\\&=-(v\cdot\omega)\chi_0+s\bigg[ji\sqrt{\frac{\gamma+1}{2}}(v\cdot\omega)\chi_0+i\frac12P_r(v\cdot\omega)^2\chi_0\bigg]+O(s^2),\quad j=\pm1,
\emas
we can obtain the expansions of $\psi_{\pm1,k}~(k=0,1)$ as listed in \eqref{lfae4}. The proof  is completed.
\end{proof}
\begin{remark}\label{tzznx}
In fact, for the general viscous term $\nu\Delta_x u$ with $\nu>0$ in the equation $\eqref{introe2}_3$, the asymptotic expansions of the eigenvalues $\{\lambda_j(s),~j=-1,0,1,2\}$ become
\be\label{nxtzz}\left\{\ball
\lambda_{\pm1}(s)&=\pm i\sqrt{\frac{\gamma+1}{2}}s-\frac{1+\nu}{4}s^2+O(s^3),
\\
\lambda_0(s)&=\lambda_2(s)=-\frac{1+2\nu}4s^2+O(s^3),
\eall\right.
\ee
which can be obtained by an argument similar to the proof of Lemma \ref{weieig}.
\end{remark}
\section{Optimal Time-Decay Rates of Linearized VPFP/NSP}
In the section, we consider the Cauchy problem \eqref{introe4} for the linearized VPFP/NSP system and establish the optimal time decay rates of the global solution based on the results obtained in Section 2.

\subsection{Decomposition and Asymptotics of Linear Semigroup}
We start by proving the following lemmas.
\begin{lem}\label{orlem8}
The operator $G_3(\xi)$ defined by \eqref{g3g4} generates a strongly continuous contraction semigroup on $ N_0\times \C \times \C^3$ for any fixed $\xi\ne 0$, which satisfies for any $t>0$ and $U\in N_0\times \C \times \C^3$ that
\be\label{da17}
\|e^{tG_3(\xi)}U\|_{\xi,\gamma}\le \|U\|_{\xi,\gamma}.
\ee
In addition, for any $\lambda=x+iy\neq0$ with $x\in(-\frac{1}{2},0)$ and $U\in N_0\times \C \times \C^3$, it holds that
\be\label{da18}
\int_{-\infty}^{+\infty}\|[(x+iy)-G_3(\xi)]^{-1}U\|_{\xi,\gamma}^2dy\le C\|U\|_{\xi,\gamma}^2.
\ee
\end{lem}
\begin{proof}
Note that $G_3(\xi)$ is a densely defined closed operator on $N_0\times \C \times \C^3$,  and both $G_3(\xi)$ and its adjoint operator $G_3(\xi)^*=G_3(-\xi)$ are dissipative on $N_0\times \C \times \C^3$ satisfying
$$
\re(G_3(\xi)U,U)_{\xi,\gamma}=\re(G_3(\xi)^*U,U)_{\xi,\gamma}\le0.
$$
Thus,  $G_3(\xi)$ generates a strongly continuous contraction semigroup on $N_0\times \C \times \C^3$ satisfying \eqref{da17}. By  Lemma \ref{srlem8}, we have for $U\in N_0\times \C \times \C^3$,
\begin{align*}
&\int_{-\infty}^{\infty}\|(\lambda -G_3(\xi))^{-1} U\|_{\xi,\gamma}^2dy=\int_{-\infty}^{\infty}\sum\limits_{j=-1}^3\left\|\frac{1}{\lambda-\tilde{\eta}_j}(U,\tilde{\beta}_j)_{\xi,\gamma}^2\tilde{\beta}_j\right\|_{\xi,\gamma}^2dy
\\
&\le\|U\|_{\xi,\gamma}^2\max_{-1\le j\le 3}\int_{-\infty}^{\infty}\frac{1}{|\lambda-\tilde{\eta}_j|^2}dy=\|U\|_{\xi,\gamma}^2\max_{-1\le j\le 3}\int_{-\infty}^{\infty}\frac{1}{\left(x-\re\tilde{\eta}_j\right)^2+\left(y-\im\tilde{\eta}_j\right)^2}dy
\\
&=\|U\|_{\xi,\gamma}^2\max_{-1\le j\le 3}\frac{\pi}{|x-\re\tilde{\eta}_j|}\le C\|U\|_{\xi,\gamma}^2.
\end{align*}
The proof of the lemma is completed.
\end{proof}

\begin{lem}\label{orlem1}
The operator $G_4(\xi)$ defined by \eqref{g3g4} generates a strongly continuous contraction semigroup on $N^{\perp}_0\times\{0\}\times\{0\}$ for any fixed $\xi\in\mathbb{R}^3$, which satisfies for any $t>0$ and $U=(f,0,0)$ with $f\in N^{\perp}_0 $ that
\be\label{da1}
\|e^{tG_4(\xi)}U\|\le e^{-t}\|U\|.
\ee
In addition, for any $\lambda=x+iy $ with $x>-1$ and $U=(f,0,0)$ with $f\in N^{\perp}_0 $, it holds that
\be\label{da2}
\int_{-\infty}^{+\infty}\|[(x+iy)-G_4(\xi)]^{-1}U\|^2dy\le\pi(x+1)^{-1}\|U\|^2.
\ee
\end{lem}
\begin{proof}
Note that for $U=(f,0,0)\in N^{\perp}_0\times\{0\}\times\{0\}$,
$$
e^{tG_4(\xi)}U=(e^{tQ(\xi)}f,0,0)^T.
$$
Note that $Q(\xi)=L-iP_r(v\cdot\xi)P_r$ is a densely defined closed operator on $N^{\perp}_0$, and both $Q(\xi)$ and its adjoint operator $Q(\xi)^*=Q_1(-\xi)$ are dissipative on $N^{\perp}_0$ satisfying
$$
\re(Q(\xi)f,f)=\re(Q(\xi)^*U,U)=\re(Lf,f)\le-\|f\|^2,\quad f\in N_0^\perp.
$$
Thus,  $G_4(\xi)$ generates a strongly continuous contraction semigroup on $N^{\perp}_0$ satisfying \eqref{da1}. Moreover, we have for $U=(f,0,0)\in N^{\perp}_0\times\{0\}\times\{0\}$ that
\bmas
[\lambda-G_4(\xi)]^{-1}U&=\int_0^\infty e^{-\lambda t}(e^{tQ(\xi)}f,0,0)^Tdt
\\
&=\frac1{\sqrt{2\pi}}\int_{-\infty}^{+\infty} e^{-iyt}\(\sqrt{2\pi}1_{t\ge0}e^{-x t}e^{tG_4(\xi)}U\)dt,\quad x>-1 ,
\emas
where the right hand side is the Fourier transform of the function $\sqrt{2\pi}1_{t\ge0}e^{-x t}e^{tG_4(\xi)}U$ with respect to $t$. By using the Plancherel's equality, we have for $U=(f,0,0)$ with $f\in N^{\perp}_0\cap L^2(\mathbb{R}^3_v)$ that
\bmas
&\int_{-\infty}^{+\infty} \|[(x+iy)-G_4(\xi)]^{-1}U\|^2dy=\int_{-\infty}^{+\infty}\|\sqrt{2\pi}1_{t\ge0}e^{-x t}e^{tG_4(\xi)}U\|^2dy
\\
&=2\pi\int_{-\infty}^{+\infty}e^{-2xt}\|e^{tG_4(\xi)}U\|^2dt\le2\pi\int_0^\infty e^{-2(x+1)t}dt\|U\|^2\le\pi(x+1)^{-1}\|U\|^2,
\emas
which proves \eqref{da2}. The proof of the lemma is completed.
\end{proof}

\begin{lem}\label{orlem2}The operator $G_1(\xi)$ defined by \eqref{sra2} generates a strongly continuous contraction semigroup on $L^2(\R^3_v)\times\C \times\C^3$  for any fixed $|\xi|\ge r_0 >0$, which satisfies for any  $t>0$ and $U\in L^2(\R^3_v)\times\C \times\C^3$ that
\be\label{da20}
\|e^{tG_1(\xi)}U\|\le Ce^{-\frac12t}\|U\|,
\ee
where $C>1$ is a constant depending on $r_0$. For any $x>-\frac12$ and $U\in L^2(\R^3_v)\times\C \times\C^3$, we obtain
\be\label{dat21}
\int_{-\infty}^{+\infty}\|[(x+iy)-G_1(\xi)]^{-1}U\|^2dy\le C\|U\|^2.
\ee
\end{lem}
\begin{proof}
Note that for $U=(f,X)\in L^2(\R^3_v)\times\C^4$,
$$
e^{tG_1(\xi)}U=(e^{tA(\xi)}f,e^{tB_1(\xi)}X)^T.
$$
The operator $A(\xi)=A-i(v\cdot\xi)$ generates a strongly continuous contraction semigroup on $L^2(\mathbb{R}^3_v)$, which satisfies for any  $t>0$ and $f\in L^2(\mathbb{R}^3_v)$ that
\be\label{da3}
\|e^{tA(\xi)}f\|\le e^{-\frac32t}\|f\|.
\ee
In addition, for any $x>-\frac32$ and $f\in L^2(\R^3_v)$, we can prove by the same way as \eqref{da2} that
\be\label{da4}
\int_{-\infty}^{+\infty}\|[(x+iy)-A(\xi)]^{-1}f\|^2dy\le\pi\(x+\frac32\)^{-1}\|f\|^2.
\ee
By  \eqref{sraa16}, we obtain for any $X\in  \C^4$,
\be\ball\label{dat17}
\|e^{tB_1(\xi)}X\|^2 &\le\|e^{tB_1(\xi)}X\|^2_{\mathcal{C}_{\xi,\gamma}}=\sum\limits_{j=-1}^2\left\|e^{\eta_jt}(X,\beta_j)_{\mathcal{C}_{\xi,\gamma}}\beta_j\right\|^2_{\mathcal{C}_{\xi,\gamma}}
\\
&\le Ce^{- t}\|X\|^2_{\mathcal{C}_{\xi,\gamma}}\le Ce^{- t}\|X\|^2 ,\quad C>1.
\eall\ee
By  \eqref{sraa2}, we obtain for $\lambda=x+iy$ with $x>-\frac12$ that
\bma
&\nnm\int_{-\infty}^{\infty}\|(\lambda-B_1(\xi))^{-1}X\|^2dy\le\int_{-\infty}^{\infty}\sum\limits_{j=-1}^2\left\|\frac{1}{\lambda-\eta_j}(X,\beta_j)_{\mathcal{C}_{\xi,\gamma}}\beta_j\right\|_{\mathcal{C}_{\xi,\gamma}}^2dy
\\
&\label{da19}\le\|X\|_{\mathcal{C}_{\xi,\gamma}}^2\max_{-1\le j\le 2}\frac{\pi}{|x-\re\eta_j|}\le C\|X\|_{\mathcal{C}_{\xi,\gamma}}^2\le C\|X\|^2.
\ema
By \eqref{da3} and \eqref{dat17}, we obtain
\bmas
\|e^{tG_1(\xi)}U\|&\le\|e^{tA(\xi)}f\|+\|e^{tB_1(\xi)}X\|  \le Ce^{-\frac12t}\|U\|,\quad C>1.
\emas
Combining with \eqref{da4}, \eqref{da19} and
\bmas
&\int_{-\infty}^{+\infty}\|[(x+iy)-G_1(\xi)]^{-1}U\|^2dy\\
&\le\int_{-\infty}^{+\infty}\|[(x+iy)-A(\xi)]^{-1}f\|^2+\|[(x+iy)-B_1(\xi)]^{-1}X\|^2dy,
\emas
we prove \eqref{dat21} for $x>-\frac12$. The proof of the lemma is completed.
\end{proof}

By virtue of Lemmas \ref{orlem8}--\ref{orlem2}, we have the decomposition of the semigroup $S(t,\xi)=e^{t\mathbbm{M}(\xi)}$ generated by $\mathbbm{M}(\xi)$ as follows.

\begin{thm}\label{orlem5}
The semigroup $S(t,\xi)=e^{t\mathbbm{M}(\xi)}$ with $s=|\xi|\neq0$ and $\omega=\xi/|\xi|$ has the following decomposition
\be\label{dat1}
S(t,\xi)U=S_1(t,\xi)U+S_2(t,\xi)U,\quad U\in Z_{\xi,\gamma},~t>0,
\ee
where
\be\label{dat2}
S_1(t,\xi)U=\sum\limits_{j=-1}^2e^{t\lambda_j(|\xi|)}\(U,\overline{\Psi_j(s,\omega)}\)_{\xi,\gamma}\Psi_j(s,\omega)
1_{\{|\xi|\le r_0\}},
\ee
with $(\lambda_j(s),\Psi_j(s,\omega))$ being the eigenvalue and eigenvector of the operator $\mathbbm{M}(\xi)$ given by Theorem \ref{srlem13} for $|\xi|\le r_0$,  and $S_2(t,\xi)=:S(t,\xi)-S_1(t,\xi)$ satisfies for $\kappa_0>0$ independent of $\xi$  that
\be\label{dat3}
\|S_2(t,\xi)U\|_{\xi,\gamma}\le Ce^{-\kappa_0t}\|U\|_{\xi,\gamma},\quad t>0.
\ee
\end{thm}
\begin{proof}
By Lemma \ref{dense}, it is sufficient to prove \eqref{dat1} for $U\in D(\mathbbm{M}(\xi)^2)$ because the domain $D(\mathbbm{M}(\xi)^2)$ is dense in $Z_{\xi,\gamma}$. By Lemma \ref{semigroup-1}, the semigroup $e^{t\mathbbm{M}(\xi)}$ can be represented by
\be\label{dat4}
e^{t\mathbbm{M}(\xi)}U=\frac{1}{2\pi i}\int_{\kappa_1-i\infty}^{\kappa_1+i\infty}e^{\lambda t}(\lambda-\mathbbm{M}(\xi))^{-1}Ud\lambda,\quad U\in D(\mathbbm{M}(\xi)^2),~\kappa_1>0.
\ee
It remains to analyze the resolvent $(\lambda-\mathbbm{M}(\xi))^{-1}$  in order to obtain the decomposition \eqref{dat1} for the semigroup $e^{t\mathbbm{M}(\xi)}$. By \eqref{sraa9}, we rewrite $(\lambda-\mathbbm{M}(\xi))^{-1}$ for $\lambda\in\rho(\mathbbm{M}(\xi))\cap\{\lambda\in\mathbb{C}\mid\re\lambda\ge-\frac12\}$ and $|\xi|\le r_0$ as
\be\label{dat5}
(\lambda-\mathbbm{M}(\xi))^{-1}=[(\lambda -G_3(\xi))^{-1}P_A+(\lambda -G_4(\xi))^{-1}P_B]+Z_1(\lambda,\xi),
\ee
with
\bmas
Z_1(\lambda,\xi)&=[(\lambda -G_3(\xi))^{-1}P_A+(\lambda -G_4(\xi))^{-1}P_B][I-Y_1(\lambda,\xi)]^{-1}Y_1(\lambda,\xi),
\\
Y_1(\lambda,\xi)&=G_5(\xi)[(\lambda -G_4(\xi))^{-1}P_B+(\lambda -G_3(\xi))^{-1}P_A].
\emas
Substituting \eqref{dat5} into \eqref{dat4}, we have the following decomposition of semigroup $e^{t\mathbbm{M}(\xi)}$
\be\label{dat6}
e^{t\mathbbm{M}(\xi)}U = e^{tG_4(\xi)}U + \frac{1}{2\pi i}\int_{\kappa_1 - i\infty}^{\kappa_1 + i\infty} e^{\lambda t}Z_2(\lambda,\xi)U d\lambda,\quad |\xi| \leq r_0,
\ee
with
$$
Z_2(\lambda,\xi)=Z_1(\lambda,\xi)+(\lambda -G_3(\xi))^{-1}P_A.
$$
In order to estimate the last term of the right-hand side of \eqref{dat6}, let us denote
\be\label{dat7}
\mathcal{U}_{\kappa_1,N}U=\frac{1}{2\pi i}\int_{\kappa_1 - iN}^{\kappa_1 + iN} e^{\lambda t}Z_2(\lambda,\xi)U d\lambda,
\ee
where the constant $N>0$ is chosen large enough so that $N>y_2$ with $y_2$ defined in Lemma \ref{srlem9}. Since $Z_2(\lambda,\xi)=(\lambda-\mathbbm{M}(\xi))^{-1}-(\lambda -G_4(\xi))^{-1}P_B$ is analytic on the domain $\re\lambda\ge -\frac12 $ with finite singularities at $\lambda=\lambda_j(s)\in\sigma(\mathbbm{M}(\xi)$ for $-1\le j\le2$, we can shift \eqref{dat7} from the line $\re\lambda=\kappa_1>0$ to $\re\lambda= -\kappa_2\in (-\frac12,-r_0^2) $. Due to the Residue Theorem, we obtain
$$
\mathcal{U}_{\kappa_1,N}U=\mathcal{U}_{-\kappa_2,N}U+H_NU+2\pi i\sum\limits_{j=-1}^2\res\big\{e^{\lambda t}Z_2(\lambda,\xi)U;\lambda_j(|\xi|)\big\}1_{\{|\xi|\le r_0\}},
$$
where $\res\{f(\lambda);\lambda\}$ means the residue of $f$ at $\lambda$ and
$$
H_NU=\frac{1}{2\pi i}\left(\int_{-\kappa_2+iN}^{\kappa_1+iN}-\int_{-\kappa_2-iN}^{\kappa_1-iN}\right)e^{\lambda t}Z_2(\lambda,\xi)U1_{\{|\xi|\le r_0\}}d\lambda.
$$
By Lemma \ref{srlem8} and \eqref{sraa8}, it is easy to verify that for $\lambda=x+iN$,
\bmas
\|H_NU\|_{\xi,\gamma}
&\le C\int_{-\kappa_2}^{\kappa_1}\|Z_2(\lambda,\xi)U\|_{\xi,\gamma}dx
\\
&\le C(1+r_0)^2\|U\|_{\xi,\gamma}\int_{-\kappa_2}^{\kappa_1}e^{xt}\(\max\limits_{-1\le j\le3}\frac{1}{|\lambda-\tilde{\eta}_j(|\xi|)|}+\frac{1}{1+|\lambda|}\)dx
\\
&\le C(1+r_0)^2e^{\kappa_1 t}(1+N)^{-1}\|U\|_{\xi,\gamma}\to0,~~N\to\infty .
\emas
Now, let us estimate $\mathcal{U}_{-\kappa_2,N}U$. By  Lemma \ref{srlem8}, we have
\bma\label{dat9}
\lim\limits_{N\to\infty}\bigg\|\int_{-\kappa_2-iN}^{-\kappa_2+iN}e^{\lambda t}(\lambda -G_3(\xi))^{-1}P_AUd\lambda\bigg\|_{\xi,\gamma}
=\bigg\|-2\pi i\sum\limits_{j=-1}^2e^{\tilde{\eta}_jt}(U,\tilde{\beta}_j)_{\xi,\gamma}\tilde{\beta}_j\bigg\|_{\xi,\gamma}\le Ce^{-\frac12t}\|U\|_{\xi,\gamma}.
\ema
By Lemma \ref{orlem4}, it holds $\sup\limits_{0<|\xi|<r_0,y\in\mathbb{R}}\|[I-Y_1(-\kappa_2+iy,\xi)]^{-1}\|_{\xi,\gamma}\le C$. Thus, we have by \eqref{sraa5} and \eqref{sraa15} that for any $U,V\in Z_{\xi,\gamma}$ and $\lambda=-\kappa_2+iy$,
\bma
&\left|(\mathcal{U}_{-\kappa_2,\infty}(t)U,V)_{\xi,\gamma}\right|\le e^{-\kappa_2t}\int_{-\infty}^{+\infty}|(Z_1(\lambda,\xi)U,V)_{\xi,\gamma}|dy\nnm
\\
&\le Ce^{-\kappa_2t}\int_{-\infty}^{+\infty}\(\| (\lambda -G_4(\xi))^{-1}P_BU\| +\| (\lambda -G_3(\xi))^{-1}P_AU\|_{\xi,\gamma}\)\nnm
\\
&\quad\times\(\|(\overline{\lambda} -G_3(-\xi))^{-1}P_AV\|_{\xi,\gamma}+\|(\overline{\lambda} -G_4(-\xi))^{-1}P_BV\|_{\xi,\gamma}\)dy .\label{dat10}
\ema
By \eqref{da18} and \eqref{da2}, we have $|(\mathcal{U}_{-\kappa_2,\infty}(t)U,V)_{\xi,\gamma}|\le Ce^{-\kappa_2t}\|U\|_{\xi,\gamma}\|V\|_{\xi,\gamma}$, and then
\be\label{dat11}
\|\mathcal{U}_{-\kappa_2,\infty}(t)\|_{\xi,\gamma}\le Ce^{-\kappa_2t}.
\ee
Since $\lambda_j(s)\in\rho(G_4(\xi))$ and $Z_2(\lambda,\xi)=(\lambda-\mathbbm{M}(\xi))^{-1}-(\lambda -G_4(\xi))^{-1}P_B$, we can obtain by a similar argument as Theorem 3.4 in \cite{zhongspvpb} that
\bma
\res\{e^{\lambda t}Z_2(\lambda,\xi)U;\lambda_j(s)\}&=\res\{e^{\lambda t}(\lambda-\mathbbm{M}(\xi))^{-1}U;\lambda_j(s)\}\nnm
\\
&=e^{t\lambda_j(s)}\(U,\overline{\Psi_j(s,\omega)}\)_{\xi,\gamma}\Psi_j(s,\omega),~~s\le r_0.\label{dat12}
\ema
Therefore, we conclude from \eqref{dat6}-\eqref{dat12} that
\be\label{dat13}
e^{t\mathbbm{M}(\xi)}U = e^{tG_4(\xi)}U+\mathcal{U}_{-\kappa_2,\infty}(t)U + \sum\limits_{j=-1}^2e^{t\lambda_j(s)}\(U,\overline{\Psi_j(s,\omega)}\)_{\xi,\gamma}\Psi_j(s,\omega),~~|\xi| \leq r_0.
\ee

For $|\xi|>r_0$, in terms of \eqref{sra9}, we have  for $\lambda\in\rho(\mathbbm{M}(\xi))\cap\{\lambda\in\mathbb{C}\mid\re\lambda\ge-\frac12\}$,
\be\label{dat14}
(\lambda-\mathbbm{M}(\xi))^{-1}=(\lambda-G_1(\xi))^{-1}+Z_3(\lambda,\xi),
\ee
with the operator $Z_3(\lambda,\xi)$ defined by
\bmas
Z_3(\lambda,\xi)&=(\lambda-G_1(\xi))^{-1}[I-Y_2(\lambda,\xi)]^{-1}Y_2(\lambda,\xi),
\\
Y_2(\lambda,\xi)&=:G_2(\xi)(\lambda-G_1(\xi))^{-1}.
\emas
Substituting \eqref{dat14} into \eqref{dat4}, we have the following decomposition of   $e^{t\mathbbm{M}(\xi)}$
\be\label{dat15}
e^{t\mathbbm{M}(\xi)}U =e^{tG_1(\xi)}U+\frac{1}{2\pi i}\int_{\kappa_1 - i\infty}^{\kappa_1 + i\infty} e^{\lambda t}Z_3(\lambda,\xi)U d\lambda, \quad |\xi| >r_0.
\ee
Similarly, in order to estimate the last term of the right-hand side of \eqref{dat15}, let us denote
\be\label{dat16}
\mathcal{V}_{\kappa_1,N}U=\frac{1}{2\pi i}\int_{\kappa_1 - iN}^{\kappa_1 + iN} e^{\lambda t}Z_3(\lambda,\xi)U1_{\{|\xi|> r_0\}} d\lambda,
\ee
where the constant $N>0$ is chosen large enough so that $N>y_1$ with $y_1$ defined in Lemma \ref{srlem5}. Since $Z_3(\lambda,\xi)=(\lambda-\mathbbm{M}(\xi))^{-1}-(\lambda-G_1(\xi))^{-1}$ is analytic on the domain $\re\lambda\ge-\beta $ with the constant $\beta >0$ defined in Lemma \ref{srlem5}, we can shift \eqref{dat16} from the line $\re\lambda=\kappa_1>0$ to $\re\lambda=-\beta $ and obtain
$$
\mathcal{V}_{\kappa_1,N}=\mathcal{V}_{-\beta ,N}+I_N,
$$
with
$$
I_NU=\frac{1}{2\pi i}\left(\int_{-\beta +iN}^{\kappa_1+iN}-\int_{-\beta -iN}^{\kappa_1-iN}\right)e^{\lambda t}Z_3(\lambda,\xi)U1_{\{|\xi|> r_0\}}d\lambda.
$$
By Lemma \ref{orlem3} and Lemma \ref{Azhengze}, it is easy to verify that for $\lambda=x+iN,$
\bmas
\|I_NU\| &\le C\int_{-\beta }^{\kappa_1}e^{xt}\|Z_3(\lambda,\xi)U\| dx
\le C\int_{-\beta }^{\kappa_1}e^{xt}\| (\lambda-G_1(\xi))^{-1}U\| dx
\\
&\le C\int_{-\beta }^{\kappa_1}e^{xt}\(\|(\lambda-A(\xi))^{-1}f\|+\|(\lambda-B_1(\xi))^{-1}(\rho,u)\| \)dx
\to0,\quad N\to\infty .
\emas
Now, let us estimate $\mathcal{V}_{-\beta ,N}$. By Lemma \ref{orlem4}, it holds $\sup\limits_{|\xi|\ge r_0,y\in\mathbb{R}}\|[I-G_2(\xi)(-\beta +iy-G_1(\xi))]^{-1}\|\le C$. Similar to the treatment of \eqref{dat10}, we obtian by \eqref{dat21} that for any $U,V\in Z_{\xi,\gamma}$ and $\lambda=-\beta +iy,$
\bma
&\left|(\mathcal{V}_{-\beta ,\infty}(t)U,V) \right|\le e^{-\beta t}\int_{-\infty}^{+\infty}|(Z_3(\lambda,\xi)U,V) |dy\nnm
\\
&\le Ce^{-\beta t}\int_{-\infty}^{+\infty}\|(\lambda-G_1(\xi))^{-1}U\| \|(\overline{\lambda}-G_1(\xi)^*)^{-1}U\| dy
\le Ce^{-\beta t}\|U\| \|V\|.\label{dat18}
\ema
From \eqref{dat18} and the fact $(\gamma+2|\xi|^{-2})^{-1}\|U\|^2\le \|U\|^2_{\xi,\gamma}\le (\gamma+2|\xi|^{-2})\|U\|^2 $ for $|\xi|> r_0$, we have
\be\label{dat19}
\|\mathcal{V}_{-\beta ,\infty}(t)\|_{\xi,\gamma} \le C\|\mathcal{V}_{-\beta ,\infty}(t)\| \le Ce^{-\beta t}.
\ee
Therefore, we conclude from \eqref{dat15}-\eqref{dat19} that
\be\label{dat20}
e^{t\mathbbm{M}(\xi)}U = e^{tG_1(\xi)}U +\mathcal{V}_{-\beta ,\infty}(t)U, \quad |\xi| >r_0.
\ee
By \eqref{dat13} and \eqref{dat20}, we obtain \eqref{dat1} with $S_1(t,\xi)U$ and $S_2(t,\xi)U$ defined by
\bmas
&S_1(t,\xi)U=\sum_{j=-1}^2e^{t\lambda_j(|\xi|)}\(U,\overline{\Psi_j(s,\omega)}\)_{\xi,\gamma}\Psi_j(s,\omega)
1_{\{|\xi|\le r_0\}},
\\
&S_2(t,\xi)U=\(e^{tG_4(\xi)}U+\mathcal{U}_{-\kappa_2,\infty}(t)U\)1_{\{|\xi|\le r_0\}}+\(e^{tG_1(\xi)}U+\mathcal{V}_{-\beta ,\infty}(t)U\)1_{\{|\xi|>r_0\}},
\emas
where $S_2(t,\xi)U$ satisfies \eqref{dat3} in terms of \eqref{da1}, \eqref{da3}, \eqref{dat17}, \eqref{dat9}, \eqref{dat11}, \eqref{dat19}. The proof of theorem is completed.
\end{proof}
\subsection{Optimal Time-Decay Rates of Linearized VPFP/NSP}
Based on the decomposition of the semigroup given in the previous subsection, we now study the optimal convergence rates of the solution of the linearized system to the equilibrium.

Set $U=(f,\rho,u)$ with $f=f(x,v)$, $\rho=\rho(x)$ and $u=u(x)$, denote the Banach space $Z^q=\{U=(f,\rho,u)\in L^{2,q}\times L^q_x\times L^q_x\mid\|U\|_{Z^q}<\infty\}$ with the norm
$$\|U\|_{Z^q}^2=\|f\|_{L^{2,q}}^2+\|\rho\|_{L^q_x}^2+\|u\|_{L^q_x}^2.$$
Define the Sobolev space $H^l_P$  as
$H^l_P=\{U=(f,\rho,u)\in Z^2\mid\|U\|_{H^l_P}<\infty\}$
equipped with the norm
\bmas
\|U\|^2_{H^l_P}&=\int_{\mathbb{R}^3}(1+|\xi|^2)^l\|\hat{U}\|^2_{\xi,\gamma}d\xi
\\&=\int_{\mathbb{R}^3}(1+|\xi|^2)^l\left(\int_{\mathbb{R}^3}|\hat{f}|^2dv+\gamma|\hat{\rho}|^2+|\hat{u}|^2+\frac{1}{|\xi|^2} |\tau_a\hat{f} -\hat{\rho} |^2\right)d\xi,
\emas
where $\hat{f}=\hat{f}(\xi,v)$, $\hat{\rho}=\hat{\rho}(\xi)$ and $\hat{u}=\hat{u}(\xi)$.

For any $U_0=(f_0,\rho_0,u_0)\in L^2_v(H^l_x)\times H^l_x\times H^l_x$, set
$$
e^{t\mathbbm{M}}U_0=(\mathcal{F}^{-1}e^{t\mathbbm{M}(\xi)}\mathcal{F})U_0.
$$
By Lemma \ref{srlem1}, we have
$$
\|e^{t\mathbbm{M}}U_0\|_{H^l_P}^2=\int_{\mathbb{R}^3}(1+|\xi|^2)^l\|e^{t\mathbbm{M}(\xi)}\hat{U}_0\|_{\xi,\gamma}^2d\xi\le\int_{\mathbb{R}^3}(1+|\xi|^2)^l\|\hat{U}_0\|_{\xi,\gamma}^2d\xi=\|U_0\|_{H^l_P}^2.
$$
This means that the linear operator $\mathbbm{M}$ generates a strongly continuous contradiction semigroup $e^{t\mathbbm{M}}$ in $H^l_P$. Therefore, $U=e^{t\mathbbm{M}}U_0$ is a global solution to \eqref{introe4} for the linearized VPFP/NSP system with initial data $U_0\in H^l_P$.
\begin{thm}\label{orlem6}
Let  $U(t)=(f(t),\rho(t),u(t))=e^{t\M}U_0$ be a solution of the system \eqref{introe4}. Then it holds for any $\alpha,\alpha'\in\mathbb{N}^3$ with $\alpha'\le\alpha$  that
\be\label{orl1}
\left\{\ball
&\|\dxa f(t)\|_{L^2_{x,v}}\le C(1+t)^{-(\frac{3}{4}+\frac{k}{2})}(\|\dxa U_0\|_{Z^2}+\|\dxap U_0\|_{Z^1}),
\\&\|\dxa \rho(t)\|_{L^2_x}\le C(1+t)^{-(\frac{3}{4}+\frac{k}{2})}(\|\dxa U_0\|_{Z^2}+\|\dxap U_0\|_{Z^1}),
\\&\|\dxa u(t)\|_{L^2_x}\le C(1+t)^{-(\frac{3}{4}+\frac{k}{2})}(\|\dxa U_0\|_{Z^2}+\|\dxap U_0\|_{Z^1}),
\\&\|\dxa \nabla_x\Phi(t)\|_{L^2_x}\le C(1+t)^{-(\frac{5}{4}+\frac{k}{2})}(\|\dxa U_0\|_{Z^2}+\|\dxap U_0\|_{Z^1}),
\eall\right.\ee
where $k=|\alpha-\alpha'|$, and $\nabla_x\Phi(t)=\nabla_x\ddx ^{-1}(\tau_af-\rho)$.
\end{thm}
\begin{proof}
By Theorem \ref{orlem5}, we have
$$
(\hat{f}(t),\hat{\rho}(t),\hat{u}(t))=e^{t\mathbbm{M}(\xi)}\hat{U}_0=S_1(t,\xi)\hat{U}_0+S_2(t,\xi)\hat{U}_0,
$$
where
$$
S_1(t,\xi)\hat{U}_0:=(h(t),H(t),E(t)),~~S_2(t,\xi)\hat{U}_0:=(g(t),G(t),B(t)).
$$
By Plancherel's equality, we have
\bma
\|\dxa f(t)\|_{L^2_{x,v}}&\label{orl2}=\|\xi^{\alpha}\hat{f}(t)\|_{L^2_{\xi,v}}\le\|\xi^{\alpha}h(t)\|_{L^2_{\xi,v}}+\|\xi^{\alpha}g(t)\|_{L^2_{\xi,v}},
\\
\|\dxa \rho(t)\|_{L^2_x}&=\|\xi^{\alpha}\hat{\rho}(t)\|_{L^2_\xi}\le\|\xi^{\alpha}H(t)\|_{L^2_\xi}+\|\xi^{\alpha}G(t)\|_{L^2_\xi},
\\
\|\dxa u(t)\|_{L^2_x}&=\|\xi^{\alpha}\hat{u}(t)\|_{L^2_\xi}\le\|\xi^{\alpha}E(t)\|_{L^2_\xi}+\|\xi^{\alpha}B(t)\|_{L^2_\xi},
\\
\|\dxa \nabla_x\Phi(t)\|_{L^2_x}&\nnm=\|\xi^{\alpha}|\xi|^{-1}[\tau_a\hat{f}(t)-\hat{\rho}(t)]\|_{L^2_\xi}\\&\label{orl3}\le\|\xi^{\alpha}|\xi|^{-1}[\tau_ah(t)-H(t)]\|_{L^2_\xi}+\|\xi^{\alpha}|\xi|^{-1}[\tau_ag(t)-G(t)]\|_{L^2_\xi}.
\ema
By \eqref{dat3} and the following fact
\bmas
\int_{\mathbb{R}^3}\frac{(\xi^{\alpha})^2}{|\xi|^2}|\tau_a\hat{f}_0-\hat{\rho}_0|^2d\xi
&\le\int_{|\xi|>1}(\xi^{\alpha})^2|\tau_a\hat{f}_0-\hat{\rho}_0|^2d\xi+\sup\limits_{|\xi|\le1}|\tau_a\hat{f}_0-\hat{\rho}_0|^2\int_{|\xi|\le1}\frac{1}{|\xi|^2}d\xi
\\
&\le C(\|\dxa(\tau_af_0-\rho_0)\|_{L^2_x}^2+\|\tau_af_0-\rho_0\|_{L^1_x}^2) ,
\emas
we can estimate the last terms of the right-hand side of \eqref{orl2}-\eqref{orl3} as follows
\bma
&\nnm\|\xi^{\alpha}g(t)\|_{L^2_{\xi,v}}^2+\gamma\|\xi^{\alpha}G(t)\|_{L^2_\xi}^2+\|\xi^{\alpha}B(t)\|_{L^2_\xi}^2+\|\xi^{\alpha}|\xi|^{-1}[\tau_ag(t)-G(t)]\|_{L^2_\xi}^2
\\
&\nnm=\int_{\mathbb{R}^3}(\xi^{\alpha})^2\|(g(t),G(t),B(t))\|_{\xi,\gamma}^2d\xi=\int_{\mathbb{R}^3}(\xi^{\alpha})^2\|S_2(t,\xi)\hat{U}_0\|_{\xi,\gamma}^2d\xi
\\
& \le Ce^{-2\kappa_0 t}\int_{\mathbb{R}^3} (\xi^{\alpha})^2\|\hat{U}_0\|_{\xi,\gamma}^2d\xi
\le Ce^{-2\kappa_0 t}(\|\dxa U_0\|_{Z^2}^2+\|U_0\|_{Z^1}^2), \label{orl4}
\ema
where we have made use of  $\|\tau_af_0\|_{L^1_x}\le\|f_0\|_{L^{2,1}}$.
By \eqref{dat2}, we have for $|\xi|\le r_0,$
\bma
\nnm S_1(t,\xi)\hat{U}_0&\nnm=\sum\limits_{j=-1}^2e^{t\lambda_j(|\xi|)}\left\{[(\hat{f}_0,\overline{\psi_{j,0}})+\gamma(\hat{\rho}_0,\overline{\zeta_{j,0}})+(\hat{u}_0,\overline{\vartheta_{j,0}})+(\tau_a\hat{f}_0-\hat{\rho}_0,\overline{\tau_a\psi_{j,2}-\zeta_{j,2}})]
\right.\\& \left.\quad\times(\psi_{j,0},\zeta_{j,0},\vartheta_{j,0}) +|\xi|((T_j(\xi)\hat{U}_0)_1,(T_j(\xi)\hat{U}_0)_2,(T_j(\xi)\hat{U}_0)_3)\right\}.\label{orl5}
\ema
By \eqref{orl5} and  \eqref{lfae4}, we have
\bma
h(t)&\nnm=\frac14\sum\limits_{j=\pm1}e^{t\lambda_j(|\xi|)}\bigg\{\sqrt{\frac{\gamma+1}{2}}(\tau_a\hat{f}_0+\hat{\rho}_0)-j[(\tau_b\hat{f}_0+\hat{u}_0)\cdot\omega]\bigg\}\[ \sqrt{\frac{2}{\gamma+1}} \chi_0-j(v\cdot\omega)\chi_0\]
\\
&\quad+\frac{1}{2}\sum\limits_{j=0,2}e^{t\lambda_j(|\xi|)}\[(\tau_b\hat{f}_0+\hat{u}_0)\cdot W^j\](v\cdot W^j)\chi_0+|\xi|\sum\limits_{j=-1}^2e^{t\lambda_j(|\xi|)}(T_j(\xi)\hat{U}_0)_1, \label{orl6}
\\
H(t)&\nnm=\frac{1}{2\sqrt{2(\gamma+1)}}\sum\limits_{j=\pm1}e^{t\lambda_j(|\xi|)}\bigg\{\sqrt{\frac{\gamma+1}{2}}(\tau_a\hat{f}_0+\hat{\rho}_0)-j[(\tau_b\hat{f}_0+\hat{u}_0)\cdot\omega]\bigg\}
\\&\quad+|\xi|\sum\limits_{j=\pm1}e^{t\lambda_j(|\xi|)}(T_j(\xi)\hat{U}_0)_2,
\\
E(t)&\nnm=-\frac{1}{4}\sum\limits_{j=\pm1}e^{t\lambda_j(|\xi|)}j\bigg\{\sqrt{\frac{\gamma+1}{2}}(\tau_a\hat{f}_0+\hat{\rho}_0)-j[(\tau_b\hat{f}_0+\hat{u}_0)\cdot\omega]\bigg\} \omega
\\
&\quad+\frac{1}{2}\sum\limits_{j=0,2}e^{t\lambda_j(|\xi|)}\left[(\tau_b\hat{f}_0+\hat{u}_0)\cdot W^j\right]W^j+|\xi|\sum\limits_{j=-1}^2e^{t\lambda_j(|\xi|)}(T_j(\xi)\hat{U}_0)_3,
\ema
where $W^j~(j=0,2)$ are given by \eqref{lfae4}, and $T_j(\xi)$ for $-1\le j\le2$ are linear operators in $L^2(\R^3_v)\times\C\times\C^3$  with the norms $\|T_j(\xi)\|$ uniformly bounded for $|\xi|\le r_0$. Moreover,
\bma
\tau_ah(t)-H(t)&\nnm=
\frac{\gamma-1}{4\sqrt{2(\gamma+1)}}|\xi|^2\sum\limits_{j=\pm1}e^{t\lambda_j(|\xi|)}\bigg\{\sqrt{\frac{\gamma+1}{2}}(\tau_a\hat{f}_0+\hat{\rho}_0)-j[(\tau_b\hat{f}_0+\hat{u}_0)\cdot\omega]\bigg\}
\\
&\quad+|\xi|^3\sum\limits_{j=\pm1}e^{t\lambda_j(|\xi|)}R_j(\xi)\hat{U}_0,\label{orl20}
\ema
where  $R_j(\xi)$ for $j=\pm1$ are linear operators in $L^2(\R^3_v)\times\C\times\C^3$  with the norm uniformly bounded for $|\xi|\le r_0$.
Since
$$
\re\lambda_{j}(|\xi|) \le -\frac12|\xi|^2,~~|\xi|\le r_0,~j=-1,0,1,2,
$$
we obtain by \eqref{orl6}-\eqref{orl20} that
\bma
\|\xi^{\alpha}h(t)\|_{L^2_{\xi,v}}^2+\|\xi^{\alpha}H(t)\|_{L^2_\xi}^2+\|\xi^{\alpha}E(t)\|_{L^2_\xi}^2&\label{orl8}\le C(1+t)^{-(3/2+k)}\|\dxap U_0\|_{Z^1}^2,
\\
\|\xi^{\alpha}|\xi|^{-1}[\tau_ah(t)-H(t)]\|_{L^2_\xi}^2&\le C(1+t)^{-(5/2+k)}\|\dxap U_0\|_{Z^1}^2, \label{orl9}
\ema
with $\alpha'\le\alpha$ and $k=|\alpha-\alpha'|$. The combination of \eqref{orl2}-\eqref{orl4} and \eqref{orl8}-\eqref{orl9} leads to \eqref{orl1}. The proof of theorem is completed.
\end{proof}

In fact, the following theorem will prove that the time deca rates above are optimal.
\begin{thm}\label{linearoptdr}
Let $U(t)=(f(t),\rho(t),u(t))=e^{t\M}U_0$ be a solution of the system \eqref{introe4}. Assume that the initial data $U_0=(f_0,\rho_0,u_0)\in L^2_v(H^l_x\cap L^1_x)\times H^l_x\cap L^1_x\times H^l_x\cap L^1_x$ for $l\ge2$ and that there exists a constant $d_0>0$ such that
$$
\inf\limits_{|\xi|\le r_0}|\tau_a\hat{f}_0+\hat{\rho}_0|\ge d_0,\quad \sup\limits_{|\xi|\le r_0}|\tau_b\hat{f}_0+\hat{u}_0|=0.
$$
Then, it holds for $t>0$ being large enough that
\be\label{orl10}
\left\{\ball
C_1(1+t)^{-\frac{3}{4}}\le&\| f(t)\|_{L^2_{x,v}}\le C_2(1+t)^{-\frac{3}{4}},
\\C_1(1+t)^{-\frac{3}{4}}\le&\| \rho(t)\|_{L^2_x}\le C_2(1+t)^{-\frac{3}{4}},
\\C_1(1+t)^{-\frac{3}{4}}\le&\| u(t)\|_{L^2_x}\le C_2(1+t)^{-\frac{3}{4}},
\eall\right.
\ee
and for $\gamma>1$ that
\be\label{phizysj}
C_1(1+t)^{-\frac{5}{4}}\le\| \nabla_x\Phi(t)\|_{L^2_x}\le C_2(1+t)^{-\frac{5}{4}},
\ee
where $0<C_1\le C_2$ are two generic constants.
\end{thm}

\begin{proof}
By Theorem \ref{orlem6}, we need only show the lower bounds of the time decay rates for the solution $U(t)=(f(t),\rho(t),u(t))$ under the assumption of Theorem \ref{srlem13}. Indeed, in terms of Theorem \ref{orlem6}, we have
\bma
\|f(t)\|_{L^2_{x,v}}&\ge\|h(t)\|_{L^2_{\xi,v}}-\|g(t)\|_{L^2_{\xi,v}}\ge\|h(t)\|_{L^2_{\xi,v}}-Ce^{-\kappa_0t},\label{orl13}
\\
\|\rho(t)\|_{L^2_x}&\ge\|H(t)\|_{L^2_\xi}-\|G(t)\|_{L^2_\xi}\ge\|H(t)\|_{L^2_\xi}-Ce^{-\kappa_0t},\label{orl14}
\\
\|u(t)\|_{L^2_x}&\ge\|E(t)\|_{L^2_\xi}-\|B(t)\|_{L^2_\xi}\ge\|E(t)\|_{L^2_\xi}-Ce^{-\kappa_0t},
\\
\|\nabla_x\Phi(t)\|_{L^2_x}&\nnm\ge\||\xi|^{-1}[\tau_ah(t)-H(t)]\|_{L^2_\xi}-\||\xi|^{-1}[\tau_ag(t)-G(t)]\|_{L^2_\xi}
\\
&\ge\||\xi|^{-1}[\tau_ah(t)-H(t)]\|_{L^2_\xi}-Ce^{-\kappa_0t},\label{orl15}
\ema
where we have made use of \eqref{orl4}
$$
\int_{\mathbb{R}^3}\|(g(t),G(t),B(t))\|_{\xi,\gamma}^2d\xi\le Ce^{-2\kappa_0 t}(\|U_0\|_{Z^2}^2+\|U_0\|_{Z^1}^2).
$$
By \eqref{fugonge} and \eqref{orl6}, we have
\be\label{orl16}
\|h(t)\|_{L^2_v}^2\ge\frac{1}{16}e^{2t\re\lambda_1(|\xi|)}|\tau_a\hat{f}_0+\hat{\rho}_0|^2-C|\xi|^2e^{-|\xi|^2t}\|\hat{U}_0\|_{\xi,\gamma}^2,
\ee
where $C>0$ is a constant. Noticing that for $t\ge t_0=\frac{1}{r_0^2}$, we have
$$
\int_{|\xi|\le r_0}e^{-|\xi|^2t}d\xi=4\pi t^{-\frac{3}{2}}\int_0^{r_0\sqrt{t}}r^2e^{-r^2}dr\ge C(1+t)^{-\frac{3}{2}},
$$
this together with \eqref{orl16} yeild
\bma
\|h(t)\|_{L^2_{\xi,v}}^2&\nnm\ge\frac{1}{16}d_0^2\int_{|\xi|\le r_0}e^{-|\xi|^2t}d\xi-C\int_{|\xi|\le r_0}|\xi|^2e^{-|\xi|^2t}\|\hat{U}_0\|_{\xi,\gamma}^2d\xi
\\&\ge C(1+t)^{-\frac{3}{2}}-C(1+t)^{-\frac{5}{2}}\label{orl17}.
\ema
Similar to the discussion above, we can also obtian
\bmas
|E(t)|^2,|H(t)|^2&\ge\frac{1}{16}e^{2t\re\lambda_1(|\xi|)}|\tau_a\hat{f}_0+\hat{\rho}_0|^2-C|\xi|^2e^{-|\xi|^2t}\|\hat{U}_0\|_{\xi,\gamma}^2,
\\|\tau_ah(t)-H(t)|^2&\ge C|\xi|^4e^{2t\re\lambda_{1}(|\xi|)}|\tau_a\hat{f}_0+\hat{\rho}_0|^2-C|\xi|^6e^{-|\xi|^2t}\|\hat{U}_0\|_{\xi,\gamma}^2,\quad \gamma>1,
\emas
and then
\bma
\|E(t)\|_{L^2_\xi}^2,\|H(t)\|_{L^2_\xi}^2&\label{orl18}\ge C(1+t)^{-\frac{3}{2}}-C(1+t)^{-\frac{5}{2}},
\\\||\xi|^{-1}[\tau_ah(t)-H(t)]\|_{L^2_\xi}^2& \ge C(1+t)^{-\frac{5}{2}}-C(1+t)^{-\frac{7}{2}},\quad \gamma>1.\label{orl19}
\ema
Combining \eqref{orl13}-\eqref{orl15} and \eqref{orl17}-\eqref{orl19}, we obtain \eqref{orl10}-\eqref{phizysj} for $t>0$ large enough. The proof of theorem is completed.
\end{proof}
\begin{remark}\label{phi4jie}
Noticing that the special case $\gamma=1$ leads to the vanishing of the $|\xi|^2$ order term in \eqref{orl20}. Indeed, for the general viscous term $\nu\Delta_x u$ in the equation $\eqref{introe2}_3$ and $\gamma=1$, we obtain the following asymptotic expansion from \eqref{lfae5} and \eqref{nxtzz} that
$$
\tau_ah(t)-H(t)=\frac{\nu-1}{32}|\xi|^3\sum\limits_{j=\pm1}e^{t\lambda_j(|\xi|)}\Big\{(\tau_a\hat{f}_0+\hat{\rho}_0)-j[(\tau_b\hat{f}_0+\hat{u}_0)\cdot\omega]\Big\}+O(|\xi|^4e^{-\frac12|\xi|^2t}).
$$
Thus, for $\gamma=1$ and $\nu\ne 1$, the electric field $\nabla_x\Phi$ admits the following optimal time decay rate
$$
C_1(1+t)^{-\frac{7}{4}}\le\|\nabla_x\Phi(t)\|_{L^2_x}\le C_2(1+t)^{-\frac{7}{4}},
$$
while for $\nu=1$ only an upper bound holds $\|\nabla_x\Phi\|_{L^2_x}\le C(1+t)^{-\frac{9}{4}}$. 
\end{remark}

\section{The Original Nonlinear Problem}
In this section, we establish the global existence of the solution to the original nonlinear  VPFP/NSP system \eqref{introe2},  and with the help of the estimates on the linearized system obtained in Section 3, we prove the long time decay rates of the solution to this system.

\subsection{Energy Estimate}

Let $U=(f,\rho,u)$ be a solution to \eqref{introe2} on $0\le t\le T^*$. We assume that
\be
\sup_{0\le t\le T^*}E_3(U)\le \delta_0, \label{assume}
\ee
where $\delta_0>0$ is a sufficiently small constant.

\begin{lem}
Let $U=(f,\rho,u)$ be a strong solution to \eqref{introe2}. Then, there exists a constant $ q_0>0$ such that
\bma
&\frac{1}{2}\Dt\(\|f\|_{L^2_{x,v}}^2+\gamma\|\rho\|_{L^2_x}^2+\|u\|_{L^2_x}^2+\|\tdx\Phi\|_{L^2_x}^2\)
\nnm
\\
&+ q_0\(\|P_1f\|_{L^2_\sigma(L^2_x)}^2+\|\nabla_xu\|_{L^2_x}^2+\|\tau_bf-u\|_{L^2_x}^2\)
\le C\sqrt{E_3(U)}D_3(U).\label{geine2}
\ema
\end{lem}

\begin{proof}
Multiplying $\eqref{introe2}_1$, $\eqref{introe2}_2$, $\eqref{introe2}_3$ by $f,\gamma\rho$ and $u$, respectively, taking integration and summation, and using
\bma
&\nnm-\intr\tdx\Phi\cdot(\tau_bf-u)dx=\intr\Phi\divx(\tau_bf-u) dx
\\
&=-\intr\Phi[\dt(\tau_af-\rho)-\divx(\rho u)] dx=\frac{1}{2}\Dt\|\tdx\Phi\|_{L^2_x}^2-\intr\tdx\Phi\cdot(\rho u)dx, \label{0jiedian}
\ema
we obtain
\bmas
&\frac{1}{2}\Dt\(\|f\|_{L^2_{x,v}}^2+\gamma\|\rho\|_{L^2_x}^2+\|u\|_{L^2_x}^2+\|\tdx\Phi\|_{L^2_x}^2\)+\intr(-LP_1f,f)dx+\|\tau_bf-u\|_{L^2_x}^2+\|\nabla_xu\|_{L^2_x}^2
\\
&=\intr(\tdx\Phi+u)\(\frac{1}{2}vf,f\)dx+\intr u\cdot R_2dx+\intr(\gamma\tdx\rho+\tdx\Phi)\cdot(\rho u)dx:=\sum\limits_{k=1}^3J_k.
\emas
We have for $J_3$ that
\bmas
J_3&\le C\|(\tdx\rho,\tdx\Phi)\|_{L^2_x}\|\rho\|_{L^3_x}\|u\|_{L^6_x}\le C\|\rho\|_{H^1_x}\|(\tdx\rho,\tdx u,\tdx\Phi)\|_{L^2_x}^2
\le C\sqrt{E_3(U)}D_3(U).
\emas
By using the macro-micro decomposition $f=P_0f+P_1f$, we have
$$
\frac{1}{2} (vf,f )=\tau_af\tau_bf+(P_0f,vP_1f)+\frac{1}{2} (P_1f,vP_1f ),
$$
and then $J_1$ gives the terms as follows
\bmas
J_1=&\intr \tdx\Phi\cdot\tau_af\tau_bfdx+\intr u\cdot\tau_af\tau_bfdx
+\intr(\tdx\Phi+u)\cdot(P_0f,vP_1f)dx\\
&+\frac{1}{2}\intr(\tdx\Phi+u)\cdot (P_1f,vP_1f )dx:=\sum\limits_{k=1}^4J_{1k},
\emas
where
\bmas
J_{11}&\le C\|\tdx\Phi\|_{L^2_x}\|\tau_af\|_{L^3_x}\|\tau_bf\|_{L^6_x}
\\
&\le C\|\tau_af\|_{H^1_x}\|\tdx\Phi\|_{L^2_x}\|\tdx\tau_bf\|_{L^2_x}
\le C\sqrt{E_3(U)}D_3(U),
\\
J_{13}+J_{14}&\le C \|\tdx\Phi+u\|_{L^3_x}(\|P_0f\|_{L^2_v(L^6_x)}\|vP_1f\|_{L^2_{x,v}}+\|P_1f\|_{L^6_{x,v}}\|vP_1f\|_{L^2_{x,v}})
\\
&\le C\|(\tdx\Phi,u)\|_{H^1_x}\|\tdx P_0f\|_{L^2_{x,v}}\|P_1f\|_{L^2_\sigma(H^1_x)}\le C\sqrt{E_3(U)}D_3(U).
\emas
$J_2$ gives the terms as follows
\bmas
J_2=&-\intr u\cdot\left\{(u\cdot\nabla_x u)-\frac{\rho}{1+\rho}\ddx u-\[\frac{p'(1+\rho)}{1+\rho}-\gamma\]\tdx\rho +\frac{\rho}{1+\rho}(\tau_bf-u)\right\}dx
\\
&-\intr u\cdot\frac{\tau_afu}{1+\rho}dx:=J_{21}+J_{22},
\emas
where
\bmas
J_{21}&\le C\|u\|_{L^6_x}\[\|u\|_{L^3_x}\|\tdx u\|_{L^2_x}+\|\rho\|_{L^3_x}(\|\ddx u\|_{L^2}+\|\tdx\rho\|_{L^2}+\|\tau_bf-u\|_{L^2_x})\]
\\&\le C\|(\rho,u)\|_{H^1_x}\|\tdx u\|_{L^2_x}(\|\tdx u\|_{H^1_x}+\|\tdx\rho\|_{L^2_x}+\|\tau_bf-u\|_{L^2_x})
\le C\sqrt{E_3(U)}D_3(U).
\emas
Finally,
\bmas
J_{12}+J_{22}&=\intr u\cdot\tau_af(\tau_bf-u)dx+\intr \frac{\rho}{1+\rho}\tau_afu\cdot udx
\\
&\le C\|u\|_{L^3_x}\|\tau_af\|_{L^6_x}\|\tau_bf-u\|_{L^2_x}+C\|u\|_{L^2_x}\|\rho\|_{L^6_x}\|\tau_af\|_{L^6_x}\|u\|_{L^6_x}
\\
&\le C\|u\|_{H^1_x}(\|\tdx(\tau_af,\rho)\|_{L^2_x}^2+\|\tau_bf-u\|_{L^2_x}^2)\le C\sqrt{E_3(U)}D_3(U).
\emas
Combined with \eqref{assume}, summing up the results above yields \eqref{geine2}.
\end{proof}

\begin{lem}
Let $U=(f,\rho,u)$ be a strong solution to \eqref{introe2}. Then, there exists a constant $ q_0>0$ such that
\bma
&\nnm\frac{1}{2}\Dt\sum\limits_{1\le|\alpha|\le 3}\(\|\dxa f\|_{L^2_{x,v}}^2+\gamma\|\dxa\rho\|_{L^2_x}^2+\|\dxa u\|_{L^2_x}^2+\|\dxa\tdx\Phi\|_{L^2_x}^2\)
\\
&\nnm+ q_0\sum\limits_{1\le|\alpha|\le 3}\(\|\dxa P_1f\|_{L^2_\sigma(L^2_x)}^2+\|\dxa\nabla_xu\|_{L^2_x}^2+\|\dxa(\tau_bf-u)\|_{L^2_x}^2\)
\\&\le C\|(\tau_bf,\rho,u,\tdx\Phi)\|_{H^3_x}\(\|\tdx P_0f\|_{L^2_v(H^2_x)}^2+\|\tdx\rho\|^2_{H^2_x}\)\le C\sqrt{E_3(U)}D_3(U).\label{geine3}
\ema
\end{lem}

\begin{proof}
Applying $\dxa$ with $1\le|\alpha|\le 3$ to $\eqref{introe2}_1$, $\eqref{introe2}_2$ and $\eqref{introe2}_3$, respectively, and then multiplying them by $\dxa f,\gamma\dxa\rho$ and $\dxa u$, respectively, taking integration and summation, by using
\bma
&\nnm-\intr\dxa(\tau_bf-u)\cdot\dxa\tdx\Phi dx=-\intr\dxa[\dt(\tau_af-\rho)-\divx(\rho u)]\dxa\Phi dx
\\
&=\frac{1}{2}\Dt\|\dxa\tdx\Phi\|_{L^2_x}^2-\intr\dxa\tdx\Phi\cdot\dxa(\rho u)dx,\label{gaojiedian}
\ema
we obtain
\bmas
&\frac{1}{2}\Dt\(\|\dxa f\|_{L^2_{x,v}}^2+\gamma\|\dxa\rho\|_{L^2_x}^2+\|\dxa u\|_{L^2_x}^2+\|\dxa\tdx\Phi\|_{L^2_x}^2\)
\\
&+\intr(-L\dxa P_1f,\dxa f)dx+\|\dxa\tdx u\|_{L^2_x}^2+\|\dxa(\tau_bf-u)\|_{L^2_x}^2:=\sum\limits_{k=1}^3J_k.
\emas
$J_1$ gives the terms as follows
\bmas
J_1&=\intr\(\dxa\[(\tdx\Phi+u)\cdot \(\frac{1}{2}vf-\tdv f\)\],\dxa f\)dx
\\
&\le C\sum\limits_{1\le|\alpha'|\le|\alpha|-1}
\intr \|\dx^{\alpha-\alpha'}(\tdx\Phi+u)\|_{L^3_x}(\|\dx^{\alpha'} |v| f\|_{L^6_x}+\|\dx^{\alpha'} \tdv f\|_{L^6_x})\|\dxa f\|_{L^2_x}dv
\\
&\quad+C\intr \|\tdx\Phi+u\|_{L^\infty_x}(\|\dxa |v|f\|_{L^2_x}+\|\dxa \tdv f\|_{L^2_x})\|\dxa f\|_{L^2_x}dx\\
&\quad+C\intr \|\dxa(\tdx\Phi+u)\|_{L^2_x}(\||v|f\|_{L^\infty_x}+\| \tdv f\|_{L^\infty_x})\|\dxa f\|_{L^2_x}dx
\\
&\le C\|(u,\tdx\Phi)\|_{H^2_x}\|\tdx P_0f\|_{L^2_v(H^2_x)}^2+C\|(u,\tdx\Phi)\|_{H^2_x} \|\tdx P_1f\|_{L^2_\sigma(H^2_x)}^2.
\emas
$J_2$ gives the terms as follows
$$
J_2=-\gamma\intr\dxa\rho\dxa\divx(\rho u)dx=-\gamma\intr\dxa\rho\dxa(\tdx\rho\cdot u+\rho\divx u)dx:=J_{21}+J_{22}.
$$
We denote $[A,B]:=AB-BA$ for two operators $A$ and $B$. By a direct calculation, we have
\bmas
J_{21}
& =\frac{1}{2}\gamma\intr(\dxa\rho)^2\divx udx-\gamma\intr\dxa\rho[\dxa,u\cdot\tdx]\rho dx
\\
&\le C\|\dxa\rho\|_{L^2_x}^2\|\divx u\|_{L^\infty_x}+C\|\dxa\rho\|_{L^2_x}\|u\|_{H^3_x}\|\tdx\rho\|_{H^2_x}
\\
&\le C\|u\|_{H^3_x}(\|\dxa\rho\|_{L^2_x}^2+\|\tdx\rho\|_{H^2_x}^2),
\\
J_{22}&
\le C\|\rho\|_{H^3_x}\|\dxa\rho\|_{L^2_x}^2+C\|\rho\|_{H^3_x}\sum\limits_{1\le|\alpha|\le3}\|\dxa\tdx u\|_{L^2_x}^2.
\emas
$J_3$ gives the terms as follows
\bmas
J_3&=\intr\dxa u\cdot \dxa\left\{-(u\cdot\tdx u)-\frac{\rho}{1+\rho}\ddx u-\[\frac{p'(1+\rho)}{1+\rho}-\gamma\]\tdx\rho
\right.\\&\left.\quad+\frac{\rho}{1+\rho}(\tau_bf-u)-\frac{\tau_afu}{1+\rho}\right\}dx+\intr\dxa\tdx\Phi\cdot\dxa(\rho u) dx:=\sum\limits_{k=1}^{6}J_{3k}.
\emas
We calculate $J_{32}$ and $J_{33}$ first
\bmas
J_{32}&=-\intr \dxa u\cdot\frac{\rho}{1+\rho}\dxa\ddx udx-\intr\dxa u\cdot\[\dxa,\frac{\rho}{1+\rho}\ddx\] udx:=J_{32,1}+J_{32,2},
\\
J_{33}&=-\intr\dxa u\cdot\left\{\[\frac{p'(1+\rho)}{1+\rho}-\gamma\]\dxa\tdx\rho +\[\dxa,\[\frac{p'(1+\rho)}{1+\rho}-\gamma\]\tdx\]\rho\right\} dx:=J_{33,1}+J_{33,2}.
\emas
By integrating by parts, we obtain
\bmas
J_{32,1}&=\intr \dxa \tdx u\cdot\frac{\rho}{1+\rho}\dxa\tdx udx+\intr \dxa u\cdot\tdx\(\frac{\rho}{1+\rho}\)\dxa\tdx udx
\\
&\le C\|\rho\|_{L^\infty_x}\|\dxa\tdx u\|_{L^2_x}^2+C\|\dxa u\|_{L^6_x}\|\tdx\rho\|_{L^3_x}\|\dxa\tdx u\|_{L^2_x}
\le C\|\rho\|_{H^2_x}\|\dxa\tdx u\|_{L^2_x}^2,
\\
J_{33,1}&=\intr\tdx\[\frac{p'(1+\rho)}{1+\rho}-\gamma\]\cdot\dxa u\dxa\rho dx+\intr\[\frac{p'(1+\rho)}{1+\rho}-\gamma\]\dxa \divx u\dxa\rho dx
\\
&\le C\|\tdx\rho\|_{L^3_x}\|\dxa u\|_{L^6_x}\|\dxa\rho\|_{L^2_x}+C\|\rho\|_{L^\infty_x}\|\dxa\divx u\|_{L^2_x}\|\dxa\rho\|_{L^2_x}
\\
&\le C\|\rho\|_{H^2_x}(\|\dxa\tdx u\|_{L^2_x}^2+\|\dxa\rho\|_{L^2_x}^2),
\\
J_{32,2}&\le C\|\rho\|_{H^3_x}\|\dxa\tdx u\|_{L^2_x}^2+C\|\rho\|_{H^3_x}\sum\limits_{1\le|\alpha|\le3}\|\dxa\tdx u\|_{L^2_x}^2,
\\
J_{33,2}&\le C\|\dxa u\|_{L^2_x}\|\tdx\rho\|_{H^2_x}^2.
\emas
By a direct calculation, we have for the rest of $J_3$ that
\bmas
J_{31}&\le C\|u\|_{H^3_x}\|\dxa\tdx u\|_{L^2_x}^2+C\|u\|_{H^3_x}\sum\limits_{1\le|\alpha|\le3}\|\dxa\tdx u\|_{L^2_x}^2,
\\
J_{34}&\le C\|\dxa u\|_{L^2_x}\|\tdx\rho\|_{H^2_x}^2+C\|\dxa u\|_{L^2_x}\sum\limits_{1\le|\alpha|\le3}\|\dxa(\tau_bf-u)\|_{L^2_x}^2,
\\
J_{35}&=-\intr\dxa u\cdot\dxa\(\tau_afu\)dx+\intr\dxa u\cdot\dxa\(\frac{\rho\tau_afu}{1+\rho}\)dx
\\
&\le C\|u\|_{H^3_x}\(\|\dxa\tdx u\|_{L^2_x}^2+\|\tdx(\tau_af,\rho)\|_{H^2_x}^2\)+C\|u\|_{H^3_x}\sum\limits_{1\le|\alpha|\le3}\|\dxa\tdx u\|_{L^2_x}^2,
\\
J_{36}&\le C\|\dxa\tdx\Phi\|_{L^2_x}\|\tdx\rho\|_{H^2_x}^2+C\|\dxa\tdx\Phi\|_{L^2_x}\sum\limits_{1\le|\alpha|\le3}\|\dxa\tdx u\|_{L^2_x}^2.
\emas
Combined with \eqref{assume}, summing up the results above yields \eqref{geine3}.
\end{proof}

\begin{lem}Let $U=(f,\rho,u)$ be a strong solution to \eqref{introe2}. Then, there exists two constants $q_1,C_1>0$ such that
\bma
&\nnm\sum\limits_{|\alpha|\le 2}\left\{\Dt\[\|\dxa\tdx\Phi\|_{L^2_x}^2-\intr\dxa\tdx\Phi\cdot\dxa(\tau_bf-u)dx\]+\|\dxa\tdx\Phi\|_{L^2_x}^2+\|\dxa\ddx\Phi\|_{L^2_x}^2\right\}
\\
&\le C_1\|\tdx\rho\|_{H^2_x}^2+q_1\(\|\tdx P_1f\|_{L^2_\sigma(H^2_x)}^2+\|\tdx u\|_{H^3_x}^2\)+q_1\sum\limits_{|\alpha|\le 2}\|\dxa(\tau_bf-u)\|_{L^2_x}^2.\label{geine1}
\ema
\end{lem}

\begin{proof}By taking the inner product between $\chi_0$, $v\chi_0$ and $\eqref{introe2}_1$, we obtain 
\be\label{ge2}
\left\{\ball
&\dt{\tau_af}+\div_x\tau_bf=0,
\\
&\dt\tau_bf+\nabla_x\tau_af+(\tau_bf-u)-\nabla_x\Phi=R_3,
\eall\right.
\ee
where
$$
R_3=-(v\cdot\tdx P_1f,v\chi_0)+\left(\nabla_x\Phi+u\right)\tau_af.
$$
Subtracting $\dxa\eqref{ge2}_2$ from $\dxa\eqref{introe2}_3$, and taking the inner product of $\dxa\tdx\Phi$, we have
\bmas
&-\intr\dxa\tdx\Phi\cdot\dt\dxa(\tau_bf-u) dx-\intr\dxa\tdx\Phi\cdot\dxa\tdx(\tau_af-\rho) dx+(\gamma-1)\intr\dxa\tdx\Phi\cdot\dxa\tdx\rho dx
\\
&-\intr\dxa\tdx\Phi\cdot\dxa\ddx u dx-2\intr\tdx\Phi\cdot\dxa(\tau_bf-u)\dxa dx+2\|\dxa\tdx\Phi\|_{L^2_x}^2
\\
&=\intr \dxa\tdx\Phi\cdot\dxa(R_2-R_3) dx.
\emas
By using \eqref{gaojiedian} and
\bma\label{tidudian}
-\intr\dxa\tdx(\tau_af-\rho)\cdot\dxa\tdx\Phi dx=\intr\dxa(\tau_af-\rho)\dxa\ddx\Phi dx=\|\dxa\ddx\Phi\|_{L^2_x}^2,
\ema
we obtain
\bmas
&\nnm\Dt\[\|\dxa\tdx\Phi\|_{L^2_x}^2-\intr\dxa\tdx\Phi\cdot\dxa(\tau_bf-u)dx\]+2\|\dxa\tdx\Phi\|_{L^2_x}^2+\|\dxa\ddx\Phi\|_{L^2_x}^2
\\
&=-\intr \dxa\dt\tdx\Phi\cdot\dxa(\tau_bf-u) dx+2\intr\dxa\tdx\Phi\cdot\dxa(\rho u)dx+\intr\dxa\tdx\Phi\cdot\dxa\ddx udx
\\
&\quad+\intr\dxa\tdx\Phi\cdot\dxa R_2dx-\intr\dxa\tdx\Phi\cdot\dxa R_3dx-(\gamma-1)\intr\dxa\tdx\Phi\cdot\dxa\tdx\rho dx:=\sum\limits_{k=1}^6J_k.
\emas
By direct calculation, we have
\bmas
J_1 &=-\intr\dxa\dt\tdx\ddx^{-1}(\tau_af-\rho)\cdot\dxa(\tau_bf-u)dx
\\
&=\intr\dxa\tdx\ddx^{-1}\[-\divx(\tau_bf-u)+\divx(\rho u)\]\cdot\dxa(\tau_bf-u)dx
\\
&\le C(\|\dxa(\tau_bf-u)\|_{L^2_x}^2+C\|u\|_{H^3_x}\|\dxa(\tau_bf-u)\|_{L^2_x}\|\tdx\rho\|_{H^2_x}
\\
&\le C(1+\|u\|_{H^3_x})\|\dxa(\tau_bf-u)\|_{L^2_x}^2+C\|u\|_{H^3_x}\|\tdx\rho\|_{H^2_x}^2,
\\
J_2&\le C\|u\|_{H^3_x}\(\|\dxa\tdx\Phi\|_{L^2_x}^2+\|\tdx\rho\|_{H^2_x}^2\),
\\
J_3&\le \epsilon\|\dxa\tdx\Phi\|_{L^2_x}^2+\frac{C}{\epsilon}\|\dxa\ddx u\|_{L^2_x}^2,
\\
J_6&\le \epsilon\|\dxa\tdx\Phi\|_{L^2_x}^2+\frac{C}{\epsilon}\|\dxa\tdx\rho\|_{^2_x}^2.
\emas
Similarly, $J_k~(k=4,5)$ give the terms as follows
\bmas
J_4&
=\intr\dxa\tdx\Phi\cdot\dxa\left\{-(u\cdot\nabla_x)u-\frac{\rho}{1+\rho}\ddx u-\[\frac{p'(1+\rho)}{1+\rho}-\gamma\]\tdx\rho
\right.\\&\left.\quad-\frac{\rho}{1+\rho}(\tau_bf-u)-\frac{\tau_afu}{1+\rho}\right\}dx
\\&\le \|(\tau_af,\tau_bf,\rho,u)\|_{H^3_x}\(\|\dxa\tdx\Phi\|_{L^2_x}^2+\|(\tdx\rho,\tdx u,\ddx u)\|_{H^2_x}^2\),
\\
J_5&=-\intr\dxa\tdx\Phi\cdot\dxa(v\cdot\tdx P_1f,v\chi_0)dx+\intr\dxa\tdx\Phi\cdot\dxa[(\tdx\Phi+u)\tau_af]dx\\
&\le \epsilon\|\dxa\tdx\Phi\|_{L^2_x}^2+\frac{C}{\epsilon}\|\dxa\tdx P_1f\|_{L^2_\sigma(L^2_x)}^2+C\|\tau_af\|_{H^3_x}\(\|\dxa\tdx\Phi\|_{L^2_x}^2+\|(\tdx u,\tdx\Phi)\|_{H^2_x}^2\).
\emas
Combined with \eqref{assume}, summing up the results above and choosing $\eps>0$ small enough yields \eqref{geine1}.
\end{proof}

\begin{lem}
Let $U=(f,\rho,u)$ be a strong solution to \eqref{introe2}. Then, there exists two constants $ q_0, q_1>0$ such that
\bma
&\nnm\sum\limits_{|\alpha|\le 2}\left\{\Dt \intr\(\dxa\tdx\tau_af\cdot\dxa\tau_bf+\dxa\tdx\rho\cdot\dxa u\)dx
+\frac\gamma2\|\dxa\tdx\rho\|_{L^2_x}^2+ q_0\|\dxa(\tdx\tau_af,\ddx\Phi)\|_{L^2_x}^2\right\}
\\
&\le  q_1\(\|\tdx P_1f\|_{L^2_\sigma(H^2_x)}^2+\|\tdx(\tau_bf-u)\|_{H^2_x}^2+\|\tdx u\|_{H^3_x}^2\)+\|\divx\tau_bf\|_{H^2_x}^2.\label{geine4}
\ema
\end{lem}

\begin{proof}
Applying $\dxa$ with $|\alpha|\le 2$ to $\eqref{introe2}_3$ and $\eqref{ge2}_2$, respectively, and then multiplying them by $\dxa\tdx\rho$ and $\dxa\tdx\tau_af$, respectively, taking integration and using \eqref{tidudian}, we obtain
\bmas
&\Dt\(\intr\dxa\tdx\tau_af\cdot\dxa\tau_bfdx+\intr\dxa\tdx\rho\dxa udx\)
+\gamma\|\dxa\tdx\rho\|_{L^2_x}^2+\|\dxa(\tdx\tau_af,\ddx\Phi)\|^2_{L^2_x}
\\
&=-\intr\dxa\tdx(\tau_af-\rho)\cdot\dxa(\tau_bf-u)dx+\intr\dt\dxa\tdx\tau_af\cdot\dxa\tau_bfdx+\intr\dt\dxa\tdx\rho\cdot\dxa udx
\\
&\quad+\intr\dxa\tdx\rho\cdot\dxa\ddx udx+\intr\dxa\tdx\rho\cdot\dxa R_2dx+\intr\dxa\tdx\tau_af\cdot\dxa R_3dx:=\sum\limits_{k=1}^6J_k.
\emas
By a direct calculation, we have
\bmas
J_1&\le\epsilon\|\dxa\ddx\Phi\|_{L^2_x}^2+\frac{C}{\epsilon}\|\dxa\tdx(\tau_bf-u)\|_{L^2_x}^2,
\\
J_2&=-\intr\dt\dxa\tau_af\dxa\divx\tau_bfdx=\|\dxa\divx\tau_bf\|_{L^2_x}^2,
\\
J_4&\le \epsilon\|\dxa\tdx\rho\|_{L^2_x}^2+\frac{C}{\epsilon}\|\dxa\ddx u\|_{L^2_x}^2.
\emas
Similarly, $J_k~(k=3,5,6)$ give the terms as follows
\bmas
J_3&=-\intr\dt\dxa\rho\dxa\divx udx=\intr\dxa\divx u\dxa[\divx u+\divx(\rho u)]dx
\\
&\le C\|\dxa\tdx u\|_{L^2_x}^2+C\|u\|_{H^3_x}(\|\tdx\rho\|_{H^2_x}^2+\|\dxa\tdx u\|_{L^2_x}^2),
\\
J_5&
=\intr\dxa\tdx\rho\cdot\dxa\left\{-(u\cdot\nabla_x)u-\frac{\rho}{1+\rho}\ddx u-\[\frac{p'(1+\rho)}{1+\rho}-\gamma\]\tdx\rho
\right.\\&\left.\quad-\frac{\rho}{1+\rho}(\tau_bf-u)-\frac{\tau_afu}{1+\rho}\right\}dx
\\&\le \|(\tau_af,\tau_bf,\rho,u)\|_{H^3_x}\(\|\dxa\tdx\rho\|_{L^2_x}^2+\|(\tdx\rho,\tdx u,\ddx u)\|_{H^2_x}^2\),
\\
J_6&=-\intr\dxa\tdx\tau_af\cdot\dxa(v\cdot\tdx P_1f,v\chi_0)dx+\intr\dxa\tdx\tau_af\cdot\dxa[(\tdx\Phi+u)\tau_af]dx\\
&\le \epsilon\|\dxa\tdx\tau_af\|_{L^2_x}^2+\frac{C}{\epsilon}\|\dxa\tdx P_1f\|_{L^2_\sigma(L^2_x)}^2+C\|(\tdx\Phi,u)\|_{H^3_x}\(\|\dxa\tdx\tau_af\|_{L^2_x}^2+\|\tdx\tau_af\|_{H^2_x}^2\).
\emas
Combined with \eqref{assume}, summing up the results above and choosing $\eps>0$ small enough yields \eqref{geine4}.
\end{proof}

\begin{thm}\label{ztcz}
There exists a small constant $\delta_0>0$ 
such that if the initial data $U_0$ satisfies $E_3(U_0)\le\delta_0$, then the Cauchy problem for the VPFP/NSP system admits a unique global solution $U=(f,\rho,u)$ satisfying
$$
E_3(U)+ \intt D_3(U)ds\le CE_3(U_0).
$$
\end{thm}

\begin{proof}
By summing $\eqref{geine2}+\eqref{geine3}+A_1[A_0\eqref{geine1}+\eqref{geine4}]$, where $A_0,A_1>0$   satisfy $A_0<\frac\gamma{2C_1}$ and $ A_1<\frac{q_0}{q_1(A_0+1)+2}$ with $C_1>0$ given by \eqref{geine1}, we obtain
\bma
&\nnm\Dt \[\frac{1}{2}\(\|f\|_{L^2_v(H^3_x)}^2+\gamma\|\rho\|_{H^3_x}^2+\| u\|_{H^3_x}^2+\|\tdx\Phi\|_{H^3_x}^2\)+A_0A_1\|\tdx\Phi\|_{H^2_x}^2\]
\\
&\nnm\quad+\Dt A_1\sum\limits_{|\alpha|\le 2}\[-A_0\intr\dxa\tdx\Phi\cdot\dxa(\tau_bf-u)dx+\intr(\dxa\tdx\tau_af\cdot\dxa\tau_bf+\dxa\tdx\rho\cdot\dxa u)dx\]
\\
&\nnm\quad+[q_0-q_1A_1(A_0+1)-2A_1]\sum\limits_{|\alpha|\le 3}(\|\dxa P_1f\|_{L^2_\sigma(L^2_x)}^2+\|\dxa\nabla_xu\|_{L^2_x}^2+\|\dxa(\tau_bf-u)\|_{L^2_x}^2)
\\
&\nnm\quad+A_1\[\(\frac\gamma2-C_1A_0\)\|\tdx\rho\|_{H^2_x}^2
+q_0\|(\tdx\tau_af,\ddx\Phi)\|_{H^2_x}^2+ A_0\|\tdx\Phi\|_{H^2_x}^2\]
\\
&\le C\sqrt{E_3(U)}D_3(U).\label{en1}
\ema
Note that
\be\label{macrovis}
\sum\limits_{|\alpha|\le 3}( \|\dxa\nabla_xu\|_{L^2_x}^2+\|\dxa(\tau_bf-u)\|_{L^2_x}^2)\ge \frac12\sum\limits_{|\alpha|\le 2}\|\dxa\tdx\tau_b f\|^2_{L^2_x}.
\ee
Therefore, there exists a constant $ q_2>0$ such that
$$
\Dt\mathcal{E}_3(U)+ q_2D_3(U)\le C\sqrt{E_3(U)}D_3(U),
$$
where
\bmas
\mathcal{E}_3(U)=&\frac{1}{2} \(\|f\|_{L^2_v(H^3_x)}^2+\gamma\|\rho\|_{H^3_x}^2+\| u\|_{H^3_x}^2+\|\tdx\Phi\|_{H^3_x}^2\)+A_0A_1\|\tdx\Phi\|_{H^2_x}^2
\\
&\nnm+\sum\limits_{|\alpha|\le 2}A_1\[-A_0\intr\dxa\tdx\Phi\cdot\dxa(\tau_bf-u)dx+\intr(\dxa\tdx\tau_af\cdot\dxa\tau_bf+\dxa\tdx\rho\cdot\dxa u)dx\].
\emas
By using Cauchy-Schwartz's inequality, for sufficiently small $A_0,A_1>0$, we have $\frac14E_3(U)\le \mathcal{E}_3(U)\le E_3(U)$, which implies there exists a constant $q_3>0$ such that
\be\label{ztcubds}
\Dt\mathcal{E}_3(U)+ q_3D_3(U)\le 0.
\ee
Thus, we close the priori assumption \eqref{assume}. The global solution for the Cauchy problem for the VPFP/NSP system follows from the local existence Theorem \ref{local}, the uniform energy estimate and the continuity argument.

Let
$$
H_N(U)=\sum\limits_{1\le|\alpha|\le N}\(\|\dxa f\|_{L^2_{x,v}}^2+\gamma\|\dxa\rho\|_{L^2_x}^2+\|\dxa u\|_{L^2_x}^2\)+\sum\limits_{ |\alpha|\le N}\|\dxa\tdx\Phi\|_{L^2_x}^2 +\|\tau_bf-u\|_{L^2_x}^2.
$$
Similarly, we can establish the corresponding higher-order energy estimate. Subtracting $\eqref{ge2}_2$ from $\eqref{introe2}_3$, taking the inner product of the result with $\tau_bf-u$, there exists a constant $q_1>0$ such that
\bma
&\frac12\Dt\(\|\tau_bf-u\|_{L^2_x}^2+\|\tdx\Phi\|_{L^2_x}^2\)+\|\tau_bf-u\|_{L^2_x}^2\nnm
\\&\le q_1\(\|\tdx P_1f\|_{L^2_\sigma(L^2_x)}^2+\|\tdx(\tau_af,\rho)\|_{L^2_x}^2+\|\Delta_x u\|^2_{L^2_x}\)+C\|\rho\|_{H^1_x}\|(\tdx\Phi,\tdx u)\|_{L^2_x}^2.\label{geine5}
\ema
Summing $\eqref{geine3}+B_1[B_0\eqref{geine1}+\eqref{geine5}]$, adding $B_1\eqref{geine4}$ over $1\le|\alpha|\le2$, with $B_0,B_1>0$ satisfying $B_0<\min\{\frac\gamma{2C_1},\frac1{q_1}\}$ and $B_1<\frac{q_0}{q_1(B_0+2)+2}$ where $C_1>0$ is given by \eqref{geine1}, we obtain
\bmas
&\Dt\frac{1}{2}\sum\limits_{1\le|\alpha|\le 3}\(\|\dxa f\|_{L^2_{x,v}}^2+\gamma\|\dxa\rho\|_{L^2_x}^2+\|\dxa u\|_{L^2_x}^2+\|\dxa\tdx\Phi\|_{L^2_x}^2\)
\\
&\quad+\Dt B_1B_0\sum\limits_{ |\alpha|\le 2}\bigg[\|\dxa\tdx\Phi\|_{L^2_x}^2 -\intr\dxa\tdx\Phi\cdot\dxa(\tau_bf-u)dx\bigg]
\\
&\quad+\Dt B_1\sum\limits_{1\le|\alpha|\le 2} \intr\(\dxa\tdx\tau_af\cdot\dxa\tau_bf+\dxa\tdx\rho\cdot\dxa u\)dx
\\
&\quad+\Dt B_1\(\|\tau_bf-u\|_{L^2_x}^2+\|\tdx\Phi\|_{L^2_x}^2\)+B_1B_0 \|\tdx\Phi\|_{H^1_x}^2 + B_1(1-q_1B_0)\|\tau_bf-u\|_{L^2_x}^2
\\
&\quad+ [ q_0-q_1B_1(B_0+2)-2B_1]\sum\limits_{1\le|\alpha|\le 3}\(\|\dxa P_1f\|_{L^2_\sigma(L^2_x)}^2+\|\dxa\nabla_xu\|_{L^2_x}^2+\|\dxa(\tau_bf-u)\|_{L^2_x}^2\)
\\
&\quad+\sum\limits_{1\le|\alpha|\le 2}B_1\[\(\frac\gamma2-C_1B_0\)\|\dxa\tdx\rho\|_{L^2_x}^2+ q_0\|\dxa(\tdx\tau_af,\ddx\Phi)\|_{L^2_x}^2+B_0\|\dxa\tdx\Phi\|_{L^2_x}^2\]
\\
&\le C\|(\tau_af,\tau_bf,\rho,u,\tdx\Phi)\|_{H^3_x}\(\|\tdx P_0f\|_{L^2_v(H^2_x)}^2+\|\tdx(\rho,u)\|_{H^2_x}^2+B_1\|\tdx\Phi\|_{L^2_x}^2\)+q_1\|\tdx(\tau_af,\rho)\|_{L^2_x}^2.
\emas
Making use of \eqref{assume} and \eqref{macrovis}, there exists a constant $ d>0$ such that
\be\label{en2}
\Dt\mathcal{H}_3(U)+ dH_3(U)\le C\(\|\tdx P_0f\|_{L^2_v(L^2_x)}^2+\|\tdx(\rho,u)\|_{L^2_x}^2 \),
\ee
where
\bmas
\mathcal{H}_3(U)=&\frac{1}{2}\sum\limits_{1\le|\alpha|\le 3}\(\|\dxa f\|_{L^2_{x,v}}^2+\gamma\|\dxa\rho\|_{L^2_x}^2+\|\dxa u\|_{L^2_x}^2+\|\dxa\tdx\Phi\|_{L^2_x}^2\)
\\&+B_1B_0\sum\limits_{|\alpha|\le 2}\bigg[\|\dxa\tdx\Phi\|_{L^2_x}^2-\intr\dxa\tdx\Phi\cdot\dxa(\tau_bf-u)dx\bigg]
\\&+B_1\sum\limits_{1\le|\alpha|\le 2}\intr\(\dxa\tdx\tau_af\cdot\dxa\tau_bf+\dxa\tdx\rho\cdot\dxa u\)dx+B_1\(\|\tau_bf-u\|_{L^2_x}^2+\|\tdx\Phi\|_{L^2_x}^2\).
\emas
The proof of theorem is completed.
\end{proof}
\subsection{Optimal Time-Decay Rates of Nonlinear VPFP/NSP}
Consider the following Fokker-Planck equation with damping:
\be\label{znfp}
\dt g+v\cdot\tdx g-Lg+2g=0.
\ee
Denote $A_1=L-v\cdot\tdx-2$. Then, we have
\begin{lem}[\cite{vpfp2}]\label{green1}
For any $\varsigma\ge0$, if the function $ g_0(x,v)\in L^\infty_{v,\varsigma}(L^2_x)$, then
\bma
\| e^{tA_1}g_0\|_{L^\infty_{v,\varsigma}(L^2_x)}&\le C\| g_0\|_{L^\infty_{v,\varsigma}(L^2_x)}e^{-2t},\label{gfee1}
\\
\|\tdv e^{tA_1}g_0\|_{L^\infty_{v,\varsigma}(L^2_x)}&\le C\|g_0\|_{L^\infty_{v,\varsigma}(L^2_x)}\(1+\frac{1}{\sqrt{t}}\)e^{-2t}.\label{gfee2}
\ema
\end{lem}

\begin{lem}[\cite{vpfp2}]\label{green2}
Let $\varsigma\ge0$ and $c_0>0$. If the functions $F_i(t,x,v)$, $i=0,1,2,3$ satisfy
\bmas
\|F_k(t,x,v)\|_{L^\infty_{v,\varsigma}(L^2_x)}&\le Ct^{-\frac{k}{2}}(1+t)^{-c_0},~k=0,1,
\\
\|F_2(t,x,v)\|_{L^2_{x,v}}&\le C(1+t)^{-c_0},
\emas
then
\bma
\left\|\intt e^{(t-s)A_1}F_k(s)ds\right\|_{L^\infty_{v,\varsigma+1}(L^2_x)}&\le C(1+t)^{-c_0},\label{gfee3}
\\
\left\|\intt \tdv e^{(t-s)A_1}F_k(s)ds\right\|_{L^\infty_{v,\varsigma}(L^2_x)}&\le C(1+t)^{-c_0},\label{gfee4}
\\
\left\|\intt e^{(t-s)A_1}F_2(s)ds\right\|_{L^\infty_{v,0}(L^2_x)}&\le C(1+t)^{-c_0}.\label{gfee5}
\ema
\end{lem}
With Lemmas \ref{green1} and \ref{green2}, we are ready to prove Theorem \ref{introth2} as follows.
\begin{proof}[\underline{\textbf{Proof of Theorem \ref{introth2}}}]
Suppose that $U=(f,\rho,u)\in L^2_v(H^3_x)\times(H^3_x)^2$ and $\tdx\Phi\in H^3_x$ to the solution to the VPFP/NSP system \eqref{introe2}. In fact, according to the Duhamel principle, we can represent the solution in terms of the semigroup $e^{t\mathbbm{M}}$ as
\be\label{duhamel}
U(t)=e^{t\mathbbm{M}}U_0+\int_0^te^{(t-s)\mathbbm{M}}R(s)ds,
\ee
where the nonlinear terms $R(s)=(R_1,0,R_2)(s)$ with
\bmas
R_1&=(\nabla_x\Phi+ u)\cdot\(\frac{1}{2}vf-\nabla_vf\),
\\
R_2&=-(u\cdot\nabla_x)u-\frac{\rho}{1+\rho}\ddx u-\[\frac{p'(1+\rho)}{1+\rho}-\gamma\]\tdx\rho-\frac{\rho}{1+\rho}(\tau_bf-u)-\frac{\tau_afu}{1+\rho}.
\emas
Define a functional $Q(t)$ of the solution $U$ for any $t>0$ by
\bmas
Q(t)=
\sup\limits_{0\le s\le t}\bigg\{&\[\|f\|_{L^\infty_{v,3}(L^2_x)}+\|\tdv f\|_{L^\infty_{v,2}(L^2_x)}\sqrt{s}(1+\sqrt{s})^{-1}+\|(\rho,u)\|_{L^2_x}\](1+s)^{\frac{3}{4}}
\\&+\sum\limits_{|\alpha|=1}^3\[\|\dxa f\|_{L^\infty_{v,3}(L^2_x)}+\|\dxa\tdv f\|_{L^\infty_{v,2}(L^2_x)}\sqrt{s}(1+\sqrt{s})^{-1}+\|\dxa(\rho,u)\|_{L^2_x}\](1+s)^\frac54
\\&+[\|\tdx\Phi\|_{L^2_x}+\|\tau_bf-u\|_{L^2_x}](1+s)^\frac54\bigg\}.
\emas
We claim that it holds under the assumptions of Theorem \ref{introth2} that
\be\label{nldcre3}
Q(t)\le C\delta_0.
\ee
It is easy to verify that the estimate \eqref{nldcre1} follows from \eqref{nldcre3}. By direct calculation and using the Nirenberg's and H\"older's inequalities, we obtain the following estimates:
\bmas
\|R_1\|_{L^2_v(L^1_x)}&\le CQ(t)^2\(1+\frac{1}{\sqrt{t}}\)(1+t)^{-\frac{3}{2}},~~\|R_2\|_{L^1_x}\le CQ(t)^2(1+t)^{-\frac{3}{2}},
\\\|\dxa R_1\|_{L^2_{x,v}}&\le CQ(t)^2\(1+\frac{1}{\sqrt{t}}\)(1+t)^{-\frac{5}{4}},~~\|\dxa R_2\|_{L^2_x}\le CQ(t)^2(1+t)^{-\frac{5}{4}},~~|\alpha|=0,1.
\emas
Starting with \eqref{orl1} and the assumption \eqref{nldcre3},  we obtain
\bma
\|\dxa U\|_{Z^2}&\nnm\le C(1+t)^{-(\frac{3}{4}+\frac{|\alpha|}{2})}(\|\dxa U_0\|_{Z^2}+\|U_0\|_{Z^1})
\\&\nnm\quad+C\intt(1+t-s)^{-(\frac{3}{4}+\frac{|\alpha|}{2})}\(\|\dxa R(s)\|_{Z^2}+\|R(s)\|_{Z^1}\)ds
\\&\nnm\le C\delta_0(1+t)^{-(\frac{3}{4}+\frac{|\alpha|}{2})}+C\intt(1+t-s)^{-(\frac{3}{4}+\frac{|\alpha|}{2})}\(1+\frac{1}{\sqrt{s}}\)(1+s)^{-\frac{5}{4}}Q(t)^2ds
\\&\label{nldcre6}\le C\delta_0(1+t)^{-(\frac{3}{4}+\frac{|\alpha|}{2})}+C(1+t)^{-(\frac{3}{4}+\frac{|\alpha|}{2})}Q(t)^2,~~|\alpha|=0,1.
\ema
For $\dxa U,\,|\alpha|=2,3$, the electric field $\tdx\Phi$ and $\tau_bf-u$, by using the energy inequality \eqref{en2}, we have
$$
\Dt \mathcal{H}_3(U)+dH_3(U)\le C\(\|\tdx P_0f\|_{L^2_{x,v}}^2+\|\tdx(\rho,u)\|_{L^2_x}^2 \).
$$
Thus, we obtain
\bma
H_3(U)&\nnm\le Ce^{-dt}H_3(U_0)+C\intt e^{-d(t-s)}\(\|\tdx P_0f\|_{L^2_{x,v}}^2+\|\tdx(\rho,u)\|_{L^2_x}^2\)ds
\\
&\nnm\le C\delta_0^2e^{-dt}+C[\delta_0+Q(t)^2]^2\intt e^{-d(t-s)}(1+s)^{-\frac52}ds
\\
&\le C[\delta_0+Q(t)^2]^2(1+t)^{-\frac52}.\label{highclose}
\ema
By $\eqref{introe2}_1$, we have
\be\label{A1eq}
\dt f-A_1f=2f+R_1+R_3,
\ee
where $A_1=L-2-v\cdot\tdx$ and
$$
R_3=(\tdx\Phi+u)\cdot v\chi_0.
$$
By using the Nirenberg's and H\"older's inequalities, we obtain
\bma
\|R_1+R_3\|_{L^\infty_{v,2}(L^2_x)} &\le C[\delta_0+Q(t)^2]\(1+\frac{1}{\sqrt{t}}\)(1+t)^{-\frac{3}{4}},\label{nldcre7a}
\\
\|\dxa (R_1+R_3)\|_{L^\infty_{v,2}(L^2_x)} &\le C[\delta_0+Q(t)^2]\(1+\frac{1}{\sqrt{t}}\)(1+t)^{-\frac{5}{4}},~~|\alpha|=1,2,3.\label{nldcre7b}
\ema
By \eqref{A1eq}, we can represent $f$ as
\be\label{fA1}
f=e^{tA_1}f_0+\intt e^{(t-s)A_1}(2f+R_1+R_3)ds.
\ee
By Lemma \ref{green1}, it holds that
\be\label{nldcre11}
\|\dxa(e^{tA_1}f_0)\|_{L^\infty_{v,3}(L^2_x)}\le C\delta_0e^{-2t},~~|\alpha|\le3.
\ee
By \eqref{nldcre6} and \eqref{highclose}, it holds that
\bmas
\|f(t)\|_{L^2_{x,v}}&\le C[\delta_0+Q(t)^2](1+t)^{-\frac{3}{4}},
\\
\|\dxa f(t)\|_{L^2_{x,v}}&\le C[\delta_0+Q(t)^2](1+t)^{-\frac54},~~|\alpha|=1,2,3,
\emas
which together with \eqref{gfee5} and \eqref{nldcre7a}-\eqref{nldcre7b} give
\bma
\left\|\intt e^{(t-s)A_1}(2f+R_1+R_3)ds\right\|_{L^\infty_{v,0}(L^2_x)}&\le C[\delta_0+Q(t)^2](1+t)^{-\frac{3}{4}},
\\
\left\|\intt \dxa[e^{(t-s)A_1}(2f+R_1+R_3)]ds\right\|_{L^\infty_{v,0}(L^2_x)}&\le C[\delta_0+Q(t)^2](1+t)^{-\frac{5}{4}},~~|\alpha|=1,2,3.\label{a1juan1}
\ema
Thus, it follows from \eqref{fA1}-\eqref{a1juan1} that
\bmas
\|f(t)\|_{L^\infty_{v,0}(L^2_x)}&\le C[\delta_0+Q(t)^2](1+t)^{-\frac{3}{4}},
\\
\|\dxa f(t)\|_{L^\infty_{v,0}(L^2_x)}&\le C[\delta_0+Q(t)^2](1+t)^{-\frac54},~~|\alpha|=1,2,3.
\emas
Then, by using \eqref{gfee3} and induction, we obtain
\bma\label{nldcre12a}
\|f(t)\|_{L^\infty_{v,3}(L^2_x)}&\le C[\delta_0+Q(t)^2](1+t)^{-\frac{3}{4}},
\\\label{nldcre12b}
\|\dxa f(t)\|_{L^\infty_{v,3}(L^2_x)}&\le C[\delta_0+Q(t)^2](1+t)^{-\frac{5}{4}},~~|\alpha|=1,2,3.
\ema
Similarly, by \eqref{fA1}, we can represent $\tdv f$ as
\be\label{dvfA1}
\tdv f=\tdv e^{tA_1}f_0+\intt \tdv e^{(t-s)A_1}(2f+R_1+R_3)ds.
\ee
By Lemma \ref{green1}, it holds that
\be\label{nldcre14}
\|\dxa(\tdv e^{tA_1}f_0)\|_{L^\infty_{v,2}(L^2_x)}\le C\delta_0\(1+\frac{1}{\sqrt{t}}\)e^{-2t},~~|\alpha|\le3.
\ee
Combing \eqref{gfee4}, \eqref{nldcre7a}-\eqref{nldcre7b} and \eqref{nldcre12a}-\eqref{nldcre12b}, we have
\bma\label{dva1juan1}
\left\|\intt\tdv e^{(t-s)A_1}(2f+R_1+R_3)ds\right\|_{L^\infty_{v,2}(L^2_x)}&\le C[\delta_0+Q(t)^2](1+t)^{-\frac{3}{4}},
\\
\left\|\intt \dxa[\tdv e^{(t-s)A_1}(2f+R_1+R_3)]ds\right\|_{L^\infty_{v,2}(L^2_x)}&\le C[\delta_0+Q(t)^2](1+t)^{-\frac{5}{4}},~|\alpha|=1,2,3.
\ema
Thus, it follows from \eqref{dvfA1} that
\bma
\|\tdv f(t)\|_{L^\infty_{v,2}(L^2_x)}&\le C[\delta_0+Q(t)^2]\(1+\frac{1}{\sqrt{t}}\)(1+t)^{-\frac{3}{4}},\label{nldcre15a}
\\
\|\dxa\tdv f(t)\|_{L^\infty_{v,2}(L^2_x)}&\le C[\delta_0+Q(t)^2]\(1+\frac{1}{\sqrt{t}}\)(1+t)^{-\frac{5}{4}},~~|\alpha|=1,2,3.\label{nldcre15b}
\ema
Combining \eqref{nldcre6}-\eqref{highclose}, \eqref{nldcre12a}-\eqref{nldcre12b} and \eqref{nldcre15a}-\eqref{nldcre15b}, we obtain
$$
Q(t)\le C\delta_0+CQ(t)^2,
$$
which together with $\delta_0>0$ small enough leads to \eqref{nldcre3}.

Indeed, the convergence rates of Theorem \ref{introth2} can be shown to be optimal. For instance, by  \eqref{duhamel} and Theorem \ref{linearoptdr}, we obtain
\bmas
\|U(t)\|_{Z^2}&\ge \|e^{t\M}U_0\|_{Z^2}-\intt\|e^{(t-s)\mathbbm{M}}R(s)\|_{Z^2}ds
\ge C_1d_0(1+t)^{-\frac{3}{4}}-C_2\delta_0^2(1+t)^{-\frac{3}{4}}.
\emas
Thus, estimates \eqref{opnldr} hold for sufficiently large $t>0$ and sufficiently small $\delta_0>0$. The proof of theorem is completed.
\end{proof}

\section{Appendix }
\subsection{Useful Lemmas }
Let $H$ be an $n$ dimension complex inner product space equipped with inner product $(\cdot,\cdot)_H$. Let $A$ be the operator on $H$, and $A^*$ be the adjoint operator of $A$ in and only if $(AX,y)_H=(X,A^*y)_H$, $\forall X,Y\in H$. If $\lambda_j$ is an eigenvalue of $A$, $\overline{\lambda_j}$ is the eigenvalue of $A^*$. Then, we have the following Lemma.

\begin{lem} [\cite{Pazy}]\label{S_1-1}
Let $A$ be a densely defined closed linear operator on the Hilbert space $H$. If both $A$ and its adjoint operator $A^*$ are dissipative, then $A$ is the
infinitesimal generator of a $C_0$-semigroup  on $H$.
\end{lem}
\begin{lem}[\cite{Pazy}]\label{semigroup-1}
Let $A$ be the infinitesimal generator of a $C_0$-semigroup $T(t)$ satisfying $\|T(t)\|\le Me^{\kappa t}$. Then, it holds for $f\in D(A^2)$ and $\sigma>\max(0,\kappa )$ that
$$
T(t)f =\frac1{2\pi \i}\int^{\sigma+\i\infty}_{\sigma-\i\infty} e^{\lambda t}(\lambda-A)^{-1}f d\lambda.
$$
\end{lem}
\begin{lem}[\cite{Pazy}] \label{dense}
Let $A$ be the infinitesimal generator of the $C_0$ semigroup
$T(t)$. If $D(A^n)$ is the domain of $A^n$, then
$\cap^\infty_{n=1}D(A^n)$ is dense in $X$.
\end{lem}

\begin{lem}\label{srlem6}
Let $A=(a_{ij})_{n\times n}$ be a linear operator on $H$. If $A$ has $n$ linearly independent eigenvectors and $A^*=(\overline{a_{ij}})_{n\times n}$, then there exists $n$ eigenvectors $V_i\in H$ of $A$ satisfying $(V_i,\overline{V_j})_H=0$, $i\neq j=1,2,\cdots,n$. Moreover, if $(V_i,\overline{V_i})_H\neq0$, then

(1) the resolvent operator $(\lambda-A)^{-1}$ has the following form
\be\label{sraa2}
(\lambda-A)^{-1}X=\sum\limits_{j=1}^n\frac{1}{\lambda-\lambda_j}\frac{1}{(V_j,\overline{V_j})_H}(X,\overline{V_j})_HV_j,~~X\in H.
\ee

(2) the solution $Y(t)=e^{tA}Y_0$ to the equation $Y'(t) = AY(t)$ can be expressed as
\be\label{sraa16}
e^{tA}Y_0=\sum_{i = 1}^{n}e^{\lambda_i t}\frac{1}{(V_i, \overline{V_i})_H}(Y_0, \overline{V_i})_HV_i,~~Y_0\in H.
\ee
\end{lem}
\begin{proof}
For the $n$ linearly independent eigenvectors $V_i\in H$ of $A$, we have
$$
0=(V_i, A^*\overline{V_j}-\overline{\lambda_j}\cdot\overline{V_j})_H=(AV_i, \overline{V_j})_H-\lambda_j(V_i,\overline{V_j})_H=(\lambda_i-\lambda_j)(V_i,\overline{V_j})_H,
$$
which implies $(V_i,\overline{V_j})_H=0$ when $\lambda_i\neq\lambda_j$. If $V_i$ and $V_j$ are the eigenvectors corresponding to the same eigenvalue, then they can be Schmidt orthogonalized.

Let $(\lambda-A)^{-1}X=\sum\limits_{j=1}^na_iV_i$. Taking the inner product $(\cdot,\cdot)_{H}$ on it
with $\overline{V_i}$, we have
$$
(\lambda-\lambda_j)^{-1}(X,\overline{V_i})_H=a_i(V_i,\overline{V_i})_H,
$$
which proves \eqref{sraa2}. Since $V_1,V_2,\cdots,V_n$ are linearly independent, $\varPhi(t)=(e^{\lambda_1 t}V_1,e^{\lambda_2 t}V_2,\cdots,e^{\lambda_n t}V_n)$ is a fundamental solution set of the ordinary differential equation $Y'(t) = AY(t)$. Therefore, $e^{tA}$ can be expressed as
$$
Y(t)=e^{tA}Y_0=\sum_{i = 1}^{n}a_ie^{\lambda_i t}V_i,\quad Y(0)=Y_0,
$$
where $a_i$ are constants depending on $Y_0$. For $t = 0$, we have $Y_0=\sum_{i = 1}^{n}a_iV_i$. Taking the inner product $(\cdot,\cdot)_{H}$ on it with $\overline{V_i}$, we obtain
$$
(Y_0, \overline{V_i})_H=a_i(V_i, \overline{V_i})_H,
$$
which proves \eqref{sraa16}. The proof of the lemma is completed.
\end{proof}

\begin{lem}[\cite{zhongspvmb}]\label{inver} Let $K_1,K_4$ be the operators on the spaces $X$ and $Y$ respectively, and $K_2,K_3$ be the operators  $Y\to X$ and $X\to Y$ respectively. Let $K$ be a matrix operator on $X\times Y$ defined by
$$
K=\left(\ba K_1 & K_2 \\  K_3 & K_4 \ea\right).
$$
If the norms of $K_1,K_2,K_3$ and $K_4$ satisfy 
$$\|K_1\|<1,\quad \|K_4\|<1,\quad \|K_2\|\|K_3\|<(1-\|K_1\|)(1-\|K_4\|),$$
then the operator $I+K$ is invertible on $X\times Y$.
\end{lem}

\begin{defn}[\cite{dfb2}]
A Weierstrass polynomial in $w$ is a polynomial of the form
$$
P(z,w)=w^d+a_1(z)w^{d-1}+\cdots+a_d(z),\quad a_i(0)=0,\quad z\in\C^{n-1},~w\in\C.
$$
\end{defn}
\begin{thm}[\cite{dfb2},\cite{dfb1}]\label{wei}(Weierstrass Preparation Theorem) Let $z\in\C^{n-1}$ and $w\in\C$. Assume $f(z,w)$ is holomorphic near the origin in $\mathbb{C}^n$ satisfying $f(0,0)=0$, and there exist an analytic $h$ satisfying $h(0)\neq0$ and $d\in\mathbb{N}$ such that $f(0,w)=w^dh(w)$. Then, in some neighborhood of the origin, $f$ can be written uniquely as
$$
f(z,w)=U(z,w) P(z,w) ,
$$
where $P(z,w)$ is a Weierstrass polynomial of degree $d$ in $w$, $U(0,0) \neq 0$, and $P(z,w),U(z,w)$ are analytic.
\end{thm}

\begin{lem}\label{orlem4} 
Let $Y_1(\lambda,\xi)$ be defined by \eqref{sraa7}, $G_k(\xi)~(k=1,2)$ be defined by \eqref{sra2}, $r_0>0$ be given by Lemma \ref{srlem11}, and $\beta =\beta (r_0)>0$ be defined in Lemma \ref{srlem5}. Then there exists a constant $C>0$ such that
\bma
\sup\limits_{|\xi|\ge r_0,y\in\mathbb{R}}\|[I-G_2(\xi)(z_1+iy-G_1(\xi))^{-1}]^{-1}\|&\label{da6}\le C,
\\
\sup\limits_{0<|\xi|<r_0,y\in\mathbb{R}}\|[I-Y_1(z_0+iy,\xi)]^{-1}\|_{\xi,\gamma}&\label{da5}\le C,
\ema
where $z_0\in[-\frac{1}{2},-r_0^2)$ and $z_1\in [-\beta,0) $.
\end{lem}

\begin{proof}
Let $\lambda=z+iy$, $(z,y)\in\mathbb{R}\times\mathbb{R}$, where $z$ satisfies
$$
z\in\(-\frac{1}{2},-r_0^2\),~~|\xi|\le r_0;~~z\in [-\beta,0) ,~~|\xi|> r_0.
$$
From \eqref{egexpand}, we have
$$
 0<-\re\lambda_j(|\xi|)\le r_0^2,~~ |\xi|\le r_0, ~-1\le j\le2.
$$
Thus, by Lemma \ref{srlem5} and Lemma \ref{srlem13}, we  conclude that $\lambda\in\rho(\mathbbm{M}(\xi))$, namely $\lambda-\mathbbm{M}(\xi)$ is invertible. Thus,  we have by \eqref{sra9} and \eqref{sraa8} that $I-G_2(\xi)(\lambda-G_1(\xi))^{-1}$ and $I-Y_1(\lambda,\xi)$ are also invertible.

Firstly, we want to show \eqref{da6}. By Lemma \ref{srlem5}, there exist $r_1,R>0$ such that for $|\xi|>r_1$ or $|y|\ge R$,
\be
\|[I-G_2(\xi)(\lambda-G_1(\xi))^{-1}]^{-1}\|\le2. \label{bound}
\ee
Thus, we only need to prove \eqref{da6} when $|\xi|\in[r_0,r_1]$ and $|y|\le R$. For this, we first show that $H(\lambda,\xi)=[I-G_2(\xi)(\lambda-G_1(\xi))^{-1}]^{-1}$ is continuous with respect to $(\lambda,\xi)$. By using
$$
H(\lambda,\xi)-H(\lambda_0,\xi_0)
=H(\lambda,\xi)[G_2(\xi)(\lambda-G_1(\xi))^{-1}-G_2(\xi_0)(\lambda_0-G_1(\xi_0))^{-1}]H(\lambda_0,\xi_0),
$$
we only need to show $\|[G_2(\xi)(\lambda-G_1(\xi))^{-1}-G_2(\xi_0)(\lambda_0-G_1(\xi_0))^{-1}]\|\to0$ as $(\lambda,\xi)\to(\lambda_0,\xi_0)$. Indeed,
\bmas
& \|[G_2(\xi)(\lambda-G_1(\xi))^{-1}-G_2(\xi_0)(\lambda_0-G_1(\xi_0))^{-1}]\|
\\
&\le\|G_2(\xi)[(\lambda-G_1(\xi))^{-1}-(\lambda_0-G_1(\xi))^{-1}]\|+\|[G_2(\xi)-G_2(\xi_0)](\lambda_0-G_1(\xi))^{-1}\|
\\
&\quad+\|G_2(\xi_0)[(\lambda_0-G_1(\xi))^{-1}-(\lambda_0-G_1(\xi_0))^{-1}]\|:=I_1+I_2+I_3.
\emas
For $I_k$, $k=1,2,3,$ it holds that
$$I_1=|\lambda-\lambda_0|\|G_2(\xi) (\lambda-G_1(\xi))^{-1} (\lambda_0-G_1(\xi))^{-1} \|\to0,\quad \lambda\to\lambda_0,$$
$$I_2\le\|G_2(\xi)-G_2(\xi_0)\| \|(\lambda_0-G_1(\xi))^{-1}\|\to0,\quad \xi\to\xi_0,$$
and
$$I_3\le \|G_2(\xi_0)\|\|(\lambda_0-G_1(\xi))^{-1}\|\|G_1(\xi)-G_1(\xi_0)\| \|(\lambda_0-G_1(\xi_0))^{-1}\|\to0,\quad \xi\to\xi_0.$$
Thus, $H(\lambda,\xi)$ is continuous with respect to $(\lambda,\xi)$. Since $\Omega=\{(\lambda,\xi)\mid \re \lambda=-\beta , |\im \lambda|\le R, |\xi|\in[r_0,r_1]\}$ is a compact set, we can conclude that
$\|[I-G_2(\xi)(\lambda-G_1(\xi))^{-1}]^{-1}\|$ is bounded for $(\lambda,\xi)\in \Omega$. This together with \eqref{bound} implies \eqref{da6}.

Next, we will show \eqref{da5}. By \eqref{sraa8}, there exists $R>0$ such that for $|y|\ge R$,
$$
\|[I-Y_1(\lambda,\xi)]^{-1}\|_{\xi,\gamma}\le2.
$$
Thus, it's sufficient to prove \eqref{da6} for $|y|\le R$. If \eqref{da5} doesn't hold for $|y|\le R$, there are $\xi_n=s_n\omega$, $s_n=|\xi_n|$, $\lambda_n=z_0+iy_n$ with $0<|\xi_n|\le r_0$, $|y_n|\le R$, and $U_n,V_n$ with $\|U_n\|_{\xi_n,\gamma}\to0~(n\to\infty)$ and $\|V_n\|_{\xi_n,\gamma}=1$ such that
$$
[I-Y_1(\lambda_n,{\xi_n})]^{-1}U_n=V_n.
$$
This gives
$$
U_n=V_n-G_5({\xi_n})(\lambda_n-G_4({\xi_n}))^{-1}P_BV_n-G_5({\xi_n})(\lambda_n -G_3({\xi_n}))^{-1}P_AV_n,
$$
and then
\bma
P_AU_n&=P_AV_n-G_5({\xi_n})(\lambda_n-G_4({\xi_n}))^{-1}P_BV_n,\label{da7}
\\
P_BU_n&=P_BV_n-G_5({\xi_n})(\lambda_n -G_3({\xi_n}))^{-1}P_AV_n.\label{da8}
\ema
Let \be\label{da9}
W_n=(\lambda_n -G_3({\xi_n}))^{-1}P_AV_n\Longleftrightarrow P_AV_n=\lambda_nW_n-G_3({\xi_n})W_n.
\ee
By substituting \eqref{da8}  to \eqref{da7}, we btain
$$
\lambda_n W_n-G_3(\xi_n)W_n-G_5(\xi_n)(\lambda_n -G_4(\xi_n))^{-1}G_5W_n=P_AU_n+G_5(\xi_n)(\lambda_n -G_4(\xi_n))^{-1}P_BU_n.
$$
Since $\|U_n\|_{\xi_n,\gamma}\to0~(n\to\infty)$, it follows that
\be\label{da10}
\lim\limits_{n\to\infty}\|\lambda_n W_n-G_3(\xi_n)W_n-G_5(\xi_n)(\lambda_n -G_4(\xi_n))^{-1}G_5(\xi_n)W_n\|_{\xi_n,\gamma}=0.
\ee
Since for any $U\in Z_{\xi,\gamma}$,
$$\min_{-1\le j\le 3}|\lambda-\tilde{\eta}_j|^{-1}\|U\|_{\xi,\gamma}\le \|(\lambda -G_3({\xi}))^{-1}U\|_{\xi,\gamma}\le \max_{-1\le j\le 3}|\lambda-\tilde{\eta}_j|^{-1}\|U\|_{\xi,\gamma},$$
it follows that there are two constants $C_0,C_1>0$ such that
$$
C_0\le\|W_n\|_{\xi_n,\gamma}\le C_1.
$$
We can write $W_n=(C_0^n\chi_0,\rho_n,u_n)$. Then, there exists a subsequence $\{n_j\}\subset \{n\}$ such that $C_0^{n_j}\to A_0$, $\rho_{n_j}\to A_1$, $u_{n_j}\to A_2$, $\frac{1}{s_{n_j}}(C_0^{n_j}-\rho_{n_j})\to A_3$, $\xi_{n_j}\to\xi_0=s_0\omega$, $y_{n_j}\to y_0$ with $(A_0,A_1,A_2,A_3)\neq0$.  If $s_0\neq0$, then it follows from \eqref{da10} that $W_0=(A_0\chi_0,A_1,A_2)\neq0$  satisfies
$$\lambda_0 W_0-G_3(\xi_0)W_0-G_5(\xi_0)(\lambda_0 -G_4(\xi_0))^{-1}G_5(\xi_0)W_0=0,$$
 which together with Remark \ref{eigen1} implies that $W_0\ne 0$ is the eigenvector of $\mathbbm{M}(\xi_0)$ corresponding to the eigenvalue $\lambda_0=z_0+iy_0$.
 This is a contradiction. If $s_0=0$, then $A_0-A_1=0$, and $(A_0,A_1,A_2,A_3)\neq0$ satisfies the following equations with $\lambda_0=z_0+iy_0$ that
\be\label{da16}\left\{\ball
&\lambda_0A_0=\lambda_0A_1=0,
\\&(\lambda_0+1)A_2-i\omega[1+R_{11}(\lambda_0,0)]A_3+R_{11}(\lambda_0,0)\omega(\omega\cdot A_2)+R_{22}(\lambda_0,0)[A_2-\omega(\omega\cdot A_2)]=0,
\\&[\lambda_0-R_{11}(\lambda_0,0)]A_3-i[1+R_{11}(\lambda_0,0)](\omega\cdot A_2)=0,
\eall\right.
\ee
which implies $A_0=A_1=0$. Thus, if $(0,0,(\omega\cdot A_2),A_3)\neq0$, then the eigenvalue problem satisfies
$$
\lambda_0^2+\lambda_0+1+R_{11}(\lambda_0,0)=\frac{\lambda_0}{\lambda_0+1}(\lambda_0^2+2\lambda_0+2)=0,
$$
and if $(0,0,[A_2-(\omega\cdot A_2)],0)\neq0$, then
$$
\lambda_0+1+R_{22}(\lambda_0,0)=\frac{1}{\lambda_0+1}\lambda_0(\lambda_0+2)=0.
$$
Thus, $\lambda_0=-1\pm i,0$ or $\lambda_0=-2,0$, respectively, which leads to a contradiction.    The proof of the lemma is completed.
\end{proof}
\subsection{Local Existence}
We now consider the local existence of the solution to the VPFP/NSP system \eqref{introe2}. We denote $U^n(t,x,v):=(f^n,\rho^n,u^n)(t,x,v)$, and consider the following iterating system motivated by \cite{gyvpb,gyvmb,lxliu3,livfpnsbd,ns1,zhaovpb}. Let $U^0=(f^0,\rho^0,u^0)=0$, $\Phi^0=0$ and
\be\label{locale1}
\left\{\ball
&\partial_tf^{n+1}+v\cdot \nabla_xf^{n+1}-Lf^{n+1}= (\nabla_x\Phi^n+ u^n)\cdot\(\frac{1}{2}vf^n-\nabla_vf^n+v\chi_0\),
\\&\partial_t\rho^{n+1}+u^n\cdot\tdx\rho^{n+1}+(1+\rho^n)\divx u^{n+1}=0,
\\&\ball\dt u^{n+1}-\frac{1}{1+\rho^n}\ddx u^{n+1}=&-u^n\cdot\tdx u^n-\frac{p'(1+\rho^n)}{1+\rho^n}\tdx\rho^n+\frac{1}{1+\rho^n}(\tau_bf^n-u^n)
\\&-\frac{1}{1+\rho^n}u^n\tau_af^n-\tdx\Phi^n,\eall
\\&\ddx\Phi^{n+1}=\tau_af^{n+1}-\rho^{n+1},
\eall\right.
\ee
with $n\ge 0$ and the initial data
$$
U^{n+1}(0,x,v) =U_0=(f_0(x,v),\rho_0(x),u_0(x)).
$$

We introduce the solution space as follows
$$
X_N(0,T;M_0):=\left\{\ball
&f\in C([0,T],L^2_v(H^N_x)),~(\rho,u,\tdx\Phi)\in C([0,T],H^N_x),
\\&\E_N(U):=\|f\|_{L^2_v(H^N_x)}^2+\|(\rho,u,\tdx\Phi)\|_{H^N_x}^2,~\sup_{0\le t\le T}\E_N(U(t))\le 2M_0.
\eall\right\}.
$$
\begin{thm}\label{local}(Local existence) Let $N=3$. There exist sufficiently small constants $M_0$ and $T^*>0$ such that, for $t\in [0,T^*]$ and
$\E_3(U_0)\le M_0$, the Cauchy problem for the VPFP/NSP system \eqref{introe2} and \eqref{initial} admits a unique local solution $U(t,x,v)$ satisfying
\be\label{localenergy}
U(t,x,v)\in X_3(0,T^*;M_0).
\ee
\end{thm}
\begin{proof}
First, we claim that   $U^n\in X_3(0,T^*;M_0)$. By taking the inner product between $\chi_0$ and $\eqref{locale1}_1$, we have
$$\dt\tau_af^{n+1}+\divx\tau_bf^{n+1}=0,$$
and then
\bma\label{dianchang1}
&\nnm-\intr\dxa\tdx\Phi^{n+1}\cdot\dxa(\tau_bf^{n+1}-u^{n+1}) dx=\intr\dxa\Phi^{n+1}\dxa\divx(\tau_bf^{n+1}-u^{n+1}) dx
\\&=\frac{1}{2}\Dt\|\dxa\tdx\Phi^{n+1}\|_{L^2_x}^2+\intr\dxa\Phi^{n+1}\dxa(u^n\cdot\tdx\rho^{n+1}+\rho^n\divx u^{n+1})dx.
\ema
For $n=0$, applying $\dxa$ with $0\le|\alpha|\le N$ to $\eqref{locale1}_1$, $\eqref{locale1}_2$, $\eqref{locale1}_3$, respectively, and then multiplying  them by $\dxa f^1,\dxa\rho^1$ and $\dxa u^1$, respectively, taking integration and summation, using \eqref{dianchang1}, we obtain
\bma\label{u1energy}
&\frac12\Dt\(\|\dxa f^1\|^2_{L^2_{x,v}}+\|\dxa(\rho^1,u^1,\tdx\Phi^1)\|^2_{L^2_x}\)+\intr(-L\dxa f^1,\dxa f^1)dx+\|\dxa\tdx u^1\|_{L^2_x}^2\nnm
\\&=-\intr\dxa\rho^1\dxa\divx u^1dx-\intr\dxa\tdx\Phi^1\cdot\dxa(\tau_bf^1-u^1)dx\nnm
\\&\le \frac12\|\dxa\tdx u^1\|_{L^2_x}^2+\frac12\|\dxa\rho^1\|_{L^2_x}^2+ \|\dxa(\tau_bf^1,u^1,\tdx\Phi^1)\|_{L^2_x}^2.
\ema
Summing up the result above over $|\alpha|\le 3$,
$$
\Dt \E_3(U^1(t))+\mu_0 \D_3(U^1(t))\le C\E_3(U^1(t)),
$$
where
$$
\D_N(U)=  \| P_rf\|_{L^2_{\sigma}(H^N_x)}^2+\| \tdx u\|_{H^N_x}^2.
$$
Integrating the result above over $[0,T^*]$ and taking the supremum in time, we have
$$
\sup\limits_{0\le t\le T^*}\E_3(U^1(t))+ \mu_0 \int_0^{T^*}\D_3(U^1(t))dt \le M_0+CT^*\sup\limits_{0\le t\le T^*}\E_3(U^1(t)).
$$
Thus, we choose $T^*>0$ sufficiently small  such that $CT^*<1/2$, and hence,
$$
\sup\limits_{0\le t\le T^*}\E_3(U^1(t))\le\frac{M_0}{1-CT^*}\le 2M_0.
$$
Moreover, it holds that $\int_0^{T^*}\D_3(U^1(t))dt\le CM_0$. On the other hand, combined with
\bmas
(Lf,f)=-\|\tdv P_rf\|^2-\frac14\|vP_rf\|^2+\frac32\|P_rf\|^2\ge-\|P_rf\|_{L^2_\sigma}^2,
\emas
we have by \eqref{u1energy}
\bmas
&\frac12\Dt\(\|\dxa f^1\|^2_{L^2_{x,v}}+\|\dxa(\rho^1,u^1,\tdx\Phi^1)\|^2_{L^2_x}\)
\\&\ge -\|P_rf\|_{L^2_\sigma}^2-\frac32\|\dxa\tdx u^1\|_{L^2_x}^2-\frac12\|\dxa\rho^1\|_{L^2_x}^2- \|\dxa(\tau_bf^1,u^1,\tdx\Phi^1)\|_{L^2_x}^2,
\emas
and then
$$
\Dt \E_3(U^1(t))\ge-C[\D_3(U^1(t))+\E_3(U^1(t))].
$$
Thus, the continuity of $\E_3(U^1(t))$ follows from
\bma\label{u1lianxu}
\left|\E_3(U^1(t))-\E_3(U^1(s))\right|=\left|\int_s^t\Dtau\E_3(U^1(\tau)d\tau\right|\le C\int_s^tD_3(U^1(\tau))d\tau+CM_0|t-s|\to0,~t\to s.
\ema
Combining the results above, we obtain $U^1\in X_3(0,T^*;M_0)$.

Suppose that $U^{k}\in X_3(0,T^*;M_0)$ and $\int_0^{T^*}\D_3(U^k(t))dt\le CM_0$ 
for $k\le n$. For $k=n+1$, similar to the treatment of \eqref{geine3}, applying $\dxa$ with $0\le|\alpha|\le N$ to $\eqref{locale1}_1$, $\eqref{locale1}_2$, $\eqref{locale1}_3$, respectively, and then multiplying  them by $\dxa f^{n+1},\dxa\rho^{n+1}$ and $\dxa u^{n+1}$, respectively, making use of \eqref{dianchang1}, taking integration and summation, we obtain 
\bmas
&\frac12\Dt\(\|f^{n+1}\|_{L^2_v(H^3_x)}^2+\|(\rho^{n+1},u^{n+1},\tdx\Phi^{n+1})\|_{H^3_x}^2\)+\mu_0\(\|P_rf^{n+1}\|_{L^2_\sigma(H^3_x)}^2+\|\tdx u^{n+1}\|_{H^3_x}^2\)
\\&\le C\(1+\|(\rho^n,u^n,\tdx\Phi^n)\|_{H^3_x}^2\)\(\|f^{n+1}\|_{L^2_v(H^3_x)}^2+\|(\rho^{n+1},u^{n+1},\tdx\Phi^{n+1})\|_{H^3_x}^2\)
\\&\quad+C\|(\rho^n,u^n,\tdx\Phi^n)\|_{H^3_x}^2(1+\|\tau_af^n\|_{H^3_x}^2)+C\|(\tau_af^n,\tau_bf^n,u^n)\|_{H^3_x}^2
\\&\quad+C\|(\rho^n,u^n,\tdx\Phi^n)\|_{H^3_x}^2\(\|P_rf^n\|_{L^2_\sigma(H^3_x)}^2+\|\tdx u^n\|_{H^3_x}^2\),
\emas
and hence
\bma
&\nnm\frac12\Dt \E_3(U^{n+1}(t))+\mu_0\D_3(U^{n+1}(t))
\\
&\le C[1+\E_3(U^n(t))]\E_3(U^{n+1}(t))+C\E_3(U^n(t))[1+C\E_3(U^n(t))
+\D_3(U^n(t))].\label{unenergy}
\ema
Integrating \eqref{unenergy} over $[0,T^*]$ and taking the supremum in time yields
\be\label{intunengry}
[1-CT^*(1+M_0)]\sup\limits_{0\le t\le T^*}\E_3(U^{n+1}(t))\le M_0+CT^*(M_0+M_0^2)+CM_0^2.
\ee
By choosing $T^*>0$ sufficiently small, we obtain
$$
\sup\limits_{0\le t\le T^*}\E_3(U^{n+1}(t))\le 2M_0.
$$
Moreover, it holds that $\int_0^{T^*}\D_3(U^{n+1}(t))dt\le CM_0$. Similar to the treatment of \eqref{u1lianxu}, the continuity of $\E_3(U^{n+1}(t))$ follows from
\bma
&\nnm|\E_3(U^{n+1}(t))-\E_3(U^{n+1}(s))|=\left|\int_s^t\Dtau\E_3(U^{n+1}(\tau)d\tau\right|
\\&\label{unlianxu}\le CM_0\int_s^t[\D_3(U^n(t))+\D_3(U^{n+1}(t))]d\tau+C(M_0+M_0^2)|t-s|\to0,~t\to s.
\ema
Combining the results above, we obtain $U^{n+1}(t)\in X_3(0,T^*;M_0)$ with $n\in\N$. This completes the induction.

Denote
$$
\tilde{U}^1=U^1,~~\tilde{U}^{n+1}=(\tilde{f}^{n+1},\tilde{\rho}^{n+1},\tilde{u}^{n+1}):=(f^{n+1}-f^n,\rho^{n+1}-\rho^n,u^{n+1}-u^n),~~n\ge1,
$$
and study the following system of $\tilde{U}^{n+1}~(n\ge1)$
\bmas
&\partial_t\tilde{f}^{n+1}+v\cdot\nabla_x\tilde{f}^{n+1}-L\tilde{f}^{n+1}
\\&\quad=(\tdx\Phi^{n-1}+u^{n-1})\cdot\(\frac12v\tilde{f}^n-\tdv \tilde{f}^n\)+(\tdx\tilde{\Phi}^n+\tilde{u}^n)\cdot\(\frac12vf^n-\tdv f^n+v\chi_0\),
\\
&\dt \tilde{\rho}^{n+1}+\divx \tilde{u}^{n+1}+u^n\cdot\tdx \tilde{\rho}^{n+1}+\tilde{u}^n\cdot\tdx\rho^n+\rho^n\divx \tilde{u}^{n+1}+\tilde{\rho}^n\divx u^n=0,
\\
&\dt \tilde{u}^{n+1}-\frac{1}{1+\rho^n}\ddx \tilde{u}^{n+1}
\\&\quad=-\tilde{u}^n\cdot\tdx u^n-u^{n-1}\cdot\tdx \tilde{u}^n-\tdx\tilde{\Phi}^n-\[\frac{p'(1+\rho^n)}{1+\rho^n}-\frac{p'(1+\rho^{n-1})}{1+\rho^{n-1}}\]\tdx\rho^n
\\&\quad\quad-\frac{p'(1+\rho^{n-1})}{1+\rho^{n-1}}\tdx\tilde{\rho}^n
+\(\frac{1}{1+\rho^n}-\frac{1}{1+\rho^{n-1}}\)(\ddx u^n+\tau_bf^n-u^n-\tau_af^nu^n)
\\&\quad\quad+\frac{1}{1+\rho^{n-1}}(\tau_b\tilde{f}^n-\tilde{u}^n-\tau_a\tilde{f}^nu^n-\tau_af^{n-1}\tilde{u}^n)
\\
&\ddx\tilde{\Phi}^{n+1}=\tau_a\tilde{f}^{n+1}-\tilde{\rho}^{n+1},
\emas
with the initial data
$$
\tilde{U}^1(0,x,v)=U_0,~~\tilde{U}^{n+1}(0,x,v)=0,~~n\ge1.
$$
Since $U^n$ only has $H^3_x$ regularity, we can only establish an $H^2_x$ estimate for $\tilde{U}^{n+1}$ due to the terms such as $\tilde{u}^n\cdot\nabla_x \rho^n$. With this, we now show that $U^n$ is a Cauchy sequence in $X_2(0,T^*;M_0)$. It is easy to verify that $\tilde{U}^1(t)\in X_2(0,T^*;M_0)$. Similar to the process of \eqref{unenergy}, we obtain
\bmas
&\frac12\Dt\(\| \tilde{f}^{n+1}\|_{L^2_v(H^2_x)}^2+\| (\tilde{\rho}^{n+1},\tilde{u}^{n+1},\tdx\tilde{\Phi}^{n+1})\|_{H^2_x}^2\)+\mu_0\(\| P_r\tilde{f}^{n+1}\|_{L^2_\sigma(H^2_x)}^2+\| \tdx \tilde{u}^{n+1}\|_{H^2_x}^2\)
\\&\le C\(1+\|(\rho^n,u^n)\|_{H^3_x}^2\)\(\|\tilde{f}^{n+1}\|_{L^2_v(H^2_x)}^2+\|(\tilde{\rho}^{n+1},\tilde{u}^{n+1},\tdx\tilde{\Phi}^{n+1})\|_{H^2_x}^2\)
\\&\quad+C\(1+\|(\rho^{n-1},u^{n-1},\tdx\Phi^{n-1})\|_{H^2_x}^2\)\|\tilde{f}^n\|_{L^2_v(H^2_x)}^2
\\&\quad+C\(1+\|\rho^{n-1}\|_{H^3_x}^2+\|f^n\|_{L^2_v(H^3_x)}^2\)\|\tilde{\rho}^n\|_{H^2_x}^2
\\&\quad+C\(1+\|(f^{n-1},f^n)\|_{L^2_v(H^3_x)}^2+\|\rho^{n-1}\|_{H^3_x}^2+\|f^n\|_{L^2_v(H^2_x)}^2\|\rho^n\|_{H^2_x}^2\)\|\tilde{u}^n\|_{H^2_x}^2
\\&\quad+C\(1+\|f^n\|_{L^2_v(H^2_x)}^2\)\|\tdx\tilde{\Phi}^n\|_{H^2_x}^2.
\emas
Similar to the treatment of \eqref{intunengry}, we have
$$
[1-CT^*(1+M_0)]\sup\limits_{0\le t\le T^*}\E_2(\tilde{U}^{n+1}(t))\le CT^*(1+M_0+M_0^2)\sup\limits_{0\le t\le T^*}\E_2(\tilde{U}^{n}(t)).
$$
By choosing $T^*>0$ sufficiently small, we obtain for $\tilde{\kappa}_1\in(0,1)$ that
\be\label{Cauchyine}
\sup\limits_{0\le t\le T^*}\E_2(\tilde{U}^{n+1}(t))\le \tilde{\kappa}_1\sup\limits_{0\le t\le T^*}\E_2(\tilde{U}^{n}(t)),
\ee
which implies that $U^n=(f^n,\rho^n,u^n) $ is a Cauchy sequence in the Banach space $X_2(0,T^*;M_0)$. Hence, there exists a limit function $U=(f,\rho,u)\in X_2(0,T^*;M_0)$ such that $U^n\to U$ strongly in $X_2(0,T^*;M_0)$ as $n\to\infty$, which solves the Cauchy problem \eqref{introe2}. Moreover, the uniform bound $\sup\limits_{0\le t\le T^*}\E_3(U^{n}(t))\le 2M_0$ admits a weakly convergent subsequence still denoted by $U^n=(f^n,\rho^n,u^n)$ such that $U^n\rightharpoonup U$ as $n\to\infty$. Combined with the weak lower semicontinuity of the Sobolev norm, we obtain
$$
\sup\limits_{0\le t\le T^*}\E_3(U(t))\le\liminf\limits_{n\to\infty} \sup\limits_{0\le t\le T^*}\E_3(U^{n}(t))\le 2M_0.
$$
The proof of the continuity of $\E_3(U(t))$ is similar to \eqref{unlianxu}. Thus, we have $U=(f,\rho,u)\in X_3(0,T^*;M_0)$. For the uniqueness, let $\overline{U}=(\overline{f},\overline{\rho},\overline{u})\in X_3(0,T^*;M_0)$ be another solution. Similar to the process of  \eqref{Cauchyine}, we have for $\tilde{\kappa}_2\in(0,1)$ that
$$
\sup\limits_{0\le t\le T^*}\E_2(U(t)-\overline{U}(t))\le \tilde{\kappa}_2\sup\limits_{0\le t\le T^*}\E_2(U(t)-\overline{U}(t)),
$$
which implies $U=\overline{U}$. The proof of theorem is completed.
\end{proof}

\medskip
\noindent {\bf Acknowledgements:}  The first and third authors are supported by  the special foundation for Guangxi Ba Gui Scholars, and the National Natural Science Foundation of China  (No.  12171104). The second author is supported by the National Natural Science Foundation of China (No. 12331007).




\begin{thebibliography}{99}
\setlength{\itemsep}{-4pt}
\renewcommand{\baselinestretch}{1}
\small
\bibitem{dlzt1}F. Aminmansoor, H. Abbasi: Hybrid (Vlasov-Fluid) simulation of ion-acoustic soliton chain formation and validity of Korteweg de-Vries model, Phys. Plasmas, 22, 082108(2015).
\bibitem{pphy3}C. Baranger, L. Boudin, P.E. Jabin, S.Mancini: A modeling of biospray for the upper airways, ESAIM Probab. Stat., 14(2005), 41-47.
\bibitem{pphy4}S. Berres, R. B\"urger, E.M. Tory: Mathematical model and numerical simulation of the liquid fluidization of polydisperse solid particle mixtures, Comput. Vis. Sci., 6(2004), 67-74.
\bibitem{lxliu0}J.A. Carrillo, T. Goudon, P. Lafitte: Simulation of Fluid and Particles Flows: Asymptotic Preserving Schemes for Bubbling and Flowing Regimes, J. Comput. Phys., 227(2008), 7929-7951.
\bibitem{lxliu5}J.A.Carrillo, R.J. Duan, A. Moussa: Global classical solution close to equillibrium to the Vlasov-Euler-Fokker-Planck system, Kinet. Relat. Models, 4(2011), 227-258.
\bibitem{lxliu6}L. Chen, F.C. Li, Y. Li, N. Zamponi: Global Weak Solutions to the Vlasov-Poisson-Fokker-Planck-Navier-Stokes System, Math. Methods Appl. Sci., 46(2023), 2729-2745.
\bibitem{lxliu4}R.J. Duan, S.Q. Liu: Cauchy problem on the Vlasov-Fokker-Planck equation coupled with the compressible Euler equations through the friction force, Kinet. Relat. Models, 6(2013), 687-700.
\bibitem{pphy5}G. Falkovich, A. Fouxon, M.G. Stepanov: Acceleration of rain initiation by cloud turbulence, Nature, 219(2002), 151-154.
\bibitem{dfb2}P. Griffiths, J. Harris: Principles of Algebraic Geometry, Hoboken, Wiley-Interscience, 1994.
\bibitem{gyvpb}Y. Guo: The Vlasov-Poisson-Boltzmann system near Maxwellians, Comm. Pure Appl. Math., 55(2002), 1104-1135.
\bibitem{gyvmb}Y. Guo: The Vlasov-Maxwell-Boltzmann system near Maxwellians, Invent. math., 153(2003), 593-630.
\bibitem{dlzt2}N. Javaheri, S. Rahimi, H. Abbasi: Hybrid (Kinetic-Fluid) Simulation Scheme Based on Method of Characteristics, (2015), arXiv: 1507.01178.
\bibitem{Kato}T. Kato: Perturbation Theory of Linear Operator, Springer, 1996.
\bibitem{lxliu3}F.C. Li, Y.M. Mu, D.H. Wang: Strong solutions to the compressible Navier-Stokes-Vlasov-Fokker-Planck equations: global existence near equilibrium and large time behavior, SIAM J. Math. Anal., 49(2017), 984-1026.
\bibitem{lxliu8}F.C. Li, J.K. Ni, D.H. Wang: Global well-posedness and inviscid limit of the compressible Navier-Stokes-Vlasov-Fokker-Planck system with density-dependent friction force, (2026),  arXiv:2603.07411.
\bibitem{nsp1}H.L. Li, A. Matsumura, G.J. Zhang: Optimal decay rate of the compressible Navier-Stokes-Poisson system in $\R^3$, Arch. Ration. Mech. Anal., 196(2010), 681-713.
\bibitem{vpfp1}H.L. Li, J.W. Sun, T. Yang, M.Y. Zhong: Large time behavior of solutions to Vlasov-Poisson-Landau (Fokker-Planck) equations (in Chinese), Sci. Sin. Math., 46(2016), 981-1004.
\bibitem{lxliu7}H.L. Li, T. Wang, Y. Wang: Wave phenomena to the three-dimensional fluid-particle model, Arch. Ration. Mech. Anal., 243(2022), 1019-1089.
\bibitem{zhonggfvpb}H.L. Li, T. Yang, M.Y. Zhong: Green's Function and Pointwise Space-time Behaviors of the Vlasov-Poisson-Boltzmann System, Arch. Rational. Mech. Anal., 235(2020), 1011-1057.
\bibitem{zhongspvpb}H.L. Li, T. Yang, M.Y. Zhong: Spectrum Analysis for the Vlasov-Poisson-Boltzmann System, Arch. Rational. Mech. Anal., 241(2021), 311-355.
\bibitem{zhongspbvpb} H.L. Li, T. Yang, M.Y. Zhong: Spectrum Analysis and Optimal Decay Rates ofthe Bipolar Vlasov-Poisson-Boltzmann Equations, Indiana Univ. Math. J., 65(2016), 665-725.
\bibitem{zhongspvmb}H.L. Li, T. Yang, M.Y. Zhong: Spectrum structure and behaviours of the Vlasov-Maxwell-Boltzmann systems, SIAM J. Math. Anal., 48(2016), 595-669.
\bibitem{livfpnsbd}H.L. Li, S.Q. Liu, T. Yang: The Navier-Stokes-Vlasov-Fokker-Planck System in Bounded Domains, J. Stat. Phys., 186(2022), 42.
\bibitem{glhsns}T.P. Liu, W.K. Wang: The Pointwise Estimates of Diffusion Wave for the Navier–Stokes Systems in Odd Multi-Dimensions, Comm. Math. Phys., 196(1998), 145–173.
\bibitem{glhsb}T.P. Liu, S.H. Yu: The Green’s function of Boltzmann equation, 3D waves. Bull. Inst. Math. Acad. Sin. (N. S.), 1(2006), 1–78.
\bibitem{yusplfp}L. Luo, H.J. Yu: Spectrum analysis of the linear Fokker-Planck equation, Anal. Appl., 15(2017), 313-331.
\bibitem{ns1}A. Matsumura, T. Nishida: The initial value problem for the equation of motion of viscous and heat-conductive gases. J. Math. Kyoto. Univ., 20(1980), 67-104.
\bibitem{lxliu1}A. Mellet, A. Vasseur: Global weak solutions for a Vlasov-Fokker-Planck/compressible Navier-Stokes system of equations, Math. Models Methods Appl. Sci., 17(2007)1039-1063.
\bibitem{lxliu2}A. Mellet, A. Vasseur: Asymptotic anslysis for a Vlasov-Fokker-Planck/Navier-Stokes system of equations, Comm. Math. Phys., 281(2008), 573-596.
\bibitem{Pazy}A. Pazy: Semigroups of Linear Operators and Applications to Partial Differential Equations, Springer, 1983.
\bibitem{dfb1}V. Scheidemann: Introduction to Complex Analysis in Several Variables, Springer, 2005.
\bibitem{ns2}Y.J. Wang, Z. Tan: Global existence and optimal decay rate for the strong solutions in $H^2$ to the compressible Navier–Stokes equations, Appl. Math. Lett., 24(2011), 1778-1784.
\bibitem{pphy2}F.A. Williams: Spray combustion and atomization, Phys. Fluid, 1(1958), pp.541-555.
\bibitem{zhaovpb}T. Yang, H.J. Yu, H.J. Zhao: Cauchy Problem for the Vlasov-Poisson-Boltzmann System. Arch. Rational, Mech. Anal., 182(2006), 415-470.
\bibitem{vpfp2} M.Y. Zhong: Green’s Function and the Pointwise Behaviors of the Vlasov-Poisson-Fokker-Planck System, Acta Math. Sci., 43(2023), 205-236.
\bibitem{vmfp1} M.Y. Zhong: Spectrum structure and behaviours of the Vlasov-Maxwell-Fokker-Planck system, (2026), preprint.
\end{thebibliography}
\end{document}